\documentclass[12pt]{amsart}

\usepackage{amscd,latexsym,amsthm,amsfonts,amssymb,amsmath,amsxtra}
\usepackage[all]{xy}

\usepackage[mathscr]{eucal}
\usepackage[pdftex,bookmarks,colorlinks,pdfmenubar]{hyperref}

\renewcommand\theequation{\thesection.\arabic{equation}}

\newcommand{\BC}{{\mathbb {C}}}

\newcommand{\BG}{{\mathbb {G}}}
\newcommand{\BH}{{\mathbb {H}}}

\newcommand{\BQ}{{\mathbb {Q}}}
\newcommand{\BR}{{\mathbb {R}}}

\newcommand{\BZ}{{\mathbb {Z}}}

\newcommand{\CC}{{\mathcal {C}}}
\renewcommand{\CD}{{\mathcal {D}}}
\newcommand{\CE}{{\mathcal {E}}}

\newcommand{\CI}{{\mathcal {I}}}
\newcommand{\CJ}{{\mathcal {J}}}

\newcommand{\CL}{{\mathcal {L}}}

\newcommand{\CO}{{\mathcal {O}}}
\newcommand{\CP}{{\mathcal {P}}}
\newcommand{\CQ}{{\mathcal {Q}}}

\newcommand{\CS}{{\mathcal {S}}}

\newcommand{\CW}{{\mathcal {W}}}
\newcommand{\CX}{{\mathcal {X}}}

\newcommand{\CZ}{{\mathcal {Z}}}

\newcommand{\FD}{{\mathfrak {D}}}

\newcommand{\Fe}{{\mathfrak {e}}}

\newcommand{\Fl}{{\mathfrak {l}}}
\newcommand{\Fm}{{\mathfrak {m}}}
\newcommand{\Fn}{{\mathfrak {n}}}
\newcommand{\Fo}{{\mathfrak {o}}}
\newcommand{\Fp}{{\mathfrak {p}}}

\newcommand{\RI}{{\mathrm {I}}}

\newcommand{\RO}{{\mathrm {O}}}

\newcommand{\Ad}{{\mathrm{Ad}}}

\newcommand{\bp}{{\mathrm{bp}}}

\newcommand{\disc}{{\mathrm{disc}}}

\newcommand{\End}{{\mathrm{End}}}

\newcommand{\GL}{{\mathrm{GL}}}

\newcommand{\gp}{{\mathrm{gp}}}
\newcommand{\gen}{{\mathrm{gen}}}

\newcommand{\Hom}{{\mathrm{Hom}}}
\newcommand{\Hss}{{\mathrm{Hss}}}

\renewcommand{\Im}{{\mathrm{Im}}}
\newcommand{\Ind}{{\mathrm{Ind}}}

\newcommand{\nsd}{{\mathrm{nsd}}}

\newcommand{\QD}{\mathrm{QD}}

\newcommand{\SL}{{\mathrm{SL}}}

\newcommand{\SO}{{\mathrm{SO}}}

\newcommand{\sgn}{{\mathrm{sgn}}}
\newcommand{\Sp}{{\mathrm{Sp}}}

\newcommand{\st}{{\mathrm{st}}}
\newcommand{\Span}{{\mathrm{Span}}}
\newcommand{\sg}{{\mathrm{sg}}}

\newcommand{\temp}{{\mathrm{temp}}}

\newcommand{\ud}{\,\mathrm{d}}

\newcommand{\udl}{\underline}

\newcommand{\wh}{\widehat}

\newcommand{\apair}[1]{\left\langle {#1} \right\rangle}

\newcommand{\cpair}[1]{\left\{{#1}\right\}}
\newcommand{\ppair}[1]{\left( {#1} \right)}

\def\bks{{\backslash}}

\def\eps{{\epsilon}}

\def\lam{{\lambda}}

\def\sig{{\sigma}}

\def\zet{{\zeta}}

\newtheorem{thm}{Theorem}[section]
\newtheorem{cor}[thm]{Corollary}
\newtheorem{lem}[thm]{Lemma}
\newtheorem{prop}[thm]{Proposition}
\newtheorem {conj}[thm]{Conjecture}
\newtheorem {ques/conj}[thm]{Question/Conjecture}

\newtheorem{defn}[thm]{Definition}
\newtheorem{rmk}[thm]{Remark}

\newtheorem{exmp}[thm]{Example}

\usepackage[numbers]{natbib}

\def\ve{{\varepsilon}}

\makeatletter

\newcommand{\Rmnum}[1]{\expandafter\@slowromancap\romannumeral #1@}
\makeatother

\begin{document}
\renewcommand{\theequation}{\arabic{equation}}
\numberwithin{equation}{section}

\title[Spectral Decomposition and Local Descents]{Local Root Numbers and Spectrum of the Local Descents for Orthogonal Groups: $p$-adic case}

\author{Dihua Jiang}
\address{School of Mathematics, University of Minnesota, Minneapolis, MN 55455, USA}
\email{dhjiang@math.umn.edu}

\author{Lei Zhang}
\address{Department of Mathematics,
National University of Singapore,
Singapore 119076}
\email{matzhlei@nus.edu.sg}

\subjclass[2010]{Primary 11F70, 22E50; Secondary 11S25, 20G25}
\keywords{Restriction and Local Descent, Generic Local Arthur Packet, Local Langlands Correspondence, Local Root Numbers, and Local Gan-Gross-Prasad Conjecture}

\date{\today}
\thanks{The research of the first named author is supported in part by the NSF Grants DMS--1301567 and DMS--1600685,
and that of the second named author is supported in part by the start-up grant and AcRF Tier 1 grant R-146-000-237-114 of National University of Singapore.}

\begin{abstract}
We investigate the local descents for special orthogonal groups over $p$-adic local fields of characteristic zero, and obtain explicit
spectral decomposition of the local descents at the first occurrence index in terms of the local Langlands data via the explicit
local Langlands correspondence and explicit calculations of relevant local root numbers. The main result can be regarded as a refinement of the local Gan-Gross-Prasad conjecture (\cite{GGP}).
\end{abstract}

\maketitle
\tableofcontents

\section{Introduction}

Let $G$ be a group and $H$ be a subgroup of $G$. For any representation $\pi$ of $G$, it is a classical problem that looks for the spectral decomposition
of the restriction of $\pi$ from $G$ to $H$. The spectral decomposition problem can also be reformulated in a different way.
For a given $\pi$ of $G$ and a subgroup $H$, which
representation $\sigma$ of $H$ has the property that
\begin{equation}\label{eq-m}
\Hom_H(\pi,\sigma)\neq 0?
\end{equation}
And what is the dimension of this $\Hom$-space?
When $\pi$ or $H$ is arbitrarily given, such a spectral decomposition is hard to be well understood, and those questions may not have reasonable
answers.

When $G$ is a Lie group or more generally a locally compact topological group defined by a reductive algebraic group, one may seek geometric conditions
on the pair $(G,H)$ such that the multiplicity $m(\pi,\sigma)$, which is the dimension of the $\Hom$-space in \eqref{eq-m}, is bounded or at most one.
In such a circumstance, one may seek invariants attached to $\pi$ and $\sigma$ that detect the multiplicity $m(\pi,\sigma)$. The local Gan-Gross-Prasad
conjecture for classical groups $G$ defined over a local field $F$ (\cite{GGP}) is one of the most successful examples concerning with those general questions.
When the local field $F$ is a finite extension of the $p$-adic number field $\BQ_p$ for some prime $p$, the local Gan-Gross-Prasad conjecture
for orthogonal groups has been completely resolved by the work of J.-L. Waldspurger and of C. M\oe glin and Waldspurger (\cite{W10}, \cite{W12a},
\cite{W12b}, and \cite{MW12}).

One of the basic notions in the local Gan-Gross-Prasad conjecture for orthogonal groups $G$ is the so called {\sl Bessel models}.
Over a $p$-adic local field $F$, Bessel models are defined in terms of a special family of twisted Jacquet functors. It is proved through the work of
\cite{AGRS}, \cite{SZ}, \cite{GGP}, and \cite{JSZ} that the Bessel models over any local fields of characteristic zero are of multiplicity free.
The local Gan-Gross-Prasad conjecture is to detect the multiplicity (which is either $1$ or $0$) in terms of the sign of the relevant local
$\varepsilon$-factors.

Meanwhile, the Bessel models have been widely used in the theory of the Rankin-Selberg method to
study families of automorphic $L$-functions and to define the corresponding local $L$-factors and local $\gamma$-factors.
In terms of representation theory and the local Langlands functoriality, the Bessel models produce the local descent method, which has been
successfully used in the explicit construction of the certain local Langlands functorial transfers for classical groups (\cite{JS-AM} and \cite{JS-AJM}).
In the spirit of the Bernstein-Zelevinsky derivatives for irreducible admissible representations of general
linear groups over $p$-adic local fields (\cite{BZ}), the Bessel models can be regarded as a tool to investigate basic properties of
irreducible admissible representations of $G(F)$ in general.

For instance, if $G$ is an odd special orthogonal group $\SO_{2n+1}$, then the local descents constructed via the family of Bessel models may produce
representations on the family of even special orthogonal groups $\SO_{2m}$, whose $F$-ranks should be controlled (up to $\pm1$)
by the $F$-rank $\delta$ of $\SO_{2n+1}$,
with $m=n-\delta, n-\delta+1,\cdots,n-1, n$. When $m=n$, it is the restriction from
$\SO_{2n+1}$ to $\SO_{2n}$, which is the case of the classical problem of {\sl symmetric breaking}. Hence the explicit spectral decomposition
when a representation $\pi$ of $\SO_{2n+1}(F)$ descends, via the twisted Jacquet functors of Bessel type, to $\SO_{2m}(F)$ is an interesting and
important problem, and may be considered as a refinement of the local Gan-Gross-Prasad conjecture.
One of the problems in our mind is to understand the spectral decomposition of the local descent for special orthogonal groups over $p$-adic local fields
in terms of the local Langlands parameters.

We explain our approach below with more details. The method is applicable to other classical groups. In some cases, we have to replace the Bessel models by the
Fourier-Jacobi models, following the formulation of the local Gan-Gross-Prasad conjecture in \cite{GGP}. The connection of the
results in this paper to automorphic forms is considered in the work of the authors \cite{JZ-BF}.

\subsection{Local descents}\label{ssec-LD}
Let $F$ be a non-archimedean local field of characteristic zero, which is a finite extension of $p$-adic number field $\BQ_p$ for some prime $p$.
As in \cite{JZ-BF} and \cite[Chapter 9]{A13}, we use $G_n^*=\SO(V^*,q^*_{V^*})$ to denote an $F$-quasisplit special orthogonal group that is defined by a non-degenerate, $\Fn$-dimensional quadratic space $(V^*,q^*)$ over $F$ with $n=[\frac{\Fn}{2}]$ and use $G_n=\SO(V,q)$ to denote a pure inner $F$-form of $G_n^*$. This means that both quadratic spaces $(V^*,q^*)$ and $(V,q)$ have the same dimension and the same discriminant,
as discussed in \cite{GGP}, for instance.

Let $\Pi(G_n)$ be the set of equivalence classes of irreducible smooth representations of $G_n(F)$. It is well-known that any $\pi\in\Pi(G_n)$ is
also admissible. Let $\mathfrak{r}=\dim X^\pm$ be the $F$-rank of $G_n$. Take $X^+$ to be an $\mathfrak{r}$-dimensional totally isotropic subspace of
$(V,q)$, and take $X^-$ to be the dual subspace of $X^+$. Then one has a polar decomposition of $(V,q)$:
$V=X^-\oplus V_0\oplus X^+$,
where $(V_0,q_{V_0})$ is the $F$-anisotropic kernel of $(V,q)$. With a suitable choice of the order of the dual bases in $X^-$ and $X^+$, one must
have a minimal parabolic subgroup $P_0$ of $G_n$, whose unipotent radical can be realized in the upper triangular matrix form.
For any standard $F$-parabolic subgroup $P=MN$, containing $P_0$, of $G_n$, take a character
$\psi_N$ of the unipotent radical $N(F)$ of $P(F)$, which is defined through a non-trivial additive character $\psi_F$ of $F$.
One may define the {\sl twisted Jacquet module} for any $\pi\in\Pi(G_n)$ with respect to $(N,\psi_N)$ to be
the quotient
$$
\CJ_{N,\psi_N}(V_\pi):=V/V(N,\psi_N),
$$
where $V(N,\psi_N)$ is the span of the subset
$$
\{\pi(n)v-\psi_N(n)v\ |\ \forall v\in V_\pi, \forall n\in N(F)\}.
$$
Let $M_{\psi_N}$ is the stabilizer of $\psi_N$ in $M$. Then the twisted Jacquet module $\CJ_{N,\psi_N}(V_\pi)$ is a smooth representation of
$M_{\psi_N}(F)$. In such a generality, one may not have much information about the twisted Jacquet module $\CJ_{N,\psi_N}(V_\pi)$, as
a representation of $M_{\psi_N}(F)$. Following the inspiration of the
Bernstein-Zelevinsky theory of derivatives for representations of $p$-adic $\GL_n$ (\cite{BZ}),
the theory of the local descents is to obtain more explicit information about the twisted Jacquet module $\CJ_{N,\psi_N}(V_\pi)$
in terms of the given $\pi$ and its local Langlands parameter, for a family of specially chosen data $(N,\psi_N)$.

To introduce the twisted Jacquet modules of Bessel type,
we take a family of partitions of the form:
\begin{equation}\label{udlp}
\udl{p}_\ell:=[(2\ell+1)1^{{\Fn-2\ell-1}}],
\end{equation}
with $0\leq\ell\leq \mathfrak{r}$. Those partitions $\udl{p}_\ell$ are {\sl $G_n$-relevant} in the sense that they correspond to $F$-rational unipotent
orbits of $G_n(F)$. As in \cite{JZ-BF},
the $F$-stable nilpotent orbit $\CO_{\udl{p}_\ell}^\st$ corresponding to the partition $\udl{p}_\ell$ defines a unipotent subgroup
$V_{\udl{p}_\ell}$ of $G_n$ over $F$, and each $F$-rational orbit $\CO_\ell$ in the $F$-stable orbit $\CO_{\udl{p}_\ell}^\st$ defines
a {\sl generic} character $\psi_{\CO_\ell}$ of $V_{\udl{p}_\ell}(F)$.

More precisely, let $\{e_{\pm 1},e_{\pm2},\dots,e_{\pm\mathfrak{r}}\}$ be a basis of $X^{\pm}$, respectively  such that $q(e_i,e_{-j})=\delta_{i,j}$ for all $1\leq i,j\leq \mathfrak{r}$.
Then we may choose the minimal parabolic subgroup $P_0$ to fix the following totally isotropic flag in $(V,q)$:
\begin{equation}\label{eq:V-ell}
V^{+}_1\subset V^{+}_2\subset\cdots \subset V^{+}_{\mathfrak{r}}\quad \text{ where } V^{\pm}_i=\Span\{e_{\pm 1},\dots,e_{\pm i}\}. 	
\end{equation}
For the partition $\underline{p}_{\ell}$ in \eqref{udlp}, we consider the standard parabolic subgroup $P_{1^\ell}=M_{1^\ell}N_{1^\ell}$,
containing $P_0$, with the Levi subgroup
$M_{1^\ell}\cong\GL_1^{\times\ell}\times G_{n-\ell}$
and $V_{\udl{p}_\ell}=N_{1^\ell}$.
Here  $V_{\udl{p}_\ell}$ consists of elements of form:
\begin{equation}\label{eq:V-p-ell}
V_{\udl{p}_\ell}=\cpair{v=\begin{pmatrix}
z&y&x\\&I_{\Fn-2\ell}&y'\\&&z^*	
\end{pmatrix}\in G_n\colon z\in Z_{\ell} },	
\end{equation}
where $Z_\ell$ is the standard maximal (upper-triangular) unipotent subgroup of $\GL_\ell$.
Then the $F$-rational nilpotent orbits $\CO_\ell$ in the $F$-stable nilpotent orbit $\CO^{\rm st}_{\udl{p}_\ell}$ correspond to the
$\GL_1(F)\times G_{n-\ell}(F)$-orbits of $F$-anisotropic vectors in $(F^{\Fn-2\ell},q)$.
The generic character $\psi_{\CO_\ell}$ of $V_{\udl{p}_\ell}(F)$ may be explicitly defined as follows.
Fix a non-trivial additive character $\psi_F$ of $F$.
For an anisotropic vector $w_0$ in $((V_\ell^+\oplus V_\ell^-)^\perp,q)$ associated to the $F$-rational orbit $\CO_\ell$ in $\CO^{\rm st}_{\udl{p}_\ell}$,
define a character $\psi_{\ell, w_0}$ of $V_{\udl{p}_\ell}(F)$ by
\begin{equation}\label{eq:psi-ell}
\psi_{\CO_\ell}(v)=\psi_{\ell,w_0}:=\psi(\sum_{i=1}^{\ell-1}z_{i,i+1}+q(y_\ell,w_0))
\end{equation}
where $z_{i,j}$ is the entry of the matrix $z$ in the $i$-th row and $j$-th column
and  $y_\ell$ is the last row of the matrix $y$ in \eqref{eq:V-p-ell}.

The Levi subgroup $M_{1^\ell}$ acts on the set of those generic characters $\psi_{\ell,w_0}$ via the adjoint action on $V_{\udl{p}_\ell}$.
We denote by
$G_{n}^{\CO_\ell}$ the  identity component  of the stabilizer in $M_{1^\ell}$ of the character $\psi_{\CO_\ell}$, viewed as a subgroup of $G_{n-\ell}$.
By Proposition 2.5 of \cite{JZ-BF}, the algebraic group $G_{n}^{\CO_\ell}$ is a special orthogonal group defined over $F$ by a non-degenerate
quadratic subspace $(W_\ell,q)$ of $(V,q$) with dimension $\Fn-2\ell-1$.
Here if $\psi_{\CO_\ell}$ is of form $\psi_{\ell,w_0}$,
\begin{equation}\label{eq:w}
W_\ell=(V^+_\ell\oplus\Span\{w_0\}\oplus V^-_\ell)^\perp
\end{equation}
and $G_{n}^{\CO_\ell}$ can be identified as the special orthogonal group $\SO(W_\ell,q)$. We refer to \cite[Proposition 2.5]{JZ-BF} for more structures of $G_{n}^{\CO_\ell}$.

We define the following subgroup of $G_n$, which is a semi-direct product,
\begin{equation}\label{stab}
R_{\CO_\ell}:=G_{n}^{\CO_\ell}\ltimes V_{\udl{p}_\ell}.
\end{equation}
For any $\pi\in\Pi(G_n)$, the twisted Jacquet module with respect to the pair
$(V_{\udl{p}_\ell}, \psi_{\CO_\ell})$ is defined by
\begin{equation}\label{tjm}
\CJ_{\CO_\ell}(\pi)=\CJ_{\CO_\ell}(V_\pi)=\CJ_{V_{\udl{p}_\ell}, \psi_{\CO_\ell}}(V_\pi):=V_\pi/V_\pi(V_{\udl{p}_\ell}, \psi_{\CO_\ell}),
\end{equation}
which may also be called the {\sl twisted Jacquet module of Bessel type} of $\pi$.

For an irreducible admissible representation $\sigma$ of $G_{n}^{\CO_\ell}(F)$, the linear functionals that
belong to the following $\Hom$-space
\begin{equation}\label{lbm}
\Hom_{R_{\CO_\ell}(F)}(\pi\otimes\sigma^\vee,\psi_{\CO_\ell}),
\end{equation}
where $\sigma^\vee$ is the admissible dual of $\sigma$, are called the {\sl local Bessel functionals} for the pair $(\pi,\sigma)$. The uniqueness of local Bessel functionals asserts that
\begin{equation}\label{lunq}
\dim \Hom_{R_{\CO_\ell}(F)}(\pi\otimes\sigma^\vee,\psi_{\CO_\ell})\leq 1.
\end{equation}
This was proved in \cite{AGRS}, \cite{SZ}, \cite{GGP}, and \cite{JSZ}. It is clear that
$$
\Hom_{R_{\CO_\ell}(F)}(\pi\otimes\sigma^\vee,\psi_{\CO_\ell})
\cong
\Hom_{G_{n}^{\CO_\ell}(F)}(\CJ_{\CO_\ell}(\pi),\sigma).
$$
It is a natural problem to understand possible irreducible quotients $\sigma$ of the module $\CJ_{\CO_\ell}(\pi)$ of $G_{n}^{\CO_\ell}(F)$.
The local Gan-Gross-Prasad conjecture is to determine such a quotient $\sigma$ by means of the sign of the epsilon factor for the pair $(\pi,\sigma)$.
One of the main results of this paper is to determine all possible quotients $\sigma$'s for a given $\pi$ with explicit description of their local Langlands
parameters (Theorem \ref{th:sdlp}), when the index $\ell$ is the first occurrence index (given in Definition \ref{df:foi}).

In order to understand all possible irreducible quotients, we introduce a notion of the $\ell$-th {\sl maximal quotient} of $\pi$ if $\CJ_{\CO_\ell}(\pi)$ is nonzero. Define
\begin{equation}\label{msb}
\CS_{\CO_\ell}(\pi):=\cap_{\Fl_\sigma}\ker(\Fl_\sig),
\end{equation}
where the intersection is taken over all local Bessel functionals $\Fl_\sig$ in $\Hom_{G_{n}^{\CO_\ell}(F)}(\CJ_{\CO_{\ell}}(\pi),\sig)$ for all  $\sig\in\Pi(G^{\CO_{\ell}}_n)$. The $\ell$-th maximal quotient of $\pi$ is defined to be
\begin{equation}\label{ell-mq}
\CQ_{\CO_{\ell}}(\pi):=\CJ_{\CO_\ell}(\pi)/\CS_{\CO_\ell}(\pi).
\end{equation}
Of course, if $\CJ_{\CO_\ell}(\pi)$ is zero, we define $\CQ_{\CO_{\ell}}(\pi)$ to be zero.
In this paper, we study this quotient for irreducible admissible representations $\pi$ of $G_n(F)$ with generic local
$L$-parameters, the definition of which will be explicitly given in Section \ref{ssec-LVP}.
Since we mainly discuss the local situation, we may call  the local $L$-parameters {\sl the $L$-parameters}   for simplicity.
\begin{prop}\label{pp:iq}
For any $\pi\in\Pi(G_n)$ with a generic $L$-parameter, if the twisted Jacquet module $\CJ_{\CO_\ell}(\pi)$ is non-zero, then there exists a $\sigma\in\Pi(G^{\CO_{\ell}}_n)$
such that
$$
\Hom_{G^{\CO_{\ell}}_n(F)}(\CJ_{\CO_\ell}(\pi),\sigma)\ne 0.
$$
Namely, the twisted Jacquet module $\CJ_{\CO_\ell}(\pi)$ of $\pi$ is non-zero if and only if the $\ell$-th maximal quotient $\CQ_{\CO_{\ell}}(\pi)$
of $\pi$ is non-zero.
\end{prop}

Proposition \ref{pp:iq} follows from Lemma \ref{lm:giq} in Section \ref{sec-ITQ}.
By Proposition \ref{pp:iq}, if $\CJ_{\CO_\ell}(\pi)$
is non-zero, then the $\ell$-th maximal quotient $\CQ_{\CO_{\ell}}(\pi)$ is non-zero. In this situation, we set
\begin{equation}\label{lld}
\CD_{\CO_{\ell}}(\pi):=\CQ_{\CO_{\ell}}(\pi)
\end{equation}
and call $\CD_{\CO_{\ell}}(\pi)$ the $\ell$-th {\sl local descent} of $\pi$ (with respect to the given $F$-rational orbit $\CO_\ell$).
Note that the group $G_{n}^{\CO_\ell}(F)$ and the representation $\CD_{\CO_{\ell}}(\pi)$ depend on the $F$-rational structure of orbit $\CO_\ell$.

The {\sl theory of the local descents} is to understand the structure of the $\ell$-th local descent $\CD_{\CO_{\ell}}(\pi)$ as a representation of $G_{n}^{\CO_\ell}(F)$,
in particular, in the situation when $\ell$ is the first occurrence index. In order to define the notion of the {\sl first occurrence},
we prove the following {\sl stability} of the local descents.

\begin{prop}\label{pp:sl}
For any $\pi\in\Pi(G_n)$ with a generic $L$-parameter,
if there exists an $\ell_1$ such that the $\ell_1$-th local descent $\CD_{\CO_{\ell_1}}(\pi)$ is non-zero for some $F$-rational orbit $\CO_{\ell_1}$,
then the $\ell$-th local descent $\CD_{\CO_{\ell}}(\pi)$ is non-zero
for every $\ell\leq \ell_1$ with a certain compatible $\CO_\ell$.
\end{prop}
We note that the details of the compatibility of $\CO_{\ell_1}$ and $\CO_{\ell}$ will be given in Proposition \ref{pp:sld}.
Proposition \ref{pp:sl} follows essentially from the relation between multiplicity and parabolic induction as discussed in Section
\ref{ssec-MPI} and will be included in Proposition \ref{pp:sld} on the stability of the local descents in Section \ref{sec-ITQ}.

\begin{defn}[First Occurrence Index]\label{df:foi}
For $\pi\in\Pi(G_n)$, the first occurrence index $\ell_0=\ell_0(\pi)$ of $\pi$ is the integer $\ell_0$ in $\{0,1,\dots,\mathfrak{r}\}$
where $\mathfrak{r}$ is the $F$-rank of $G_n$,
such that the twisted Jacquet module $\CJ_{\CO_{\ell_0}}(\pi)$ is nonzero for some $F$-rational orbit $\CO_{\ell_0}$,
but for any $\ell\in \{0,1,\dots,\mathfrak{r}\}$ with $\ell>\ell_0$, the twisted Jacquet module $\CJ_{\CO_{\ell}}(\pi)$ is zero
for every $F$-rational orbit $\CO_{\ell}$ associated to the partition $\udl{p}_\ell$ as defined in \eqref{udlp}.
\end{defn}

It is clear that the definition of the first occurrence index is also applicable to the representations $\pi$ that may not be irreducible.
At the first occurrence index, we define the notion of the local descent of $\pi$.

\begin{defn}[Local Descent]\label{df:ld}
For any given $\pi\in\Pi(G_n)$, assume that $\ell_0$ is the first occurrence index of $\pi$.  The $\ell_0$-th local descent $\CD_{\CO_{\ell_0}}(\pi)$
of $\pi$ is called the first local descent of $\pi$ (with respect to $\CO_{\ell_0}$) or simply the local descent of $\pi$, which is the $\ell_0$-th nonzero maximal quotient
$$
\CD_{\CO_{\ell_0}}(\pi):=\CQ_{\CO_{\ell_0}}(\pi)
$$
for the $F$-rational orbit $\CO_{\ell_0}$ associated to the partition $\udl{p}_{\ell_0}$ as defined in \eqref{udlp}.
\end{defn}
It is clear that such an $\CO_{\ell_0}$ always exists by the definition of $\ell_0$, but may not be unique.
Also, when $G_n=G_n^*$ is $F$-quasisplit and $\pi$ is generic, i.e. has a nonzero Whittaker model, the first occurrence index
is clearly $\ell_0=n=[\frac{\Fn}{2}]$, where $\Fn=\dim V^*$. The discussion related to the first occurrence index
in this paper will exclude this trivial case.

\subsection{Main results}\label{ssec-MT}
The main results of the paper are about the spectral properties of the $\ell$-th local descents of the irreducible smooth
representations of $G_n(F)$ with generic $L$-parameters. At the first occurrence index, the spectral properties of the
local descents are explicit.

\begin{thm}[Square Integrability]\label{th:ldds}
Assume that $\pi\in\Pi(G_n)$ has a generic $L$-parameter.
Then, at the first occurrence index $\ell_0=\ell_0(\pi)$ of $\pi$,
the local descent $\CD_{\CO_{\ell_0}}(\pi)$ is square-integrable and admissible.
Moreover, the local descent $\CD_{\CO_{\ell_0}}(\pi)$ is a multiplicity-free direct sum of
irreducible square-integrable representations, with all its irreducible summands belonging to different Bernstein components.
\end{thm}

The proof of Theorem \ref{th:ldds} depends on the following weaker result about the $\ell$-th local descent for general $\ell$. Because it can be deduced
from the work on the local Gan-Gross-Prasad conjecture for orthogonal groups (\cite{GGP}) of Waldspurger (\cite{W10}, \cite{W12a} and \cite{W12b})
for tempered $L$-parameters, and of M\oe glin and Waldspurger (\cite{MW12}) with generic $L$-parameters, we state it as a corollary and refer to
Section \ref{sec-ITQ} for the details.

\begin{prop}[Irreducible Tempered Quotient]\label{th:itq}
For any $\pi\in\Pi(G_n)$ with a generic $L$-parameter, and
for any $\ell\in\{0,1,\cdots,\mathfrak{r}\}$ with $\mathfrak{r}$ being the
$F$-rank of $G_n$, if the twisted Jacquet module $\CJ_{\CO_\ell}(\pi)$ is nonzero, then $\CJ_{\CO_\ell}(\pi)$ has a tempered irreducible quotient
as a representation of $G_n^{\CO_\ell}(F)$, and so does the $\ell$-th local descent $\CD_{\CO_\ell}(\pi)$.
In other words, if $\CJ_{\CO_\ell}(\pi)\neq 0$, then there exists a $\sigma\in\Pi_\temp(G_n^{\CO_{\ell}})$ such that
$$
\Hom_{G_{n}^{\CO_\ell}(F)}(\CJ_{\CO_\ell}(\pi),\sigma)=\Hom_{G_{n}^{\CO_\ell}(F)}(\CD_{\CO_\ell}(\pi),\sigma)\neq 0.
$$
\end{prop}

It is worthwhile to mention an analogy of Proposition \ref{th:itq} to the {\sl Generic Summand Conjecture}
(Conjecture 2.3 in \cite{JZ-BF} and Conjecture 2.4 in \cite{JZ-AIT2}).
In such a generality, if the representation $\pi$ in Proposition \ref{th:itq} is unramified, then for any index $\ell$, the $\ell$-th
local descent $\CD_{\CO_\ell}(\pi)$ has a tempered, unramified irreducible quotient.
While the Generic Summand Conjecture (\cite{JZ-BF} and \cite{JZ-AIT2}) asserts that for any irreducible cuspidal automorphic representation of
$G_n$ with a generic global Arthur
parameter, its global descent at the first occurrence index is cuspidal and contains at least one irreducible summand that has a generic
global Arthur parameter. This assertion serves a base for the construction of explicit modules for irreducible cuspidal automorphic
representations of general classical groups with generic global Arthur parameters as developed in
\cite{JZ-BF} and \cite{JZ-AIT2}. We will discuss the global impact of the results obtained in this paper to the Generic Summand Conjecture in \cite{JZ-BF} and
\cite{JZ-AIT2} in our future work.

From the setup of the local descents, the theory is closely related to the local Gan-Gross-Prasad conjecture (\cite{GGP}).
By applying the theorems of M\oe glin and Waldspurger on the local Gan-Gross-Prasad conjecture for generic $L$-parameters, and
by explicit calculations of local $L$-parameters and the relevant local root numbers via the local Langlands correspondence in the
situation considered here, we are able to obtain the following explicit spectral decomposition for the local descent at the first occurrence
index. In order to state the result, we briefly explain the notation used in Theorem \ref{th:sdlp} below, and leave the details
in Sections \ref{sec-LGGP} and \ref{sec-MTHM}.

Write $\CZ:=F^\times/F^{\times 2}$. After fixing a rationality of the local Langlands correspondence $\iota_a$ as in Section \ref{ssec-rllc},
when a $\pi\in\Pi(G_n)$ is
determined by the parameter $(\varphi,\chi)$, the abelian group $\CZ$ acts on the dual $\wh{\CS}_\varphi$ of $\CS_\varphi$. Denote by
$\CO_\CZ(\pi)$ the $\CZ$-orbit in $\wh{\CS}_\varphi$ determined by $\pi$ (see \eqref{eq:CZ}). Denote by $[\varphi]_c$ the pair of the local
$L$-parameters which are conjugate to each other via an element $c$ in the complex group $\RO(V)(\BC)$ with $\det(c)=-1$.
In Definition \ref{df:foi-L}, we introduce
the notion of the descent $\FD_\ell(\varphi,\chi)$
and that of the first occurrence index $\ell_0=\ell_0(\varphi,\chi)$ for the local parameters
$(\varphi,\chi)$. The basic structure of the descent $\FD_{\ell_0}(\varphi,\chi)$ at the first occurrence index is given in Theorem \ref{pp:dsqlp}.

\begin{thm}[Spectral Decomposition]\label{th:sdlp}
Assume that $\pi\in\Pi(G_n(V_\Fn))$ is associated to an equivalence class $[\varphi]_c$ of generic $L$-parameters.
\begin{enumerate}
\item \label{thm:sdlp:part-1} The first occurrence index of $\pi$ is determined by the first occurrence index of the local parameters via the formula $$\ell_0(\pi)=\max_{\chi\in\CO_\CZ(\pi)}\{\ell_0([\varphi]_c,\chi)\}.
$$
\item \label{thm:sdlp:part-2}
For each $F$-rational orbit $\CO_{\ell_0}$,
the local descent of $\pi$ at the first occurrence index $\ell_0=\ell_0(\pi)$ is a multiplicity-free, direct sum of irreducible, square-integrable
representations of $G^{\CO_{\ell_0}}_n(F)$, which can be explicitly given below.
\begin{enumerate}
\item\label{item:main-thm-even} When $\Fn=2n$ is even and for $\chi \in\CO_\CZ(\pi)$
$$
\CD_{\CO_{\ell_0}}(\pi)=\bigoplus_{\phi\in\FD_{\ell_0}(\varphi,\chi)}\pi(\phi,\chi^\star_{\phi}(\varphi,\phi)),
$$
where the local Langlands correspondence $\iota_a$ for $\pi$ is given by $a=\disc(\CO_{\ell_0})$ and $\chi_a(\pi)=\chi$;
and the quadratic space $W$ defining $G^{\CO_{\ell_0}}_n$ is given by $\disc(W)=-\disc(\CO_{\ell_0})\cdot\disc(V_\Fn)$ and Equation \eqref{eq:thm:HW}.
\item \label{item:main-thm-odd} When $\Fn=2n+1$ is odd,
$$
\CD_{\CO_{\ell_0}}(\pi)=\bigoplus_{\substack{\phi\in\FD_{\ell_0}(\varphi,\chi)\\
\det(\phi)=\disc(\CO_{\ell_0})\cdot \disc(V_\Fn)}}\pi_{-\disc(\CO_{\ell_0})}(\phi,\chi^\star_{\phi}(\varphi,\phi)),
$$
where  the quadratic space defining $G^{\CO_{\ell_0}}_n$ is given by Equations \eqref{eq:thm:DV} and \eqref{eq:thm:HV}.
\end{enumerate}
\end{enumerate}
\end{thm}

Theorems \ref{th:ldds} and \ref{th:sdlp} will be proved in Section \ref{sec-MTHM}, not only using the result of Proposition \ref{th:itq}, but also
using the proof of Proposition \ref{th:itq} in Section \ref{sec-ITQ}, with refinement. Moreover, in order to keep tracking the behavior of the
local $L$-parameters and the sign of the local root numbers with the local descents, we need explicit information about
the descent of local $L$-parameters $\FD_{\ell_0}(\varphi,\chi)$,  given in
Theorem \ref{pp:dsqlp}.

We note that it is possible that $\CD_{\CO_{\ell_0}}(\pi)=0$ for some $F$-rational orbit $\CO_{\ell_0}$.
But there exists at least one $F$-rational orbit $\CO_{\ell_0}$ such that $\CD_{\CO_{\ell_0}}(\pi)\ne 0$. Also an explicit formula for the
character $\chi^\star_{\phi}(\varphi,\phi)$ can be found in Corollary \ref{character}.

Some more comments on Theorem \ref{th:sdlp} are in order.

First of all, in terms of the local Gan-Gross-Prasad conjecture (\cite{GGP}), the spectral decomposition as given in Theorem \ref{th:sdlp} can
be interpreted as follows. For any $\pi\in\Pi(G_n)$ with a generic $L$-parameter, at the
first occurrence index, the spectral decomposition explicitly determines in terms of the local Langlands data of all possible irreducible
representations $\sigma$ of $G_n^{\CO_{\ell_0}}(F)$ that form the distinguished pair with the given $\pi$ as required by the local
Gan-Gross-Prasad conjecture. Meanwhile, this spectral decomposition indicates that for such a given $\pi$, if a $\sigma\in\Pi(G_n^{\CO_{\ell_0}})$
can be paired with $\pi$ as in the local Gan-Gross-Prasad conjecture, then $\sigma$ must be square integrable. Hence Theorem \ref{th:sdlp}
and Corollary \ref{character}
can be regarded as a refinement of the local Gan-Gross-Prasad conjecture.

Secondly,
it is interesting to compare briefly Theorem \ref{th:sdlp} with the local descent of the first named author and Soudry (\cite{JS-AM} and \cite{JS-AJM}, see also \cite{JNQ-PJM10}).
For instance, one takes $G_n^*$ to be the $F$-split $\SO_{2n+1}$. Let $\tau$ be an irreducible supercuspidal representation of $\GL_{2n}(F)$, which is
of symplectic type, i.e. the local exterior square $L$-function of $\tau$ has a pole at $s=0$. The local descent in \cite{JS-AM} and \cite{JNQ-PJM10}
is to take $\pi=\pi(\tau,2)$ to be the unique Langlands quotient of the induced representation of the $F$-split $\SO_{4n}(F)$ with supercuspidal
support $(\GL_{2n},\tau)$.
According to the endoscopic classification of Arthur (\cite{A13}), $\pi=\pi(\tau,2)$ has a non-generic (non-tempered) local Arthur parameter
$(\tau,1,2)$, and has the local $L$-parameter $\phi_\tau|\cdot|^{\frac{1}{2}}\oplus\phi_\tau|\cdot|^{-\frac{1}{2}}$, where $\phi_\tau$ is the
local $L$-parameter of $\tau$.
Now the local descent in \cite{JS-AM} and \cite{JNQ-PJM10} shows that $\CD_{n-1}(\pi)$ with $\ell_0=n-1$ being the first occurrence index is an irreducible
generic supercuspidal representation of $\SO_{2n+1}(F)$ with the generic local Arthur parameter $(\tau,1,1)$ or the local $L$-parameter $\phi_\tau$.
In this case, $\tau$ is the image of $\CD_{n-1}(\pi)$ under the local Langlands functorial transfer from $\SO_{2n+1}$ to $\GL_{2n}$.
The result is even simpler than
Theorem \ref{th:sdlp} as expected. However, from the point of view of the
local Gan-Gross-Prasad conjecture, the result in \cite{JS-AM} and \cite{JNQ-PJM10} can be viewed as a case of the local Gan-Gross-Prasad conjecture for non-tempered
local Arthur parameters. Hence the work to extend Theorem \ref{th:sdlp} to the representations with general local Arthur parameters is closely
related to the local Gan-Gross-Prasad conjecture for non-tempered local Arthur parameters. This is definitely a very interesting topic, but we
will not discuss it with any more details in this paper.

Finally, we would like to elaborate an application of Theorem \ref{th:sdlp}. For a generic local $L$-parameter $\phi$ of
an $F$-quasisplit $G_n^*$, the local $L$-packet $\Pi_\phi(G_n^*)$ as defined in \cite{MW12} contains a generic member, i.e. a member with a non-zero
Whittaker model. From the relation between unipotent orbits of $G_n^*(F)$ and the twisted Jacquet modules for $G_n^*(F)$. The Whittaker model corresponds
to the twisted Jacquet module associated to the regular unipotent orbit of $G_n^*$. In general, other members in the local $L$-packet $\Pi_\phi(G_n^*)$
may not have a non-zero twisted Jacquet module associated to the regular unipotent orbit.
Hence it is desirable to know that for any member $\pi$ in the local $L$-packet $\Pi_\phi(G_n^*)$, what kind twisted Jacquet modules does $\pi$ have?

Let $\CP(G_n^*)$ be the set of orthogonal partitions $\udl{p}=[p_1p_2\cdots p_r]$ associated to the $F$-stable unipotent orbits of $G_n^*(F)$. As defined in
\cite[Section 4]{J14} and similar to the definition of the twisted Jacquet modules of Bessel type, we may construct a twisted Jacquet module associated to
any $F$-rational unipotent orbit $\CO_{\udl{p}}$ in the corresponding $F$-stable orbit $\CO_{\udl{p}}^\st$.
More precisely, we may construct, for any $\CO_{\udl{p}}$ in $\CO_{\udl{p}}^\st$, an unipotent subgroup $V_{\CO_{\udl{p}}}$ of $G_n^*$ and a character
$\psi_{\CO_{\udl{p}}}$, and define the twisted Jacquet module $\CJ_{\CO_{\udl{p}}}(\pi)$ for any irreducible smooth representation $\pi$ of $G_n^*(F)$.
Now, we define
$\Fp(\pi)$ to be the subset of $\CP(G_n^*)$, consisting of partitions $\udl{p}$ with the property that there exists an $F$-rational $\CO_{\udl{p}}$ in
the $F$-stable $\CO_{\udl{p}}^\st$ such that the twisted Jacquet module $\CJ_{\CO_{\udl{p}}}(\pi)$ is nonzero. Let $\Fp^m(\pi)$ be the subset of
$\Fp(\pi)$ consisting of all maximal members in $\Fp(\pi)$. Following \cite{Kaw} and \cite{M96}, one may take $\Fp^m(\pi)$ to be the algebraic version of the
{\sl wave-front set} of $\pi$. It is generally believed (Conjecture 3.1 of \cite{JL-Cogdell})
that if $\pi$ is tempered, then the set $\Fp^m(\pi)$ contains only one partition. One expects that this property
holds for general $\pi$.
Assume that $\Fp^m(\pi)=\{\udl{p}=[p_1p_2\cdots p_r]\}$ with $p_1\geq p_2\geq\cdots\geq p_r>0$. In order to determine the algebraic wave-front set
$\Fp^m(\pi)$, it is an important step to understand how the largest part $p_1$ in the partition $\udl{p}\in\Fp^m(\pi)$ is determined by the local Langlands
data associated to $\pi$, via the local Langlands correspondence for $G_n^*(F)$. Here is our conjecture.
\begin{conj}\label{conj-p1}
Assume that $\pi\in\Pi_\phi(G_n^*)$ has a generic local $L$-parameter $\phi$.
Then the largest part $p_1$ in the partition $\udl{p}=[p_1p_2\cdots p_r]$ that belongs to $\Fp^m(\pi)$ is equal to $2\ell_0+1$,
where $\ell_0=\ell_0(\pi)$ is the first occurrence index in the local descents.
\end{conj}

Assume that Conjecture \ref{conj-p1} holds. Then Part (1) of Theorem \ref{th:sdlp} asserts that the largest part $p_1$ of the partition $\udl{p}$
in the algebraic wave-front set $\Fp^m(\pi)$ of $\pi$ is completely determined by the
property of the local Langlands data associated to $\pi$, if $\pi$ has a generic local $L$-parameter.
From this point of view, Theorem \ref{th:sdlp} serves
as a base of an induction argument to determine the remaining parts $p_2,\cdots,p_r$. In Section \ref{sec:ex}, we will discuss Conjecture \ref{conj-p1}
via some examples. However, we expect that the induction argument is long and complicate,
and hence leave it to our future work.

\subsection{The structure of the proofs and the paper}
The local Gan-Gross-Prasad conjecture (\cite{GGP}) as proved in \cite{W10}, \cite{W12a}, \cite{W12b} and \cite{MW12} is the starting point
and the technical backbone of the discussion of this paper. We recall it in Section \ref{sec-LGGP}.  In order to understand the $F$-rationality
of the local $L$-parameters and the $F$-rationality of the local descents, we reformulate the local Langlands correspondence
and the local Gan-Gross-Prasad conjecture in terms of the basic rationality data given by the underlying quadratic
forms and quadratic spaces. This is discussed in Section \ref{ssec-rllc}. It is clear that one might formulate such a rationality in terms of
the notion of rigid inner forms as discussed by T. Kaletha in \cite{K16}. Due to the nature of the current paper, the authors thought that
it is more convenient and more direct to use the formulation in Section \ref{ssec-rllc}.
In addition to the local Langlands correspondence as
proved by Arthur in \cite{A13}, we need the result for even special orthogonal groups as discussed by H. Atobe and W. T. Gan in \cite{AG}.

We start to prove Proposition \ref{th:itq} in Section \ref{sec-ITQ}. First we show (Lemma \ref{lm:giq}) that for any $\pi\in\Pi(G_n)$,
if the twisted Jacquet module $\CJ_{\CO_\ell}(\pi)$ of Bessel type is nonzero, then it has an irreducible
quotient. Then by applying the relation of multiplicity with parabolic induction (Proposition \ref{pro:MW12} as proved by
M\oe glin and Waldspurger \cite{MW12}), we obtain (Corollary \ref{pp:gtq}) the relation of the first occurrence index with
parabolic induction and the result that for any $\pi\in\Pi(G_n)$ with generic $L$-parameters,
every irreducible quotient of the local descent at the first occurrence index  is square-integrable. It is clear that Corollary \ref{pp:gtq} is
one step towards Theorem \ref{th:ldds}. Finally, Proposition \ref{th:itq} follows from the stability of local descents (Proposition \ref{pp:sld})
and its proof.

In order to prove Theorems \ref{th:ldds} and \ref{th:sdlp}, we have to work on the local $L$-parameters. We define in Section \ref{sec-DLP}
the {\sl descent of local $L$-parameters}. In order to explicitly determine the structure of the descent of local $L$-parameters,
we first calculate explicitly the local epsilon factors associated to a local $L$-parameter or a pair of local
$L$-parameters, and keep tracking the local Langlands data through the process of local descents.
With help of the theorem of M\oe glin and Waldspurger on the
local Gan-Gross-Prasad conjecture for orthogonal groups with generic $L$-parameters (\cite{MW12}) and the explicit local Langlands
correspondence, we undertake a long tedious calculation of the characters that parameterize the distinguished pair of representations
as given by the local Gan-Gross-Prasad conjecture and determine the descent of the local $L$-parameters. The results are stated in
Theorem \ref{pp:dsqlp} that describes explicitly all the local $L$-parameters occurring the descent of the local $L$-parameter associated to
the given representation $\pi$ we start with. The point is that all the local $L$-parameters occurring in the descent of
local $L$-parameters are discrete local $L$-parameters.

With Theorem \ref{pp:dsqlp} at hand, we are able to determine in Section \ref{sec-MTHM}
the local $L$-parameters for all the
irreducible quotients of the local descent $\CD_{\CO_{\ell_0}}(\pi)$. We then use the structure of the Bernstein components of the
local descent to prove Theorem \ref{pp:dsds}
that the irreducible quotients of local descent $\CD_{\CO_{\ell_0}}(\pi)$  belong  to different Bernstein components and are square integrable.
Hence the local descent $\CD_{\CO_{\ell_0}}(\pi)$ is  a multiplicity free direct sum of irreducible square-integrable representations
and hence is square-integrable, which is Theorem \ref{th:ldds}.
Theorem \ref{th:sdlp} that
gives explicit spectral decomposition of the local descent $\CD_{\CO_{\ell_0}}(\pi)$ can now be deduced from Theorem \ref{th:ldds} and
Theorem \ref{pp:dsqlp}.

In Section \ref{sec:ex}, we provide even more explicit results for two special families of generic local
$L$-parameters. One of them is the so called {\sl the local cuspidal $L$-parameters}. The local $L$-packets attached to the local cuspidal $L$-parameters 
are determined by M\oe glin (see \cite{M11}, and also \cite{Xu17}), and they contain at least one supercuspidal member. 
We are going to follow the expression of
the local cuspidal $L$-parameters from \cite{AMS15} (where they are called the local cuspidal Langlands parameters).  
The other is the so called local discrete unipotent $L$-parameters, following J. Arthur (\cite{A89}). The local $L$-packets attached to 
the local discrete unipotent $L$-parameters are determined by M\oe glin (\cite{M96D}). As explained in \cite{M95}, 
they are the same as these classified by G. Lusztig (\cite{Lu95}). Moreover, we will discuss Conjecture \ref{conj-p1} via examples.

{\it Acknowledgements.}
The authors are grateful to Rapha\"el Beuzart-Plessis for his help related to the proof of Lemma \ref{lm:giq} (2), and
to Anne-Marie Aubert for her very helpful comments and suggestions on the first version of this paper, which
make the paper read more smoothly.
Part of the work in this paper
was done during a visit of the authors at the IHES in May, 2016. They would also  like to thank the IHES for hospitality and to thank
Michael Harris for his warm invitation. We would like to thank the referee for helpful comments.

\section{On the Local Gan-Gross-Prasad Conjecture}\label{sec-LGGP}

\subsection{Generic local $L$-parameters and local Vogan packets}\label{ssec-LVP}
We recall from \cite{MW12} the notion of the generic local $L$-parameters and their structures. Without the assumption of the
generalized Ramanujan conjecture, the localization of the generic global Arthur parameters (\cite{A13})
will be examples of the generic local $L$-parameters defined and discussed in \cite{MW12} for a $p$-adic local field $F$ of characteristic zero.
Following \cite{MW12}, we denote by ${\Phi}_\gen(G_n^*)$ the set of conjugate classes of the generic local $L$-parameters of $G_n^*$.
We may simply call $\phi\in{\Phi}_\gen(G_n^*)$ the {\sl generic $L$-parameters} since we only consider the local situation in this paper.

Now we recall from \cite{MW12} the definition of generic local $L$-parameters for orthogonal groups.
We denote by $\CW_F$ the local Weil group of $F$. The local Langlands group of $F$, which is denoted by
$\CL_F$, is equal to the local Weil-Deligne group $\CW_F\times\SL_2(\BC)$. The local $L$-parameters for $G_n^*(F)$ are of the form
\begin{equation}\label{llp}
\phi\ :\ \CL_{F}\rightarrow {^LG_n^*}
\end{equation}
with the property that the restriction of $\phi$ to the local Weil group $\CW_{F}$ is Frobenius semisimple and trivial on an open
subgroup of the inertia group $\CI_{F}$ of $F$, the restriction to $\SL_2(\BC)$ is algebraic, and $\phi$ is compatible with the
projections of $\CL_F$ and ${^LG_n^*}$ to the Weil group $\CW_F$ in the definition.
Then the local $L$-parameters have the following properties:
there exists a datum $(L^*,\phi^{L^*},\udl{\beta})$, for each given local $L$-parameter $\phi$, such that
\begin{enumerate}
\item $L^*$ is a Levi subgroup of $G^*_n(F)$ of the form
$$
L^*=\GL_{n_1}\times\cdots\times\GL_{n_t}\times G_{n_0}^*;
$$
\item $\phi^{L^*}$ is a local $L$-parameter of $L^*$ given by
$$
\phi^{L^*}:=\phi_1\oplus\cdots\oplus\phi_t\oplus\phi_0\ :\ \CL_{F}\rightarrow {^L\!L^*},
$$
where $\phi_j$ is a local tempered $L$-parameter of $\GL_{n_j}$ for $j=1,2,\cdots,t$, and $\phi_0$ is a local tempered $L$-parameter
of $G_{n_0}^*$;
\item $\udl{\beta}:=(\beta_1,\cdots,\beta_t)\in \BR^t$, such that $\beta_1>\beta_2>\cdots>\beta_t>0$; and
\item the parameter $\phi$ can be expressed as
$$
\phi=(\phi_1\otimes|\cdot|^{\beta_1}\oplus\phi_1^\vee\otimes|\cdot|^{-\beta_1})\oplus\cdots\oplus
(\phi_t\otimes|\cdot|^{\beta_t}\oplus\phi_t^\vee\otimes|\cdot|^{-\beta_t})\oplus\phi_0.
$$
\end{enumerate}
Following \cite{A13} and \cite{MW12}, the local $L$-packets can be formed for all $L$-parameters $\phi$ as displayed above, and
are denoted by ${\Pi}_{\phi}(G_n^*)$. Now a local $L$-parameter $\phi$ is called {\sl generic} if the associated local $L$-packet
${\Pi}_{\phi}(G_n^*)$ contains a generic member, i.e. a member with a non-zero Whittaker model with respect to a certain Whittaker data for $G_n^*$.
The set of all generic local $L$-parameters of $G_n^*$ is denoted by $\phi\in{\Phi}_\gen(G_n^*)$.
It was proved in \cite{MW12} that
all the members in such local $L$-packets are given by irreducible standard modules. Note that the situation here is more general
than that considered in \cite{A13} and hence the members in a generic local $L$-packet may not be unitary.
By \cite{MW12}, the local Gan-Gross-Prasad conjecture, which may be called the local GGP conjecture for short, holds for all
generic $L$-parameters $\phi\in{\Phi}_\gen(G_n^*)$.

Recall that an $F$-quasisplit special orthogonal group $G_n^*=\SO(V^*,q^*)$ and its pure inner $F$-forms $G_n=\SO(V,q)$
share the same $L$-group $^LG_n^*$.
As explained in \cite[Section 7]{GGP}, if the dimension $\Fn=\dim V=\dim V^*$ is odd,
one may take $\Sp_{\Fn-1}(\BC)$ to be the $L$-group $^LG_n^*$, and
if the dimension $\Fn=\dim V=\dim V^*$ is even, one may take $\RO_{\Fn}(\BC)$ to be $^LG_n^*$ when $\disc(V^*)$ is not a square in $F^\times$ and
take $\SO_{\Fn}(\BC)$ to be $^LG_n^*$ when $\disc(V^*)$ is a square in $F^\times$.
Let $S_\phi$ be the centralizer of the image of $\phi$ in $\SO_\Fn(\BC)$ or $\Sp_{\Fn-1}(\BC)$, and $S^\circ_\phi$ be its identity connected component group.
Define the component group $\CS_\phi:=S_\phi/S_\phi^\circ$, which is an abelian $2$-group.

By Theorem 1.5.1 of \cite{A13}, and its extension to the generic $L$-parameters in ${\Phi}_\gen(G_n^*)$ in \cite{MW12},
the local $L$-packets $\Pi_\phi(G_n^*)$ are of multiplicity free and there exists a bijection
\begin{equation}\label{bij}
\pi\ \mapsto\  <\cdot,\pi>=\chi_\pi(\cdot)=\chi(\cdot)
\end{equation}
between the finite set ${\Pi}_\phi(G_n^*)$ and the dual of $\CS_\phi/Z({}^LG_n^*)$ of the finite abelian $2$-group $\CS_\phi$ associated to $\phi$.
Following the Whittaker normalization of Arthur, the trivial character $\chi$ corresponds to a generic member in
the local $L$-packet ${\Pi}_\phi(G_n^*)$ with a chosen Whittaker character. There is an $F$-rationality issue on which the bijection may depend.
We will discuss this issue explicitly in Section \ref{ssec-rllc}.
Under the bijection in \eqref{bij}, we may write
\begin{equation}\label{eq-parameter}
\pi=\pi(\phi,\chi)
\end{equation}
in a unique way for each member $\pi\in{\Pi}_\phi(G_n^*)$ with $\chi=\chi_\pi$ as in \eqref{bij}.
For any pure inner $F$-forms $G_n$ of $G_n^*$, the same formulation works (\cite{A13}, \cite{K16}).
The local Vogan packet associated to any generic $L$-parameter $\phi\in{\Phi}_\gen(G_n^*)$ is defined to be
\begin{equation}\label{lvp}
{\Pi}_\phi[G_n^*]:=\cup_{G_n}{\Pi}_\phi(G_n)
\end{equation}
where $G_n$ runs over all pure inner $F$-forms of the given $F$-quasisplit $G_n^*$. The $L$-packet ${\Pi}_\phi(G_n)$ is defined to be
{\sl empty} if the parameter $\phi$ is not $G_n$-relevant.

According to the structure of the generic $L$-parameter $\phi\in{\Phi}_\gen(G_n^*)$, one may easily figure out the structure of the
abelian $2$-group $\CS_\phi$. Write
\begin{equation}\label{decomp}
\phi=\oplus_{i\in \RI}m_i\phi_i,
\end{equation}
which is the decomposition of $\phi$ into simple and generic ones.
The simple, generic local $L$-parameter $\phi_i$ can be written as $\rho_i\boxtimes\mu_{b_i}$,  where $\rho_i$ is an $a_i$-dimensional
irreducible representation of $\CW_F$ and $\mu_{b_i}$ is the irreducible representation of $\SL_2(\BC)$ of dimension $b_i$.
We denote by ${\Phi}_\sg(G^*_n)$ the set of all simple, generic local $L$-parameters of $G_n^*$.
In the decomposition \eqref{decomp},
$\phi_i$ is called {\sl of good parity} if $\phi_i\in{\Phi}_\sg(G^*_{n_i})$ with $G^*_{n_i}$ being the same type as $G^*_n$ where $n_i:=[\frac{a_ib_i}{2}]$.
We denote by $\RI_{\gp}$ the subset of
$\RI$ consisting of indexes $i$ such that $\phi_i$ is of good parity; and by $\RI_{\bp}$ the subset of
$\RI$ consisting of indexes $i$ such that $\phi_i$ is self-dual, but not of good parity. We set $\RI_{\nsd}:=\RI-(\RI_{\gp}\cup\RI_{\bp})$ for
the indexes of non-self-dual $\phi_i$'s.
Hence we may write $\phi\in{\Phi}_\gen(G_n^*)$ in the following more explicit way:
\begin{equation}\label{edec}
\phi=(\oplus_{i\in \RI_\gp}m_i\phi_i)\oplus(\oplus_{j\in \RI_\bp}2m'_j\phi_j)\oplus(\oplus_{k\in \RI_\nsd}m_k(\phi_k\oplus\phi_k^\vee)),
\end{equation}
where $2m'_j=m_j$ in \eqref{decomp} for $j\in \RI_{\bp}$.
According to this explicit decomposition, it is easy to know that
\begin{equation}\label{2grp}
\CS_\phi\cong\BZ_2^{\#\RI_\gp} \text{ or } \BZ_2^{\#\RI_\gp-1}.
\end{equation}
The latter case occurs if $G_n$ is even orthogonal, and some orthogonal summand $\phi_i$ for $i\in \RI_\gp$ has odd dimension.

In all cases, when $G_n$ is even or odd orthogonal, for any $\phi\in{\Phi}_\gen(G_n^*)$, we write elements of $\CS_\phi$ in the following form
\begin{equation}\label{eq:e-i}
(e_i)_{i\in \RI_{\gp}}\in \BZ_{2}^{\#\RI_{\gp}}, \text{ (or simply denoted by $(e_{i})$)}
\end{equation}
where each $e_i$ corresponds to $\phi_i$-component in the decomposition \eqref{edec} for $i\in \RI_\gp$.
The component group  ${\rm Cent}_{\RO_{\Fn}(\BC)}(\phi)/{\rm Cent}_{\RO_{\Fn}(\BC)}(\phi)^\circ$ is denoted by $A_\phi$.
Then $\CS_\phi$ consists of elements in $A_\phi$ with determinant $1$,
which is a subgroup of index $1$ or $2$.
Also write elements in $A_\phi$ of form
$(e_i)$ where $e_i\in\{0,1\}$ corresponds to the $\phi_i$-component in the decomposition \eqref{edec} for $i\in \RI_\gp$.
When $G_n$ is even orthogonal and some $\phi_i$ for $i\in \RI_\gp$ has odd dimension, then $(e_i)_{i\in \RI_{\gp}}$ is in $\CS_\phi$ if and only if $\sum_{i\in\RI_\gp}e_i\dim\phi_i$ is even.

An $L$-parameter $\phi$ is of orthogonal type (resp. symplectic type) if its image $\Im(\phi)$ lies in $\RO_{\Fn}(\BC)$ (resp. $\Sp_{2n}(\BC)$).
In this paper, a self-dual $L$-parameter refers to be of either  orthogonal type or symplectic type.
Let $\rho$ be an irreducible smooth representation of $\CW_F$,
which is  Frobenius semisimple   and trivial on an open subgroup of the inertia group $\CI_{F}$ of $F$.
Similarly, $\rho$ is  of orthogonal type (resp. symplectic type) if $\rho\boxtimes 1$ is  of orthogonal type (resp. symplectic type).
In this case, $\rho\boxtimes 1$ is a discrete $L$-parameter.

\subsection{Rationality and the local Langlands correspondence}\label{ssec-rllc}
As explained in Section \ref{ssec-MT}, one of our motivations is to study the algebraic version of the wave-front set $\Fp^m(\pi)$ of $\pi$ via this local descent method.
Our main approach is to perform an induction argument on the parts of partitions $\underline{p}$ in $\Fp^{m}(\pi)$.
In this inductive argument, the representations are descended to the ones of special orthogonal groups with different parity alternatively.
We need to keep tracking the $F$-rational nilpotent orbits $\CO_{\ell}$, which give the nonzero local descents.
Meanwhile, $\CO_{\ell}$ also determines the quadratic forms of the descendant special orthogonal groups.
On the other hand, one needs to fix a normalization of the Whittaker datum in order to fix the local Langlands correspondence.
Such a normalization depends also on the quadratic forms of the special orthogonal groups.
Hence, we will fix the normalization for the parent representations and track the $F$-rational forms $\CO_{\ell}$ for their descendants,
then the normalization for  their descendants will be determined.
In this sense, we refer those normalizations to be the $F$-rationality of the local Langlands correspondence for $G_n=\SO(V_\Fn)$ in terms of the $F$-rationality of the underlying quadratic space $(V_\Fn,q_\Fn)$.
When more quadratic spaces get involved in the discussion, we denote by $q_\Fn$ the quadratic form of $V_\Fn$,
which was simply denoted by $q$ before.

Define the discriminant of the quadratic space $V_\Fn$ by
$$
\disc(V_\Fn)= (-1)^{\frac{\Fn(\Fn-1)}{2}}\det(V_\Fn)\in F^\times/F^{\times 2}
$$
and, similar to \cite[Page 167]{Om}, define the Hasse invariant of $V_\Fn$ by
$$
\Hss(V_\Fn)=\prod_{1\leq i\leq j\leq \Fn}(\alpha_i,\alpha_j),
$$
where $V_\Fn$ is decomposed orthogonally as $Fv_1\oplus Fv_2\oplus\cdots \oplus Fv_\Fn$ with $q_\Fn(v_i,v_i)=\alpha_i\in F^\times$.
According to this definition, if $V_\Fn$ is decomposed orthogonally as $V_\Fn=W\oplus U$,
we have, by Remark 58:3 \cite{Om}, the following formulas:
\begin{align}
&\disc(V_{\Fn})=(-1)^{ab}\disc(W)\disc(U)\\
&\Hss(V_{\Fn})=\Hss(W)\Hss(U)((-1)^{\frac{a(a-1)}{2}}\disc(W),(-1)^{\frac{b(b-1)}{2}}\disc(U))
\label{eq:decomp-Hasse}
\end{align}
where $a=\dim W$ and $b=\dim U$.

Recall from Section \ref{ssec-LD} that for each $F$-rational orbit $\CO_\ell$ the non-degenerate character $\psi_{\CO_\ell}$ is given by $\psi_{\ell,w_0}$ defined in \eqref{eq:psi-ell} where $w_0$ is an anisotropic vector $w_0$.
Define
$$
\disc(\CO_\ell):=\disc(V^+_\ell\oplus Fw_0\oplus V^-_\ell)$$
and
$$
\Hss(\CO_\ell):=\Hss(V^+_\ell\oplus Fw_0\oplus V^-_\ell),
$$
where $V^\pm_{\ell}$ is defined in \eqref{eq:V-ell}.
Since the quadratic space $V^+_\ell\oplus Fw_0\oplus V^-_\ell$ is split, one has
\begin{eqnarray*}
\disc(\CO_\ell)&=&q_\Fn(w_0,w_0),\\
\Hss(\CO_\ell)&=&(-1,-1)^{\frac{\ell(\ell+1)}{2}}((-1)^{\ell+1},\disc(\CO_\ell)).
\end{eqnarray*}

Let $W_\ell$ be defined in \eqref{eq:w}. We have the decomposition
\begin{equation}\label{eq:decomp}
V_{\Fn}=(V^+_\ell\oplus Fw_0\oplus V^-_\ell)\oplus  W_\ell.
\end{equation}
Then one has
$\disc(W_\ell)=(-1)^{\Fn-1}\disc(V_\Fn)\disc(\CO_\ell)$,
and it is easy to check that
\begin{align}
 & \Hss(W_\ell) \label{eq:Hasse-W}\\
=&(-1,-1)^{\frac{\ell(\ell+1)}{2}}\Hss(V_{\Fn})
((-1)^{\ell}\disc(\CO_\ell),(-1)^{\frac{\Fn(\Fn-1)}{2}+\ell}\disc(V_{\Fn})). \nonumber
\end{align}
In order to fix the $F$-rationality of the local Langlands correspondence, we adopt the {\bf (QD) condition} of Waldspurger in \cite[Page 119]{W12b} for the even special orthogonal group $(V_\Fn,q_\Fn)$:
\begin{itemize}
	\item[($\QD$)] {\sl The special orthogonal group of $(V_\Fn,q_\Fn)\oplus (F,q_{0})$ is split,
where $(F,q_{0})$ is the one-dimensional quadratic space with $q_{0}(x,y)= -a\cdot xy$.}
\end{itemize}
Here $a\in \CZ$ will be specified later.
In \cite{W12b}, $a$ is denoted  $2\nu_0$.

\begin{lem}\label{lm:QD}
Assume that $\Fn$ is even.
Then $(V_\Fn,q_\Fn)$ satisfies $(\QD)$ if and only if
\begin{equation}\label{eq:QD}
\Hss(V_\Fn)=(-1,-1)^{\frac{n(n+1)}{2}}((-1)^{n}a,\disc(V_\Fn)),	
\end{equation}
where $n=\frac{\Fn}{2}$.
\end{lem}
\begin{proof}
Write $V'=(V_\Fn,q_\Fn)\oplus (F,q_{0})$.
Since
$\disc(V')=-a\cdot \disc(V_\Fn)$, we deduce
\begin{align}
\Hss(V')= & \Hss(V_{\Fn})((-1)^{n}\disc(V_{\Fn}),-a)(-a,-a) \nonumber\\
=&\Hss(V_{\Fn})((-1)^{n+1}\disc(V_{\Fn}),-a).\label{eq:lm:QD-1}
\end{align}

Note that
$\SO(V')$ is split if and only if
$V'$ is isometric to the quadratic space
$\BH^n\oplus Fv_0$ for some $v_0$,
where $\BH$ is a hyperbolic plane, equivalently $q_\Fn(v_0,v_0)=\disc(V')$ and $\Hss(V')=\Hss(\BH^n\oplus Fv_0)$.
Under the assumption that $q_\Fn(v_0,v_0)=\disc(V')$,
by \eqref{eq:decomp-Hasse}, we have
\begin{align}
\Hss(\BH^n\oplus Fv_0)= & \Hss(\BH^n)((-1)^n,\disc(V')) (\disc(V'),\disc(V'))\nonumber\\
=&(-1,-1)^{\frac{n(n+1)}{2}}(-1,\disc(V'))^{ n+1 }.\label{eq:lm:QD-2}
\end{align}
$\SO(V')$ is split if and only if \eqref{eq:lm:QD-1} equals \eqref{eq:lm:QD-2}.
By using the relation that $\disc(V')=-a\cdot \disc(V_\Fn)$ and by simplifying the equality, we obtain this lemma.
\end{proof}

For example, suppose that $\dim V^*_{0}=2$, that is, $\SO(V^*_\Fn)$ is a quasi-split and non-split even special orthogonal group, whose pure inner form $\SO(V_\Fn)$ satisfies $\disc(V_\Fn)=\disc(V^*_\Fn)$ and $\Hss(V_\Fn)=-\Hss(V^*_\Fn)$.
Note that $\SO(V_\Fn)$ and $\SO(V^*_\Fn)$ are $F$-isomorphic.
Recall that ${\Pi}_\phi[G_n^*]$ is a generic Vogan packet of $\SO(V^*_\Fn)$.
The issue is  which orthogonal group  shall be assigned to $\chi\in\widehat{\CS}_\phi$ with $\chi((1))=-1$ (see \cite[Section 10]{GGP}).
We fix the $F$-rationality of the local Langlands correspondence following \cite[Section 4.6]{W12b}. This means that
for $\pi(\phi,\chi)\in {\Pi}_\phi[G_n^*]$ of $\SO(V')$ where $V'\in\{V_\Fn,V^*_\Fn\}$,
the quadratic space $V'$ is determined by \eqref{eq:QD}.

For the even special orthogonal groups, the local Langlands correspondence needs more explanation.
Define $c$ to be an element in $\RO(V_\Fn)\smallsetminus \SO(V_\Fn)$ with $\det(c)=-1$.
For instance, when $\Fn$ is odd, we can take $c=-I_\Fn$.
Consider the conjugate action of $c$ on $\Pi(G_n)$, which arises an equivalence relation $\sim_c$ on $\Pi(G_n)$.
Obviously, when $\Fn$ is odd, the $c$-conjugation is trivial.
We only discuss the even special orthogonal case here.
Denote $\Pi(G_n)/\sim_c$ to be the set of equivalence classes.
For $\sig\in\Pi(G_n)$, denote $[\sig]_c$ to be the equivalence class of $\sig$.
Similarly, one has an analogous equivalence relation on the set $\Phi(G_n)$ of all $L$-parameters of $G_n$, which is also denoted by $\sim_c$.

Let us recall the desiderata of the weak local Langlands correspondence for even special orthogonal groups $\SO(V_\Fn)$ from
\cite[Desideratum 3.2]{AG}.
For the needs of this paper, we only recall some partial facts from their desiderata, which has been verified in \cite{AG},
in order to fix the rationality of the local Langlands correspondence.

{\bf A Weak Local Langlands Correspondence for $G_n=\SO(V_{2n})$:}
\begin{enumerate}
	\item There exists a canonical surjection
	$$
	\bigsqcup_{G_n}\Pi(G_n)/\sim_c\   \longrightarrow \  \Phi(G^*_n)/\sim_c
	$$
	where $G_n$ runs over all pure inner forms of $G^*_n$.
	Note that the preimage of $\phi$ under the above map  is the Vogan packet ${\Pi}_\phi[G^*_n]$ associated to $\phi$.
	\item Let $\Phi^c(G^*_n)/\sim_c$ be the subset of $\Phi(G^*_n)/\sim_c$ consisting of $\phi$ which contains an irreducible orthogonal subrepresentation of $\CL_F$ with odd dimension.
	The following are equivalent:
	\begin{itemize}
		\item $\phi\in \Phi^c(G^*_n)/\sim_c$;
		\item some $[\sig]_c\in \Pi_\phi[G^*_n]$ satisfies that $\sig\circ\Ad(c)\cong \sig$;
		\item all $[\sig]_c\in \Pi_\phi[G^*_n]$ satisfy that $\sig\circ\Ad(c)\cong \sig$.
	\end{itemize}
	\item For each $a\in \CZ$, there exists a bijection
	$$
	\iota_a\colon {\Pi}_\phi[G^*_n]\longrightarrow \widehat{\CS_\phi},
	$$
	which satisfies the endoscopic and twisted endoscopic character identities (refer to \cite{A13} and \cite{K16} for instance).
	\item For $[\sig]_c\in {\Pi}_\phi[G^*_n]$ and $a\in\CZ$, the following are equivalent:
	\begin{itemize}
	 	\item $\sig\in \Pi(\SO(V_{2n}))$;
	 	\item
	 	$\iota_{a}([\sig]_c)((1))$ and $\Hss(V_{2n})$ satisfy the following equation
	 	\begin{equation}\label{eq:LLC-normalization}
		\Hss(V_{2n})=\iota_{a}([\sig]_c)((1))(-1,-1)^{\frac{n(n+1)}{2}}((-1)^n a,\disc(V_{2n})).	
		\end{equation}
	 \end{itemize}
\end{enumerate}
Note that the subscript $a$ in $\iota_a$ is used to indicate the $F$-rationality of the Whittaker datum.
The details can be found in \cite[Section 3]{AG}.

By the above weak local Langlands correspondence for $\SO(V_{2n})$ and the local Langlands correspondence for $\SO(V_{2n+1})$,
each irreducible admissible representation $\pi\in\Pi(G_n)$
is associated to an equivalence class $[\phi]_c$ of $L$-parameters under the $c$-conjugation.
Following \cite[Sections 9 and 10]{GGP} and \cite[Proposition 3.5]{AG}, define the action of $\CZ$  on $\widehat{\CS}_\phi$ via an one-dimensional twist given by the local  Langlands correspondence $\iota_{a}$, which is
$$
a\cdot \chi\to \chi\otimes \eta_a
$$
where $\eta_a((e_i)_{i\in \RI_\gp})=(\det\phi_i,a)^{e_i}_F$ and $(\cdot,\cdot)_F$ is the Hilbert symbol defined over $F$.
Denote  by $\CO_{\CZ}(\pi)$ the orbit in $\widehat{\CS}_\phi$ corresponding to $\pi$.
More precisely, if $\pi=\pi(\phi,\chi)$ under the local Langlands correspondence $\iota_a$ for some $a$, one has
\begin{equation}\label{eq:CZ}
\CO_{\CZ}(\pi)=\{\chi\otimes\eta_\alpha\in \widehat{\CS}_\phi\colon \alpha\in \CZ\}.
\end{equation}
Note that the set $\CO_{\CZ}(\pi)$
is uniquely determined by  $\pi$ and independent of the choice of the local Langlands correspondence $\iota_a$.

{\it In the rest of this paper, the local Langlands correspondence is referred to the weak local Langlands correspondence for even special orthogonal groups $\SO(V_{2n})$,
and the local Langlands correspondence for odd special orthogonal groups $\SO(V_{2n+1})$.}

Under the local Langlands correspondence $\iota_a$ of $G_n(F)$ with some $a\in \CZ$, for an $L$-parameter $\phi$ and a character $\chi\in\widehat{\CS}_\phi$, denote by $\pi_a(\phi,\chi)$ the corresponding irreducible admissible representation of $G_n(F)$.
Conversely, given an irreducible admissible representation $\pi$ of $G_n(F)$,
denote by $\phi_a(\pi)$ and $\chi_a(\pi)$ the associated $L$-parameter and its corresponding character in $\widehat{\CS}_\phi$, respectively.

When $G_n=\SO(V_{2n+1})$, the local Langlands correspondence is unique and independent with the choice of $a$.
We may denote the action of $\CZ$ to be the trivial action and then $\CO_\CZ(\pi)$ contains only $\pi$, and
simply write $\pi(\phi,\chi)$ and $(\phi(\pi),\chi(\pi))$, respectively.

\begin{rmk}
Let $\phi$ be an $L$-parameter of $\SO(V_{2n})$.
Suppose that all irreducible orthogonal summands of $\phi$ are even dimensional. Then the $c$-conjugate  $L$-parameter $\phi^c$ is different from
$\phi$ because $\phi^c$ is not $G^\vee_n(\BC)=\SO_{2n}(\BC)$-conjugate to $\phi$. It follows that
$\Pi_\phi[G_n^*]$ and $\Pi_{\phi^c}[G_n^*]$ are two different Vogan packets.
However, the conjugation $\Ad(c)\colon \pi_a(\phi,\chi)\mapsto\pi_a(\phi^c,\chi)$ gives a bijection between $\Pi_\phi[G_n^*]$
and $\Pi_{\phi^c}[G_n^*]$.
According to \cite[Section 3]{AG}, the $c$-conjugation stabilizes the Whittaker datum associated to $\iota_a$.
Thus, the corresponding characters of $\CS_\phi$ associated to  $\pi$ and $\pi\circ\Ad(c)$ under the same local Langlands correspondence are identical.
\end{rmk}

\subsection{On the local GGP conjecture: multiplicity one}\label{ssec-LGGP}
The local GGP conjecture  was explicitly formulated in \cite{GGP} for general classical groups.
We recall the case of orthogonal groups here.

Let $\Fn$ and $\Fm$ be two positive integers with different parity.
For a relevant pair $G_n=\SO(V,q_V)$ and $H_m=\SO(W,q_W)$, and an $F$-quasisplit relevant pair $G_n^*=\SO(V^*,q^*_V)$ and $H_m^*=\SO(W^*,q^*_W)$
in the sense of \cite{GGP}, where $m=[\frac{\Fm}{2}]$ with $\Fm=\dim W=\dim W^*$,
we are going to discuss the local $L$-parameters for the group $G_n^*\times H_m^*$ and its relevant pure inner $F$-form
$G_n\times H_m$.
Consider admissible group homomorphisms:
\begin{equation}\label{llpgh}
\phi\ :\ \CL_F=\CW_F\times\SL_2(\BC)\rightarrow {^LG_n^*}\times{^LH_m^*},
\end{equation}
with the properties as described for the local $L$-parameters in \eqref{llp}. We consider those $L$-parameters analogous to $\Phi_\gen(G_n^*)$, and
denote the set of those $L$-parameters by ${\Phi}_\gen(G_n^*\times H_m^*)$.
To each parameter $\phi\in {\Phi}_\gen(G_n^*\times H_m^*)$, one defines the associated local $L$-packet
${\Pi}_\phi(G_n^*\times H_m^*)$, as in \cite{MW12}.
For any relevant pure inner $F$-form $G_n\times H_m$, if a parameter
$\phi\in {\Phi}_\gen(G_n^*\times H_m^*)$
is $G_n\times H_m$-relevant, it defines a local $L$-packet
${\Pi}_\phi(G_n\times H_m)$, following \cite{A13} and \cite{MW12}.
If a parameter $\phi\in {\Phi}_\gen(G_n^*\times H_m^*)$
is not $G_n\times H_m$-relevant,
the corresponding local $L$-packet ${\Pi}_\phi(G_n\times H_m)$ is defined to be the empty set.
The local Vogan packet associated to a parameter $\phi\in {\Phi}_\gen(G_n^*\times H_m^*)$
is defined to be the union of the local $L$-packets
${\Pi}_\phi(G_n\times H_m)$
over all relevant pure inner $F$-forms $G_n\times H_m$ of the $F$-quasisplit group $G_n^*\times H_m^*$, which is denoted by
$$
{\Pi}_\phi[G_n^*\times H_m^*].
$$

The local GGP conjecture (\cite{GGP}) is formulated in terms of the local Bessel functionals as introduced in Section
\ref{ssec-LD}. For a given relevant pair $(G_n,H_m)$, assuming that $\Fn\geq \Fm$, take a partition of the form:
$$
\udl{p}_\ell:=[(2\ell+1)1^{{\Fn-2\ell-1}}],
$$
where $2\ell+1=\dim W^\perp=\Fn-\Fm$. As in Section \ref{ssec-LD},
the $F$-stable nilpotent orbit $\CO_{\udl{p}_\ell}^\st$ corresponding to the partition $\udl{p}_\ell$ defines a unipotent subgroup
$V_{\udl{p}_\ell}$ and a generic character $\psi_{\CO_\ell}$ associated to any $F$-rational orbit $\CO_\ell$ in the $F$-stable orbit $\CO_{\udl{p}_\ell}^\st$.
Following from \cite{JZ-BF}, there is an $F$-rational orbit $\CO_\ell$ in the $F$-stable orbit
$\CO_{\udl{p}_\ell}^\st$  such
that the subgroup $H_m=G_{n}^{\CO_\ell}$ normalizes the unipotent subgroup $V_{\udl{p}_\ell}$ and stabilizes the character $\psi_{\CO_\ell}$.
As in Section \ref{ssec-LD} again, the uniqueness of local Bessel functionals asserts that
$$
\dim \Hom_{R_{\CO_\ell}(F)}(\pi\otimes\sigma^\vee,\psi_{\CO_\ell})\leq 1,
$$
as proved in \cite{AGRS}, \cite{SZ}, \cite{GGP}, and \cite{JSZ}. The stronger version in terms of Vogan packets
is formulated as follows.

\begin{thm}[M\oe glin-Waldspurger]\label{conj:smo}
Let $G_n^*$ and $H_m^*$ be a relevant pair as given above.
For any local $L$-parameter $\phi\in {\Phi}_\gen(G_n^*\times H_m^*)$, the following identity holds:
\begin{equation}\label{smo}
\sum_{\pi\otimes\sigma\in \Pi_\phi[G_n^*\times H_m^*]}
\dim \Hom_{R_{\CO_\ell}(F)}(\pi\otimes\sigma,\psi_{\CO_\ell})=1.
\end{equation}
\end{thm}

Theorem \ref{conj:smo} is the orthogonal group case of the general local GGP conjecture (\cite{GGP}) over $p$-adic local fields.
It was proved
by Waldspurger (\cite{W10}, \cite{W12a} and \cite{W12b}) for tempered local $L$-parameters, and by M\oe glin and Waldspurger (\cite{MW12}) for generic local $L$-parameters.

\subsection{On multiplicity and parabolic induction}\label{ssec-MPI}
This is to discuss the relation between the multiplicity in the local GGP conjecture and the parabolic induction as given in the
work of M\oe glin and Waldspurger in a series of papers \cite{W10,MW12,W12a,W12b} for orthogonal groups.


Let $\phi$ and $\varphi$ be generic $L$-parameters of different type and be of even dimension.
Denote $\pi^{\GL}(\phi)$ and $\pi^{\GL}(\varphi)$ to be the two irreducible representations of general linear groups via the local Langlands functoriality, which is independent with the choice of $\varphi$ in $[\varphi]_c$.
Define
\begin{equation}\label{eq:CE}
\CE(\varphi,\phi)=\det(\varphi)(-1)^{n}\cdot \varepsilon(\frac{1}{2},\pi^{\GL}(\varphi)\times \pi^{\GL}(\phi),\psi_F),	
\end{equation}
where $\varepsilon(s,\cdot,\psi_F)$ is the $\varepsilon$-factor defined by H. Jacquet, I. Piatetski-Shapiro and J. Shalika in \cite{JPSS}.
Recall that $\det(\varphi)(-1)=(\det(\varphi),-1)_F$ and $(\cdot,\cdot)_F$ is the Hilbert symbol defined over $F$.
Decompose $\varphi$ and $\phi$ as \eqref{edec}, whose index sets are denoted by $\RI^\phi_{\rm gp}$ and $\RI^\varphi_{\rm gp}$ respectively,
and
define the character $\chi^\star(\phi,\varphi)$ to be the pair $(\chi^\star_\phi(\phi,\varphi),\chi^\star_\varphi(\phi,\varphi))$ (or simply $(\chi^\star_\phi, \chi^\star_{\varphi})$), where
\begin{equation}\label{eq:dist-eps-odd-W}
	{\bf \chi^\star_\phi}((e_{i'})_{i'\in \RI^{\phi}_{\gp}})=\prod_{i'\in \RI^{\phi}_{\gp}}\CE(\varphi,\phi_{i'})^{e_{i'}}
\end{equation}
and
\begin{equation}\label{eq:dist-eps-even-W}
	{\bf \chi^\star_\varphi}((e_{i})_{i\in \RI^{\varphi}_{\gp}})=\prod_{i\in \RI^{\varphi}_{\gp}}\CE(\varphi_i,\phi)^{e_{i}}.
\end{equation}
By convention, if $\varphi$ or $\phi$ equals $0$, then $\chi^\star(\phi,\varphi)=(1,1)$.

Note that $\chi^\star_\phi$ and $\chi^\star_\varphi$ belong to $\widehat{\CS}_\phi$ and $\widehat{\CS}_\varphi$, respectively, and are
independent of the choice of $\psi_F$ defining local root numbers (see \cite[Section 6]{GGP}).
It is easy to see that
$$
\chi^\star_\phi((1))=\chi^\star_\varphi((1))=\CE(\varphi,\phi).
$$
By \cite[Sections 6 and 18]{GGP}, the character $\chi^\star(\phi,\varphi)$  of $\CS_\phi \times \CS_{\varphi}$ only depends on $\phi$ and $[\varphi]_c$.

For $\pi\in\Pi(G_n)$ with $G_n=\SO(V_{2n})$ and $\sig\in\Pi(H_m)$ with $H_m=\SO(V_{2m+1})$,
define the multiplicity for the pair $(\pi,\sig)$ by
\begin{equation}\label{mult}
m(\pi,\sig)= \begin{cases}
\dim\Hom_{R_{\CO_\ell}}(\pi,\sig\otimes\psi_{\CO_\ell}) & \text{ if }n> m,\\
\dim\Hom_{R_{\CO_\ell}}(\sig,\pi\otimes\psi_{\CO_\ell}) & \text{ if }n\leq m.
\end{cases}
\end{equation}

\begin{thm}[Theorem in Section 4.9 of \cite{W12b}]\label{thm:W12b}\label{lm:Local-GGP}
Assume that the twisted endoscopic and twisted endoscopic character identities as described in \cite[Sections 4.2 and 4.3]{W12b} hold.
Suppose that $\varphi$ and $\phi$ are generic $L$-parameters of $G_n=\SO(V_{2n})$ and $H_m=\SO(W_{2m+1})$.
There exists an isometry class of quadratic spaces $V\times W$, unique up to a scalar multiplication, that satisfies the following conditions:
\begin{align}
&\disc(V)= \det(\varphi) \label{eq:thm:DV} \\
&\Hss(W)= (-1,-1)^{\frac{m(m+1)}{2}}((-1)^{m+1},\disc(W))\CE(\varphi,\phi)
\label{eq:thm:HW}\\
&\Hss(V)= (-1,-1)^{\frac{n(n+1)}{2}}((-1)^n \disc(W),\disc(V))\CE(\varphi,\phi).
\label{eq:thm:HV}
\end{align}
Moreover, under the local Langlands correspondence $\iota_{a}$
$$m(\pi_a(\varphi,\chi^\star_\varphi),\pi_a(\phi,\chi^\star_{\phi}))=1,$$
where $a=-\disc(W)\disc(V)$.
\end{thm}
It must be mentioned that the assumption on the twisted endoscopic and twisted endoscopic character identities for the quasi-split groups in Theorem \ref{thm:W12b} is removed
by the work of M\oe glin and Waldspurger (\cite{MW-STTF}) on the stabilization of the twisted trace formula.

We note that in \cite[Theorem 5.2]{AG}, Atobe and Gan give a version of local GGP conjecture in terms of the weak local Langlands correspondence.

\begin{rmk}\label{rmk:O-V-W}
Suppose that $V\times W$ satisfies the conditions in Theorem \ref{thm:W12b}, so is
$\lam V\times \lam W$ for any $\lam \in\CZ$.
More precisely, after choosing $\disc(W)\in \CZ$,
one can choose an $F$-rational orbit $\CO_\ell$ with $\disc (\CO_\ell)=(-1)^{\min\{2m+1,2n\}}\disc(V)\disc(W)$.
Then there is a unique isometry class $V\times W$ satisfying
Equations \eqref{eq:thm:DV}, \eqref{eq:thm:HW} and \eqref{eq:thm:HV}, which is associated with $\CO_\ell$.
\end{rmk}

The relation between the multiplicity and the parabolic induction is given by the following proposition of M\oe glin and Waldspurger.

\begin{prop}[Proposition (1.3) \cite{MW12}]\label{pro:MW12}
Assume that $\pi$ is the induced representation
$$
\Ind^{G_n}_{P}\tau_1|\det|^{\alpha_1}\otimes\cdots\otimes\tau_r|\det|^{\alpha_r}\otimes\pi_0
$$
and $\sig$ is the induced representation
$$
\Ind^{H_m}_{Q}\tau'_1|\det|^{\beta_1}\otimes\cdots\otimes\tau'_t|\det|^{\beta_t}\otimes\sig_0
$$
where $\alpha_1\geq \alpha_2\geq \cdots\geq \alpha_r\geq 0$,
$\beta_1\geq\beta_2\geq\cdots\geq\beta_t\geq 0$,
all $\tau_i$ and $\tau'_j$ are unitary tempered irreducible representations of general linear groups, and $\pi_0$ and $\sig_0$ are tempered irreducible representations of classical groups of smaller rank. Then the following equation of multiplicities holds:
$$
m(\pi,\sig^{\vee})=m(\pi_0,\sig_0).
$$
\end{prop}

\section{Proof of Proposition \ref{th:itq}}\label{sec-ITQ}

The proof of Proposition \ref{th:itq} takes a few steps. We first prove the following lemma, which implies Proposition \ref{pp:iq}.
\begin{lem}\label{lm:giq}
For any $\pi\in\Pi(G_n)$, the following hold.
\begin{enumerate}
\item Each Bernstein component of the $\ell$-th twisted Jacquet module $\CJ_{\CO_\ell}(\pi)$ is finitely generated.
\item If $\CJ_{\CO_\ell}(\pi)\ne 0$, then there exists an irreducible representation $\sig\in\Pi(G^{\CO_{\ell}}_{n})$
such that
$$
\Hom_{G^{\CO_{\ell}}_{n}(F)}(\CJ_{\CO_\ell}(\pi),\sig)\ne0.
$$
\item If $\CJ_{\CO_\ell}(\pi)\ne 0$, then the $\ell$-th local descent $\CD_{\CO_\ell}(\pi)$ of $\pi$ is not zero.
\end{enumerate}
\end{lem}

\begin{proof}
It is clear that (3) follows from (2). Assume that (1) holds. Then (2) follows, since any smooth representation of finite length has
an irreducible quotient. Hence we only need to show that (1) holds.

For any $\pi\in\Pi(G_n)$, let $\End(\pi)$ be the space of continuous endomorphisms of the space of $\pi$.
We have the following $G_n(F)\times G_n(F)$-equivariant homomorphism:
$$
f\in \CS(G_n(F)):=C^\infty_c(G_n(F))\mapsto \pi(f)\in\End(\pi)\simeq\pi\otimes\pi^\vee
$$
where $\pi(f)(v):=\int_{G_n(F)} f(g)\pi(g)v\ud g$ is the convolution operator induced by $\pi$ for all Bruhat-Schwartz functions
$f\in \CS(G_n(F))$. It is not hard to check that this homomorphism is surjective.
By the separation of variables,
the homomorphism induces a surjective homomorphism
\begin{eqnarray*}
\CS(G_n(F)\times G^{\CO_{\ell}}_{n}(F))
&=&\CS(G_n(F))\otimes \CS(G^{\CO_{\ell}}_{n}(F))\\
&\twoheadrightarrow&
\pi\otimes\pi^\vee\otimes\CS(G^{\CO_{\ell}}_{n}(F)).
\end{eqnarray*}
By taking the $(V_{\udl{p}_\ell},\psi_{\CO_\ell})$-coinvariant for the action by the left translation of $\CS(G_n(F))$,
one has the projection
$$
{}_{(V_{\udl{p}_\ell},\psi_{\CO_\ell})}\!(\CS(G_n)(F))\otimes \CS(G^{\CO_{\ell}}_{n}(F))
\twoheadrightarrow \CJ_{\CO_\ell}(\pi)\otimes\pi^\vee\otimes \CS(G^{\CO_{\ell}}_{n}(F)).
$$
Then taking $G^{\CO_{\ell}}_{n}$-coinvariant for the action by the left translation, one obtains a surjective homomorphism
$$
{}_{G^{\CO_{\ell}}_{n}}[{}_{(V_{\udl{p}_\ell},\psi_{\CO_\ell})}\!(\CS(G_n(F)))\otimes \CS(G^{\CO_{\ell}}_{n}(F))]
\twoheadrightarrow \pi^\vee\otimes
{}_{G^{\CO_{\ell}}_{n}}[\CJ_{\CO_\ell}(\pi)\otimes \CS(G^{\CO_{\ell}}_{n}(F))].
$$
Note that $G^{\CO_{\ell}}_{n}(F)$ acts on ${}_{G^{\CO_{\ell}}_{n}}[\CJ_{\CO_\ell}(\pi)\otimes \CS(G^{\CO_{\ell}}_{n}(F))]$ via the left translation on the second variable. As smooth representations of $G^{\CO_{\ell}}_{n}(F)$, one has a surjection
$$
{}_{G^{\CO_{\ell}}_{n}}[\CJ_{\CO_\ell}(\pi)\otimes \CS(G^{\CO_{\ell}}_{n}(F))]\twoheadrightarrow \CJ_{\CO_\ell}(\pi).
$$
We need to show that the map
$$
\Fp \colon \CJ_{\CO_\ell}(\pi)\otimes \CS(G^{\CO_{\ell}}_{n}(F))\to \CJ_{\CO_\ell}(\pi)
$$
defined by
$
v\otimes f\mapsto \int_{G^{\CO_{\ell}}_{n}(F)}f(h^{-1})\CJ_{\CO_\ell}(\pi)(h)v\ud h
$
factors through the quotient ${}_{G^{\CO_{\ell}}_{n}}[\CJ_{\CO_\ell}(\pi)\otimes \CS(G^{\CO_{\ell}}_{n}(F))]$ and is surjective.
First, for any smooth vector $v\in \CJ_{\CO_\ell}(\pi)$,
suppose that $v$ is fixed by a compact open subgroup $K_0$ of $G^{\CO_{\ell}}_{n}(F)$.
Let $1_{K_0}$ be a characteristic function of $K_0$ in $\CS(G^{\CO_{\ell}}_{n}(F))$.
By choosing a suitable non-zero constant $c_0$, we have
$
v\otimes c_0\cdot 1_{K_0}\mapsto v.
$
It follows that this map is surjective.
Then, if $v\otimes f$ is of form $\sum (\pi(h_i)v_i\otimes L(h_i)f_i-v_i\otimes f_i)$ for some $h_i$, $v_i$ and $f_i$, one must have that $\Fp(v\otimes f)=0$.
Thus, the map $\Fp$ factors through the quotient ${}_{G^{\CO_{\ell}}_{n}}[\CJ_{\CO_\ell}(\pi)\otimes \CS(G^{\CO_{\ell}}_{n}(F))]$.

Because
$$
\CS((G^{\CO_{\ell}}_{n}\ltimes V_{\udl{p}_\ell})\bks G_n\times G^{\CO_{\ell}}_{n},\psi_{\CO_\ell})
={}_{G^{\CO_{\ell}}_{n}}({}_{(V_{\udl{p}_\ell},\psi_{\CO_\ell})}\!(\CS(G_n))\otimes \CS(G^{\CO_{\ell}}_{n})),
$$
we obtain a projection
$$
\CS((G^{\CO_{\ell}}_{n}(F)\ltimes V_{\udl{p}_\ell}(F))\bks G_n(F)\times G^{\CO_{\ell}}_{n}(F),\psi_{\CO_\ell})\twoheadrightarrow \pi^\vee\otimes\CJ_{\CO_\ell}(\pi).
$$
By Theorem A and Remark of Theorem A in \cite{AAG}, as a representation of $G_n(F)\times G_n^{\CO_\ell}(F)$,
each Bernstein component of
$$\CS((G^{\CO_{\ell}}_{n}(F)\ltimes V_{\udl{p}_\ell}(F))\bks G_n(F)\times G^{\CO_{\ell}}_{n}(F),\psi_{\CO_\ell})$$
is finitely generated, and so is each Bernstein component of $\CJ_{\CO_\ell}(\pi)$.
This finishes the proof of (1).
\end{proof}

Assume that $\pi\in\Pi(G_n)$ has a generic $L$-parameter.
By the corollary in \cite[Section 2.14]{MW12},   $\pi$ can be written as the irreducible induced representation (standard module)
\begin{equation}\label{pigp}
\Ind^{G_n(F)}_{P(F)}\tau_1|\det|^{\alpha_1}\otimes\cdots \otimes \tau_t|\det|^{\alpha_t}\otimes\pi_0
\end{equation}
where $\alpha_1>\alpha_2>\cdots>\alpha_t>0$, all $\tau_i$ and $\pi_0$ are irreducible unitary tempered representations.
One may write the Levi subgroup of $P$ as $\GL_{a_1}\times\cdots\times\GL_{a_t}\times G_{n_0}$ and has that $\pi_0\in\Pi_\temp(G_{n_0})$.

As a corollary to Proposition \ref{pro:MW12}, one has
\begin{cor}\label{pp:gtq}
For any $\pi\in\Pi(G_n)$ that has a generic $L$-parameter and is written as in \eqref{pigp}, the following hold.
\begin{enumerate}
\item The first occurrence indexes $\ell_0(\pi)$ and $\ell_0(\pi_0)$, with $\pi_0$ being in \eqref{pigp}, enjoy the relation:
$
\ell_0(\pi)=n-n_0+\ell_0(\pi_0).
$
\item Every irreducible quotient of the local descent $\CD_{\CO_{\ell_0(\pi)}}(\pi)$ of $\pi$ is square-integrable.
\end{enumerate}
\end{cor}

Finally we prove the following proposition, which proves Proposition \ref{pp:sl} and finishes the proof of Proposition \ref{th:itq}.

Let $\pi\in\Pi(G_n)$ of a generic $L$-parameter. Assume that the character $\psi_{\CO_{\ell_0}}$ is given by $\psi_{\ell_0,w_0}$ (defined in \eqref{eq:psi-ell}).
Then the quadratic space $W_{\ell_0}$ defining $G^{\CO_{\ell_0}}$ is of form \eqref{eq:w}.
For each $\ell\leq \ell_0$, we may choose the {\it compatible} $F$-rational orbit $\CO_\ell$ such that its corresponding character is $\psi_{\ell,w_0}$.
Then its stabilizer $G^{\CO_{\ell}}=\SO(W_\ell,q)$, where
 $W_\ell=X^{+}_{\ell_0-\ell}\oplus W_{\ell_0}\oplus X^-_{\ell_0-\ell}$
 and $X^{\pm}_{\ell_0-\ell}=\Span\{e^{\pm}_{\ell+1}, e^{\pm}_{\ell+2},\dots, e^{\pm}_{\ell_0}\}$.
\begin{prop}[Stability of Local Descent]\label{pp:sld}
For any $\pi\in\Pi(G_n)$ with a generic $L$-parameter,
then the $\ell$-th local descent $\CD_{\CO_{\ell}}(\pi)$ is non-zero for $\ell\leq \ell_0$ and the above compatible $\CO_{\ell}$.
\end{prop}


\begin{proof}
By Corollary \ref{pp:gtq}, there exists an $F$-rational orbit $\CO_{\ell_0}$ associated to the partition $\udl{p}_{\ell_0}$ such that
the twisted Jacquet module $\CJ_{\CO_{\ell_0}}(\pi)$ has an irreducible quotient $\sig$ belonging to $\Pi_\temp(G_n^{\CO_{\ell_0}})$ and hence
we have that $m(\pi,\sig)\ne 0$.

For any occurrence index $\ell< \ell_0(\pi)$, we want to show that the twisted Jacquet module
$\CJ_{\CO_\ell}(\pi)$ is non-zero for the $F$-rational orbit $\CO_\ell$, which is compatible with $\CO_{\ell_0}$.
Note that the quadratic spaces defining both orthogonal groups $G^{\CO_\ell}_{n}$ and $G^{\CO_{\ell_0}}_{n}$  have the same anisotropic kernel if $\CO_\ell$ is compatible with $\CO_{\ell_0}$.
Let $\tau$ be a unitary non-self-dual irreducible supercuspidal representation of $\GL_{\ell_0-\ell}(F)$.
Define
$$
\sig_1=\Ind^{G^{\CO_\ell}_{n}(F)}_{P_{\ell_0-\ell}(F)}\tau\otimes \sig,
$$
where $P_{\ell_0-\ell}$ is a parabolic subgroup of $G^{\CO_\ell}_{n}$, whose Levi subgroup is isomorphic to $\GL_{\ell_0-\ell}\times G^{\CO_{\ell_0}}_{n}$.
It is clear that $\sig_1$ is an irreducible tempered representation of $G^{\CO_\ell}_{n}(F)$.
By applying Proposition \ref{pro:MW12}, we obtain an identity:
$
m(\pi, \sig^\vee_1)=m(\pi,\sig)\ne 0.
$
Hence the $\ell$-th twisted Jacquet module $\CJ_{\CO_\ell}(\pi)$ has a quotient $\sig^\vee_1$, as representations of $G_n^{\CO_\ell}(F)$.
In particular, the $\ell$-local descent $\CD_{\CO_\ell}(\pi)$ is nonzero, and has the irreducible tempered representation $\sig_1^\vee$ as
a quotient. This finishes the proof.
\end{proof}

\section{Descent of Local $L$-parameters}\label{sec-DLP}

We introduce a notion of the {\sl descent of the local $L$-parameters} and
determine the structure of the descent of local $L$-parameters, which forms one of the technical cores of the proofs of the main results in the paper.
To do so, we have to calculate explicitly the relevant local roots numbers.

Let $\varphi\in{\Phi}_\gen(G_n^*)$, which is a $2n$-dimensional, self-dual local $L$-parameter,
either of orthogonal type or of symplectic type.
For $\chi\in\wh{\CS}_\varphi$, define the {\it $\ell$-th descent} of $(\varphi,\chi)$ for any $\ell\in\{0,1,\cdots,n\}$,
which is denoted by $\FD_\ell(\varphi,\chi)$,
to be the set of generic self-dual local $L$-parameters $\phi$, satisfying the following conditions:
\begin{enumerate}
\item $\phi$ are local $L$-parameters of dimension $2(n-\ell)$ for $G_n^{\CO_\ell}(F)$, which have the type different from that of $\varphi$, and
\item the equation $\chi^\star_\varphi(\varphi,\phi)=\chi$ holds.
\end{enumerate}


\begin{defn}[Descent for $L$-parameters]\label{df:foi-L}
For a parameter $\varphi$ in $\Phi_\gen(G_n)$ and $\chi\in \widehat{\CS}_\varphi$, the first occurrence index $\ell_0:=\ell_0(\varphi,\chi)$ of $(\varphi,\chi)$ is the integer $\ell_0$ in $\{0,1,\dots,n\}$,
such that
$\FD_{\ell_0}(\varphi,\chi)\ne \emptyset$
but for any $\ell\in \{0,1,\dots,n\}$ with $\ell>\ell_0$, $\FD_{\ell}(\varphi,\chi)=\emptyset$.
The $\ell_0$-th descent $\FD_{\ell_0}(\varphi,\chi)$ of $(\varphi,\chi)$ is called
 the first descent of $(\varphi,\chi)$ or simply the descent of $(\varphi,\chi)$.
\end{defn}

Note that $\ell_0(\varphi,\chi)=n$ if and only if $\chi=1$ by convention. As in the local descent on the representation side, this case
will be excluded from the discussion at the first occurrence index, because in this case, the Bessel model becomes the Whittaker model.

Note that by definition $\FD_{\ell}(\varphi,\chi)=\FD_{\ell}(\varphi^c,\chi)$, and hence $\FD_{\ell}(\varphi,\chi)$ is stable under $c$-conjugate.
It follows that $\FD_{\ell}([\varphi]_c,\chi)$ is well-defined, and that $\phi^c\in\FD_{\ell}(\varphi,\chi)$ if and only if
$\phi\in \FD_{\ell}(\varphi,\chi)$.
In fact, the $c$-conjugation preserves the Bessel models, and hence the local descent of representations is stable under $c$-conjugate and
any two $c$-conjugate $L$-parameters have the same local descent.

\subsection{Local root number}

Let $\phi$ be a local $L$-parameter of $G_n$ and $\psi_F$ be a  non-trivial additive character of $F$, as before.
The {\sl local epsilon factor} defined by P. Deligne and J. Tate \cite{Ta77} is given by
\begin{equation}
\varepsilon(s,\phi,\psi_F)=\varepsilon(\frac{1}{2},\phi,\psi_F)q^{a(\phi,\psi)(\frac{1}{2}-s)}
\end{equation}
where $a(\phi,\psi)$ is the Artin conductor of $\phi$ and
$\varepsilon(\frac{1}{2},\phi,\psi_F)$ is the {\sl local root number}.
By \cite[Proposition 15.1]{GR06} and \cite{Ta77}, one also has the following properties of the local epsilon factors:
\begin{itemize}
	\item $\varepsilon(s,\phi,\psi_a)=\det\phi(a)|a|^{-\dim\phi(\frac{1}{2}-s)}\varepsilon(s,\phi,\psi_F)$, where $\psi_a$ is defined by $\psi_a(x)=\psi_F(ax)$;
	\item $\varepsilon(s,\phi\otimes |\cdot|^{s_0},\psi_F)=q_F^{-s_0a(\phi,\psi)}\varepsilon(s,\phi,\psi_F)$;
	\item $\varepsilon(s,\phi,\psi_F)\varepsilon(1-s,\phi^\vee,\psi_F)=\det(\phi)(-1)$; and
	\item $\ve(s,\phi\oplus \phi',\psi_F)=\ve(s,\phi,\psi_F)\ve(s,\phi',\psi_F)$.
\end{itemize}

Denote $\mu_n$ to be the algebraic $n$-dimensional irreducible representation of $\SL_2(\BC)$ in this paper.
Let us decompose $\phi$ as
$$
\phi=\oplus_{n\geq 0}\rho_n\otimes \mu_{n+1},
$$
where $(\rho_n,V_{\rho_n})$ is a semisimple complex representation of $\CW_F$, which may possibly be zero.
By the work of B. Gross and M. Reeder \cite{GR10} for instance,
one has the local epsilon factor
$$
\varepsilon(s,\phi,\psi_F)=\varepsilon(\frac{1}{2},\phi,\psi_F)q^{a(\phi)(\frac{1}{2}-s)},
$$
where
$
a(\phi)=\sum_{n\geq 0}(n+1)a(\rho_n)+\sum_{n\geq 1}n\cdot \dim V^{\CI}_{\rho_n},
$
and
\begin{equation}\label{eq:GR}
\varepsilon(\frac{1}{2},\phi,\psi_F)
=\prod_{n\geq 0}\varepsilon(\frac{1}{2},\rho_n,\psi_F)^{n+1}
\prod_{n\geq 1}\det(-\rho_{n}(\mathrm{Fr})\vert V^{\CI}_{\rho_n})^{n},	
\end{equation}
with $\CI=\CI_F$ being the inertia group and $V_{\rho_n}$ being the vector space defining $\rho_{n}$,
and with $a(\rho_n)$ being the Artin conductor of $\rho_n$.
More details on the normalization of $\psi$ and the Haar measure defining the Artin conductor can be found in \cite[Section 2]{GR10}.

By \cite[Propositions 5.1 and 5.2]{GGP}, if $\phi$ is a self-dual $L$-parameter and $\det(\phi)=1$, then $\varepsilon(\frac{1}{2},\phi,\psi_F)$ is independent of the choice of $\psi_F$ and is simply denoted by $\ve(\phi)$.
Moreover, $\ve(\phi)=\pm 1.$


For example, if $\tau$ is a character of $F^\times$, we may rewrite $\tau$ as $|\cdot|^{s_0}\omega$, where $\omega$ is a unitary character of $F^\times$.
Then
$$
\varepsilon(s,\tau,\psi_F)=q^{n(\frac{1}{2}-s-s_0)}
\frac{g(\bar{\omega},\psi_{\varpi^{-n}})}{|g(\bar{\omega},\psi_{\varpi^{-n}})|},
$$
where $n$ is the conductor of $\omega$ (i.e., $\omega$ is trivial on $1+\varpi^n\Fo_F$ but not trivial on $1+\varpi^{n-1}\Fo_F$)
and the Gauss sum is defined by
$$
g(\omega,\psi_a)=\int_{\Fo^\times_F}\omega(u)\psi_F(au)\ud u.
$$
In particular, if $\tau$ is unramified, then $\varepsilon(s,\tau,\psi_F)=1$.
If $\tau$ is a ramified quadratic character, then
 it is of conductor 1.
Thus $\varepsilon(s,\tau,\psi_F)=\varepsilon(\frac{1}{2},\tau,\psi_F)q^{\frac{1}{2}-s}$
and $\varepsilon^2(\frac{1}{2},\tau,\psi_F)=\tau(-1)$.

More generally, let $\pi=[\tau|\cdot|^{\frac{1-r}{2}},\tau|\cdot|^{\frac{r-1}{2}}]$ be the square-integrable representation of
a general linear group determined by
the line segment of Bernstein and Zelevinsky in \cite{BZ}, with the local $L$-parameter $\varphi_{\tau}\boxtimes \mu_r$.
Here $\varphi_\tau$ is the associated irreducible representation of $\CW_F$ via the local Langlands correspondence for general linear groups.
Denote by $\omega_\pi$ the central character of $\pi$.
If $\tau$ is a quadratic character of $\GL_1(F)$ and $\omega_\pi=1$ (that is, if $\tau$ is non-trivial, then $r$ is even),
then
\begin{equation}\label{eq:eps-tensor-character}
\varepsilon(s,\varphi_{\tau}\boxtimes \mu_r,\psi_F)
=\begin{cases}
	q^{(r-1)(\frac{1}{2}-s)} &\text{ if $\tau=1$ and $r$ is odd},\\
	-\tau(\varpi)q^{(r-1)(\frac{1}{2}-s)} &\text{ if $\tau$ is unramified and $r$ is even}, \\
	\tau(-1)^{\frac{r}{2}}q^{r(\frac{1}{2}-s)}& \text{ if $\tau$ is ramified.}
\end{cases}	
\end{equation}
Remark that under the assumption $\varphi_\tau\boxtimes \mu_r$ is self-dual and $\det(\varphi_\tau\boxtimes \mu_r)=1$.
If $\tau$ is a self-dual supercuspidal representation of $\GL_a(F)$ with $a>1$,
then
\begin{equation}\label{eq:eps-tensor-scsp}
\varepsilon(s,\varphi_{\tau}\boxtimes \mu_r,\psi_F)
=\varepsilon(\frac{1}{2},\tau,\psi_F)^{r}q^{r a(\tau)(\frac{1}{2}-s)}.
\end{equation}


\begin{lem} \label{lm:eps-tensor}
Let $\varphi_1=\varphi_{\pi}\boxtimes\mu_n$ and $\varphi_2=\varphi_{\tau}\boxtimes\mu_m$ be two irreducible local $L$-parameters, with
$\varphi_{\pi}$ and $\varphi_{\tau}$ being self-dual and of dimensions $a$ and $b$, respectively.
\begin{enumerate}
\item If $m$ and $n$ are even, then $\varepsilon(\varphi_1\otimes\varphi_2)=1$.
\item \label{lm:eps-symplectic-odd} If $\varphi_\pi\otimes\varphi_\tau$ is of symplectic type, then
$$
\varepsilon(\varphi_1\otimes\varphi_2)=\begin{cases}
	1, & \text{ when  $m+n$ is odd,}\\
	\varepsilon(\pi\times\tau), & \text{ when  $mn$ is odd.}
\end{cases}
$$
\item \label{lm:eps-ncong} If $\varphi_\pi\otimes\varphi_\tau$ is of orthogonal type
and $\pi\ncong  \tau$, then
$$
\varepsilon(\varphi_1\otimes\varphi_2)=
	\det(\varphi_\pi)(-1)^{\frac{bmn}{2}} \det(\varphi_\tau)(-1)^{\frac{amn}{2}},
$$
when $m+n$ is odd, and
$
\varepsilon(\varphi_1\otimes\varphi_2)=\varepsilon(\pi\times\tau,\psi_F),
$
when $mn$ is odd.
\item \label{lm:eps-cong} If $\varphi_\pi\otimes\varphi_\tau$ is of orthogonal type and $\pi\cong  \tau$, then
\begin{eqnarray*}
\varepsilon(\varphi_1\otimes\varphi_2)
&=&
\det(\varphi_\pi)(-1)^{\frac{bmn}{2}} \det(\varphi_\tau)(-1)^{\frac{amn}{2}}(-1)^{\min\{m,n\}}\\
&=&(-1)^{\min\{m,n\}},
\end{eqnarray*}
when $m+n$ is odd, and
$
\varepsilon(\varphi_1\otimes\varphi_2)=
\varepsilon(\pi\times\tau,\psi_F),
$
when $mn$ is odd.
\end{enumerate}
\end{lem}



\begin{rmk}\label{rk:dependence-psi}
If $mn$ is odd, and $\varphi_\pi$ and $\varphi_\tau$ are of orthogonal type, then $\det(\varphi_\pi\otimes\varphi_\tau)=\pm 1$
and $\varepsilon(\varphi_\pi\otimes\varphi_\tau)$ depends possibly on the additive character $\psi_F$.
Thus, we add $\psi_F$ in the $\varepsilon$ for  this case.
In this case the local root number is calculated in \cite[Theorem 6.2 (2)]{GGP}.
However, after the normalization in \eqref{eq:CE}, the local root number is  independent on $\psi_F$ (see \cite[Theorem 6.2 (1)]{GGP}).
\end{rmk}

\begin{proof}
Since
$\mu_n\otimes\mu_m=\bigoplus^{\min\{m,n\}}_{i=1}\mu_{n+m+1-2i}$, as representations of $\SL_2(\BC)$, one has
$$
\varphi_1\otimes\varphi_2=\bigoplus^{\min\{m,n\}}_{i=1}(\varphi_\pi\otimes\varphi_\tau) \boxtimes \mu_{n+m+1-2i}.
$$
If $\pi\ncong \chi\tau$ for any unramified character $\chi$,
there is no $\varphi_{\pi}\otimes\varphi_{\tau}(\CI)$-invariant vector.
Following \eqref{eq:eps-tensor-scsp}, one has
\begin{equation}\label{eq:eps-tensor-decomp-equal}
\varepsilon(\varphi_1\otimes\varphi_2)=\prod^{\min\{m,n\}}_{i=1}\varepsilon(\pi\times\tau,\psi_F)^{n+m+1-2i}
=\varepsilon(\pi\times\tau,\psi_F)^{mn}.
\end{equation}
Suppose that $\pi\cong \chi\tau$ for some unramified character $\chi$.
Following Equation \eqref{eq:GR} and (6.2.5) in \cite{BK93}, one has
$$
\varepsilon((\varphi_\pi\otimes\varphi_\tau)\boxtimes\mu_r)=\varepsilon(\pi\times\tau,\psi_F)^r(-\chi(\varpi)^{d(\pi)})^{r-1},
$$
where $d(\pi)$ is  the number of all unramified characters $\chi$ such that $\pi\cong \chi\pi$.
And an unramified character $\chi'$ satisfies $\pi\cong \chi'\pi$ if and only if the order of $\chi'$ divides $d(\pi)$.
It follows that
\begin{align}
\varepsilon(\varphi_1\otimes\varphi_2)=&\prod^{\min\{m,n\}}_{i=1}\varepsilon(\pi\times\tau,\psi_F)^{n+m+1-2i}(-\chi(\varpi)^{d(\pi)})^{n+m-2i} \nonumber\\
=&\varepsilon(\pi\times\tau,\psi_F)^{mn}(-\chi(\varpi)^{d(\pi)})^{mn-\min\{m,n\}}. \label{eq:eps-tensor-decomp-neq}
\end{align}
Since $\pi$ and $\tau$ are self-dual, the unramified character $\chi$ has the property that $\pi\cong \chi^2\pi$.
Hence, the order of $\chi^2$ divides $d(\pi)$, equivalently $\chi(\varpi)^{d(\pi)}=\pm 1$.
If $\chi(\varpi)^{d(\pi)}=-1$, then
the order of $\chi$  equals $2d(\pi)$, which implies $\pi\ncong \chi\pi$ and then $\pi\ncong\tau$.
In this case,
$\varepsilon(\varphi_1\otimes\varphi_2)=\varepsilon(\pi\times\tau,\psi_F)^{mn}$.


If  $\chi(\varpi)^{d(\pi)}=1$, then the order $\chi$ divides of $d(\pi)$,
which implies $\pi\cong\chi\tau\cong\tau$ and
\begin{equation}\label{eq:lm:d-even}
\varepsilon(\varphi_1\otimes\varphi_2)=\varepsilon(\pi\times\tau,\psi_F)^{mn}(-1)^{mn-\min\{m,n\}}.	
\end{equation}
By \eqref{eq:eps-tensor-decomp-neq},
one obtains Equation \eqref{eq:lm:d-even}.


Combining with the case $\pi\ncong\chi\tau$ for any unramified  $\chi$,
we summarize
$$
\varepsilon(\varphi_1\otimes\varphi_2)=\begin{cases}
	\varepsilon(\pi\times\tau,\psi_F)^{mn}  &\text{ if }\pi\ncong \tau\\
	\varepsilon(\pi\times\tau,\psi_F)^{mn}(-1)^{mn-\min\{m,n\}}  &\text{ if }\pi\cong \tau.
\end{cases}
$$
Recall that $\varepsilon(\varphi)^2=(\det\varphi)(-1)=\pm 1$
if $\varphi$ is self-dual.
Since $\varphi_\pi\otimes\varphi_\tau$ is self-dual, we have
\begin{equation}\label{eq:eps-square-det}
\varepsilon(\varphi_\pi\otimes\varphi_\tau)^2=\det(\varphi_\pi)(-1)^b\det(\varphi_\tau)(-1)^a.
\end{equation}
If $m$ and $n$ are even, then $4$ divides $mn$ and $\min\{m,n\}$ is even.
By \eqref{eq:eps-tensor-decomp-equal}, \eqref{eq:eps-tensor-decomp-neq} and \eqref{eq:eps-square-det},
we have $\varepsilon(\varphi_1\otimes\varphi_2)=1$.

Suppose that $\varphi_{\pi}\otimes\varphi_{\tau}$ is of symplectic type.
Then $\det(\varphi_{\pi}\otimes\varphi_{\tau})=1$  and
$\varphi_\pi$ and $\varphi_\tau$ are of different type.
If $m+n$ is odd, then $mn$ is even and
$\varepsilon(\pi\times\tau,\psi_F)^{mn}=1$.
As $\varphi_\pi$ and $\varphi_\tau$ are irreducible and of different type,
there is no unramified character $\chi$ satisfying $\pi\cong \chi\tau$.
Therefore, $\varepsilon(\varphi_1\otimes\varphi_2)=1$ if $m+n$ is odd, and
$\varepsilon(\varphi_1\otimes\varphi_2)=\varepsilon(\pi\times\tau)$ if $mn$ is odd.

Suppose that $\varphi_{\pi}\otimes\varphi_{\tau}$ is of orthogonal type.
If $\pi\ncong  \tau$, then, when $mn$ is even,
$$
\varepsilon(\varphi_1\otimes\varphi_2)=(\varepsilon(\pi\times\tau,\psi_F)^2)^{\frac{mn}{2}}=\det(\varphi_\pi)(-1)^{\frac{bmn}{2}}\det(\varphi_\tau)(-1)^{\frac{amn}{2}},
$$
which is independent of the choice of $\psi_F$; and
when $mn$ is odd,
$
\varepsilon(\varphi_1\otimes\varphi_2)=\varepsilon(\pi\times\tau,\psi_F),
$
which depends on the choice of $\psi_F$.

If $\pi\cong  \tau$, then, when $mn$ is even, $\varepsilon(\varphi_1\otimes\varphi_2)$ equals
$$
\det(\varphi_\pi)(-1)^{\frac{bmn}{2}}\det(\varphi_\tau)(-1)^{\frac{amn}{2}}(-1)^{\min\{m,n\}}=(-1)^{\min\{m,n\}};
$$
and when $mn$ is odd,
$
\varepsilon(\varphi_1\otimes\varphi_2)=\varepsilon(\pi\times\tau,\psi_F)
$
as $mn-\min\{m,n\}$ is even.

\end{proof}

In some special cases, the result is simple. The following example is about the quadratic unipotent $L$-parameters, which will be more
explicitly discussed in Section \ref{sec:ex}.

\begin{exmp}\label{ex:eps-tensor-char}
Let $\varphi_\pi=\chi\boxtimes\mu_n$ and $\varphi_\tau=\xi\boxtimes\mu_m$,
where $\chi$ and $\xi$ are quadratic characters.
If $m+n$ is odd, then
$$
\varepsilon(\pi\times\tau)=\begin{cases}
	(-\chi\xi(\varpi))^{\min\{m,n\}}, &\text{ if $\chi\xi$ is unramified,}\\
	\chi\xi(-1)^{\frac{mn}{2}}, & \text{ if $\chi\xi$ is ramified.}
\end{cases}
$$
\end{exmp}

\begin{rmk}\label{lm:chi-star-odd}
Let $\varphi$ and $\phi$ be discrete $L$-parameters of different type.
Then $\chi^\star(\varphi,\phi)=(\chi^\star_\varphi,\chi^\star_\phi)$, with $\chi^\star_\varphi$ and $\chi^\star_\phi$ as defined in \eqref{eq:dist-eps-odd-W} and \eqref{eq:dist-eps-even-W}, yields a pair of characters on $A_\varphi$ and $A_\phi$, which are independent of $\psi_F$.
In fact, one may decompose $\varphi$ and $\phi$ as
$\varphi=\boxplus^r_{i=1}\varphi_i$ and $\phi=\boxplus^s_{j=1}\phi_j$,
where $\varphi_i$ and $\phi_j$ are irreducible and of different type for all $i,j$.
Since  $\varphi_i\otimes \phi_j$ is of symplectic type, $\CE(\varphi_i,\phi)$ and $\CE(\varphi,\phi_j)$ are $\pm 1$ and independent of $\psi_F$.
\end{rmk}


In the remark, we extend the character $\chi^\star(\varphi,\phi)$ to be a character of $A_\varphi\times A_\phi$, which is well-defined and independent of $\psi_F$.
And it is allowed that the orthogonal parameter $\varphi$ or $\phi$ is of odd dimension.
Then the character $\chi^\star(\varphi,\phi)$ is still well-defined.

\subsection{Descent of local $L$-parameters}\label{sec:descent-L-parameters}
The main result in this subsection is to explicitly determine the descent of local $L$-parameters, which is Theorem \ref{pp:dsqlp}.



Let $\varphi$ be a local $L$-parameter of $G_n$ for a square-integrable representation of $G_n(F)$.
It can be decomposed as
\begin{equation}\label{eq:dsphi}
\varphi=\boxplus^{r}_{i=1}\boxplus^{r_i}_{j=1}\rho_i\boxtimes\mu_{2\alpha_{i,j}}\boxplus
\boxplus^{s}_{i=1}\boxplus^{s_i}_{j=1}\varrho_i\boxtimes\mu_{2\beta_{i,j}+1}	
\end{equation}
where all $\rho_i$ and $\varrho_i$ are irreducible self-dual distinct  representations of $\CW_F$ of dimension $a_i$ and $b_i$, respectively,
and
$1\leq \alpha_{i,1}< \alpha_{i,2}<\dots<\alpha_{i,r_i}$
and
$0\leq \beta_{i,1}< \beta_{i,2}<\dots<\beta_{i,s_i}$ for all $i$ are integers.
For such a local $L$-parameter $\varphi$ as in \eqref{eq:dsphi}, we define the even and odd parts according to the dimension of $\mu_k$'s:
\begin{eqnarray}
\varphi_e&:=&\boxplus^{r}_{i=1}\boxplus^{r_i}_{j=1}\rho_i\boxtimes\mu_{2\alpha_{i,j}}
\label{eq:decomp-e}\\
\varphi_o&:=&\boxplus^{s}_{i=1}\boxplus^{s_i}_{j=1}\varrho_i\boxtimes\mu_{2\beta_{i,j}+1}.
\label{eq:decomp-o}
\end{eqnarray}
For an irreducible representation $\rho$ of $\CW_F$, denote by $\varphi(\rho)$ the $\rho$-isotypic component of $\varphi$ when restricted
to $\CW_F$. It is clear that $\varphi(\rho)$ is still a local $L$-parameter.
For instance,
$
\varphi(\rho_i)=\boxplus^{r_i}_{j=1}\rho_i\boxtimes\mu_{2\alpha_{i,j}}.
$
Hence $\varphi(\rho)=0$ if $\rho$ is not isomorphic to any of the $\rho_i$ or $\varrho_i$ for all $i$.


For a local $L$-parameter $\varphi$ of $G_n$ for a square-integrable representation of $G_n(F)$, we decompose it as $\varphi=\boxplus_{i\in \RI} \varphi_i$.
Let $\varphi'$ be a subrepresentation of $\varphi$, i.e., $\varphi'=\boxplus_{i\in \RI'} \varphi_i$ where $\RI'\subseteq \RI$.
For an element $(e_i)\in \CS_\varphi$, denote $(e_i)\vert_{\varphi'}$ to be the elements in $\CS_\varphi$ such that $e_i=0$ for $i\in \RI'$.
For example, the element $1_{\varphi'}$ is given by $e_i=1$ for $i\in \RI'$ and $e_i=0$ for $i\notin \RI'$.

Define the {\it sign alternative index set, $\sgn_{\ast,\rho}(\chi)$}, as follows:
when $\rho\cong\rho_i$ for some $i$,
\begin{equation}\label{sgn-o}
\sgn_{o,\rho_i}(\chi):=\{j\colon 0\leq j< r_i,~  \chi(1_{\rho_i\boxtimes \mu_{2\alpha_{i,j}}\boxplus \rho_i\boxtimes \mu_{2\alpha_{i,j+1}}})=-1\};
\end{equation}
and when $\rho\cong\varrho_i$ for some $i$,
\begin{equation}\label{sgn-e}
\sgn_{e,\varrho_i}(\chi)=\{j\colon 1\leq j< s_i,~  \chi(1_{\varrho_i\boxtimes \mu_{2\beta_{i,j}+1}\boxplus \varrho_i\boxtimes \mu_{2\beta_{i,j+1}+1}})=-1\}.
\end{equation}
By convention, $\alpha_{i,0}=0$.
For each $(\varphi,\chi)$, define
\begin{equation}\label{eq:phi-sgn}
\varphi_{\sgn}=\boxplus^{r}_{i=1}\boxplus_{j\in \sgn_{o,\rho_i}(\chi)}\rho_i\boxtimes\mu_{2\alpha_{i,j}+1}\boxplus
\boxplus^{s}_{i=1}\boxplus_{j\in \sgn_{e,\varrho_i}(\chi)}\varrho_i\boxtimes\mu_{2\beta_{i,j}+2}.	
\end{equation}
Note that $\varphi_{\sgn}$ is uniquely determined by the given local Langlands data $(\varphi,\chi)$, but it may possibly be zero.
If $\varphi_\sgn$ is not zero, then $\varphi_\sgn$ and $\varphi$ are of different type.
It is worthwhile to mention that $\varphi_\sgn$ is a common component of each elements in the local descent $\FD_{\ell}(\varphi,\chi)$ (see Definition \ref{df:foi-L}), whose dimension gives a upper bound for the index  of the first occurrence,
i.e., $2\ell_0\leq \dim\varphi-\dim \varphi_{\sgn}$.
According to Section~\ref{sec:clp}, such types of $L$-parameters are closely related to the cuspidal local $L$-parameters. By \cite{M11} and also 
\cite{Xu17}, their $L$-packets contain at least one irreducible supercuspidal representation.

From $\varphi$, we define two new parameters $\lceil\varphi_o\rceil$  and $\varphi^\dagger$ by
\begin{equation}\label{eq:varphi-o-dagger}
\lceil\varphi_o\rceil=\boxplus^{s}_{i=1} \varrho_i\boxtimes\mu_{2\beta_{i,s_i}+1}\ \text{and}\
\varphi^\dagger=\boxplus^s_{i=1}\varrho_i\boxtimes 1,	
\end{equation}
respectively. Note that for each $i$, the piece $\varrho_i\boxtimes\mu_{2\beta_{i,s_i}+1}$ is the one with maximal dimension among the
summands $\varrho_i\boxtimes\mu_{2\beta_{i,j}+1}$'s for $j=1,2,\cdots,s_i$.
Both $\lceil\varphi_o\rceil$ and $\varphi^\dagger$ are of the same type as $\varphi$, and are possibly of orthogonal type and have odd dimension.
Define that  $\chi_{\lceil\varphi_o\rceil}=\chi\vert_{\lceil\varphi_o\rceil}$, the restriction of $\chi$ on the elements $((e_i)\vert_{\lceil\varphi_o\rceil})$, is a character of $A_{\lceil\varphi_o\rceil}$.
Also $\chi_{\lceil\varphi_o\rceil}$ is considered as a character of $\CS_{\lceil\varphi_o\rceil}$ by restriction.

Note that there is an isomorphism from $A_{\lceil\varphi_o\rceil}$ to $A_{\varphi^\dagger}$ given by $(e_i)\in A_{\lceil\varphi_o\rceil}\mapsto (e'_i)\in A_{\varphi^\dagger}$ and $e_i=e'_i$,
where  $e_i$ and $e'_i$ correspond to the component $\varrho_i\boxtimes \mu_{2\beta_{i,s_i}+1}$ in $A_{\lceil\varphi_o\rceil}$ and $\varrho_i\boxtimes 1$ in $A_{\varphi^\dagger}$, respectively.
Hence we have that $\CS_{\lceil\varphi_o\rceil}\cong \CS_{\varphi^\dagger}$.


For a local generic $L$-parameter $\varphi$ with the decomposition \eqref{edec}, define its discrete part  by
\begin{equation}\label{eq:discrete}
\varphi_{\Box}=\boxplus_{i\in \RI_\gp}\varphi_i,
\end{equation}
which is a discrete $L$-parameter.

\begin{thm}\label{pp:dsqlp}
Let $\varphi$ be a generic $L$-parameter of even dimension with the discrete part $\varphi_\Box$ of the decomposition \eqref{eq:dsphi}.
Then for $\chi\in\widehat{\CS}_\varphi$, the descent of the parameter $(\varphi,\chi)$ at the first occurrence index $\ell_0$ can be
completely determined by the following
$$
\FD_{\ell_0}(\varphi,\chi)=\cup_{\psi}[ \psi\boxplus\varphi_\sgn]_c,
$$
where $\psi$ runs through over all discrete $L$-parameters satisfying the following conditions with the minimal dimension:
\begin{enumerate}
	\item\label{item:prop-dim} $\dim\psi\equiv \sum^{r}_{i=1}\#\sgn_{o,\rho_i}(\chi)\cdot \dim\rho_i (\bmod{2})$;
	\item\label{item:prop-1} $\psi=(\boxplus^k_{i=1}\delta_i\boxtimes 1) \boxplus(\boxplus^r_{i=1}m_i\rho_i\boxtimes\mu_{2\alpha_{i,r_i}+1})\boxplus(\boxplus^s_{j=1}n_j\varrho_j\boxtimes\mu_{2\beta_{j,s_j}+2})$ with the multiplicity $m_i,n_j\in\{0,1\}$, which are multiplicity-free;
	\item \label{item:prop-2} all $\delta_i$ and $\rho_i$ are of the same type as $\psi$ and
	$$\{\delta_i\colon 1\leq i\leq k\}\cap \{\rho_i\colon 1\leq i\leq r\}=\emptyset;$$
	\item \label{item:prop-3}
	$\chi_{\lceil\varphi_o\rceil}\vert_{\CS_{\lceil\varphi_o\rceil}}=\chi^\star_{\varphi^\dagger}\vert_{\CS_{\varphi^\dagger}}$ via  $\CS_{\lceil\varphi_o\rceil}\cong \CS_{\varphi^\dagger}$, where $\chi^\star_{\varphi^\dagger}$ is the character in $\chi^\star(\varphi^\dagger,\psi)$.
\end{enumerate}
\end{thm}
Note that $\varphi_\sgn$, $\varphi^\dagger$ and $\lceil\varphi_o\rceil$ are associated to $\varphi_\Box$.
Recall that $[\phi]_c=\{\phi,\phi^c\}$ is the $c$-conjugacy class of $\phi$,
and $\chi^\star(\varphi^\dagger,\psi)$ is well-defined even though $\varphi^\dagger$ or $\psi$ is of orthogonal type of odd dimension,
following from Remark \ref{lm:chi-star-odd}.
In addition, $\psi=0$ is allowed.
We will see some examples in Section \ref{sec:ex}.

Now, let us sketch the proof of Theorem \ref{pp:dsqlp}.
First,  we only need to calculate the character $\chi^\star(\varphi,\phi)$ defined in \eqref{eq:dist-eps-odd-W} and \eqref{eq:dist-eps-even-W} for general discrete parameters $\varphi$ and $\phi$ by Corollary \ref{pp:gtq}.
Following from Lemma \ref{lm:eps-tensor} and their decompositions of form \eqref{eq:dsphi},
we may reduce $\chi^\star(\varphi,\phi)$ to the symplectic root numbers of type $\CE(\varrho_i,\varrho'_\ell)$,
where $\varrho_i$ and $\varrho'_\ell$ are irreducible self-dual distinct representations of  $\CW_F$ and are of different type.
The formula is stated in Lemma  \ref{lm:key-lem} in below.

Next, at the first occurrence $\ell_0$, we have the property that the descent of the parameter $(\varphi,\chi)$ has the minimal dimension such that $\FD_{\ell_0}(\varphi,\chi)$ is not empty.
Following this property, we apply Lemma \ref{lm:key-lem} repeatedly on the descent parameters in $\FD_{\ell_0}(\varphi,\chi)$ and
then we give a refined description on those descent parameters, which are Theorem \ref{pp:dsqlp}.
Its proof will be given in Section \ref{sec:proof}.

The rest of Section \ref{sec:descent-L-parameters} is devoted to calculate the character $\chi^\star_\varphi(\varphi,\phi)$ for two  discrete parameters $\varphi$ and $\phi$. Due to symmetry, the formula for $\chi^\star_\phi(\varphi,\phi)$ is similar.
In order to state Lemma \ref{lm:key-lem}, we introduce more notation first.

For an $L$-parameter $\varphi$ with the decomposition \eqref{eq:dsphi},
define
\begin{equation}\label{diamond}
\varphi^{\diamond m}_*(\rho):=\begin{cases}
	\boxplus_{\{j\colon 1\leq j\leq r_i,~\alpha_{i,j}\diamond m\}}\rho_i\boxtimes\mu_{2\alpha_{i,j}}, &\text{ if $\rho\cong\rho_i$ for some $i$}\\
	\boxplus_{\{j\colon 1\leq j\leq s_i,~\beta_{i,j}\diamond m\}}\varrho_i\boxtimes\mu_{2\beta_{i,j}+1}, &\text{ if $\rho\cong\varrho_i$ for some $i$}\\
	0, & \text{ otherwise,}
\end{cases}
\end{equation}
where $\diamond\in\{>,<,\geq,\leq\}$;
$*=e$ if $\varphi$ and $\rho$ are of the different type, and $*=o$, otherwise.
Also define
\begin{equation}\label{minus}
\varphi^-_*(\rho)=\varphi_*\boxminus\varphi_*(\rho),
\end{equation}
which is the remaining part of the parameter $\varphi_*$ without the $\rho$-isotypic component $\varphi_*(\rho)$.
For a subrepresentation $\varphi'$ of $\varphi$,
write $\#\varphi'$ to be the number of the irreducible summands in $\varphi'$.

We decompose $\phi$  as in \eqref{eq:dsphi}, i.e.,
\begin{equation}\label{eq:dsphi-varphi}
\phi=\boxplus^{r'}_{i=1}\boxplus^{r'_i}_{j=1}\rho'_i\boxtimes\mu_{2\alpha'_{i,j}}\boxplus
\boxplus^{s'}_{i=1}\boxplus^{s'_i}_{j=1}\varrho'_i\boxtimes\mu_{2\beta'_{i,j}+1},
\end{equation}
where $\rho'_i$ (resp. $\varrho'_i$) is of dimension $a'_i$ (resp. $b'_i$).

\begin{lem}\label{lm:key-lem}
Let $\varphi$ and $\phi$ be discrete $L$-parameters decomposed as in \eqref{eq:dsphi} and \eqref{eq:dsphi-varphi} respectively, and be  of different type.
Then as the character of $\CS_\varphi$,
\begin{align}
\chi^\star_\varphi((e_{i,j})\vert_{\varphi_e})=&
\prod^{r}_{i=1}\prod^{r_{i}}_{j=1}
(\det(\rho_i)(-1)^{\alpha_{i,j}\dim\phi}
(-1)^{\#\phi^{<\alpha_{i,j}}_o(\rho_i)})^{e_{i,j}}\label{eq:lm:chi-star-e}\\
\chi^\star_\varphi((e_{i,j})\vert_{\varphi_o})=&\prod^{s}_{i=1}
\prod^{s_{i}}_{j=1}
((\prod^{s'}_{l=1}
\CE(\varrho_{i},\varrho'_{l})^{s'_l})
(-1)^{\#\phi^{>\beta_{i,j}}_e(\varrho_i)})^{e_{i,j}},\label{eq:lm:chi-star-o}
\end{align}
where $\CE(\cdot,\cdot)$ is defined in \eqref{eq:CE}.
\end{lem}

It is clear that $A_\varphi\cong A_{\varphi_e}\times A_{\varphi_o}$ and
 $\CS_\varphi\cong A_{\varphi_e}\times \CS_{\varphi_o}$.
The isomorphism is given by $(e_i)\mapsto ((e_i)\vert_{\varphi_e},(e_i)\vert_{\varphi_o})$.
Equations \eqref{eq:lm:chi-star-e} and \eqref{eq:lm:chi-star-o} give a formula for all values of $\chi^\star_\varphi$ on $\CS_\varphi$.
Over $A_\varphi$, the formula \eqref{eq:lm:chi-star-o} will be slightly different.
And note that $\rho_i$ in \eqref{eq:dsphi} and $\rho'_i$ in \eqref{eq:dsphi-varphi} are of different types, so are $\varrho_i$ and $\varrho'_i$.

\begin{proof}
To prove this lemma, we evaluate $\chi_\varphi^\star$ at three types of elements in $\CS_\varphi$: {\bf Type (1):} $1_{\rho\boxtimes\mu_{2\alpha}}$;
{\bf Type (2):} $1_{\varrho\boxtimes\mu_{2\beta+1}}$ where $\dim\varrho$ is even; and
{\bf Type (3):} $1_{\varrho\boxtimes\mu_{2\beta+1}\boxplus\varrho_0\boxtimes \mu_{2\beta_0+1}}$ where $\dim\varrho$ and $\dim \varrho_0$ are odd.
For each type, the evaluation is reduced to the calculation of the symplectic roots of form $\varepsilon((\varsigma\boxtimes\mu_\kappa)\otimes\phi)$, where $\varsigma\boxtimes\mu_\kappa$ is the summand occurring in the subscript of above types.
From \eqref{eq:decomp-e}, \eqref{eq:decomp-o}, and \eqref{minus}, we may write that
$\phi=\phi_e\boxplus\phi_o=\phi_e\boxplus\phi_o(\varsigma)\boxplus\phi_o^-(\varsigma)$.
Then
\begin{equation}\label{eq:epsilon-decomp}
\varepsilon((\varsigma\boxtimes\mu_\kappa)\otimes\phi)=\varepsilon((\varsigma\boxtimes\mu_{\kappa})\otimes\phi_e)
\cdot \varepsilon((\varsigma\boxtimes\mu_{\kappa})\otimes\phi_o(\varsigma))
\cdot \varepsilon((\varsigma\boxtimes\mu_{\kappa})\otimes\phi^-_o(\varsigma)).
\end{equation}
Finally, we apply Lemma \ref{lm:eps-tensor} to calculate each factor on the right hand side of \eqref{eq:epsilon-decomp}.

{\bf Type (1).}
We verify Equation \eqref{eq:lm:chi-star-e}.
Consider a  summand of form $\rho\boxtimes\mu_{2\alpha}$ in $\varphi$.
Write $a=\dim\rho$.
Since the dimension of $\rho\boxtimes\mu_{2\alpha}$ is even,
$\apair{(e_i)\vert_{\rho\boxtimes\mu_{2\alpha}}}\cong \BZ_2$ is a subgroup of $\CS_\varphi$
 in both symplectic and  orthogonal types.
For all types, it is sufficient to show
\begin{equation}\label{eq:prop:even-summand}
\chi^\star_{\varphi}(1_{\rho\boxtimes\mu_{2\alpha}})=\det(\rho)(-1)^{\alpha\dim\phi}(-1)^{\#\phi^{<\alpha}_o(\rho)}.	
\end{equation}
Here $1_{\rho\boxtimes\mu_{2\alpha}}$ corresponds to the non-trivial element in $\apair{(e_i)\vert_{\rho\boxtimes\mu_{2\alpha}}}$.

By definition, if   $\varphi$ is of symplectic type,
then
\begin{equation}\label{eq:prop:even-B}
\chi^\star_{\varphi}(1_{\rho\boxtimes\mu_{2\alpha}}) = \det(\phi)(-1)^{a\alpha}
\ve((\rho\boxtimes\mu_{2\alpha})\otimes\phi);	
\end{equation}
and if $\varphi$ is of orthogonal type, then
$
\chi^\star_{\varphi}(1_{\rho\boxtimes\mu_{2\alpha}}) = \ve((\rho\boxtimes\mu_{2\alpha})\otimes\phi),
$
as $\det(\rho)=1$.
Similar to \eqref{eq:epsilon-decomp}, we have $\ve((\rho\boxtimes\mu_{2\alpha})\otimes\phi)$ equal to
$$
\varepsilon((\rho\boxtimes\mu_{2\alpha})\otimes\phi_e)
\cdot \varepsilon((\rho\boxtimes\mu_{2\alpha})\otimes\phi_o(\rho))
\cdot \varepsilon((\rho\boxtimes\mu_{2\alpha})\otimes\phi^-_o(\rho)).
$$
As discussed in Remark \ref{rk:dependence-psi}, the local root number in this case is independent of the choice of the additive character $\psi_F$.
We omit $\psi_F$ in the calculation.
By Lemma \ref{lm:eps-tensor}, $\varepsilon((\rho\boxtimes\mu_{2\alpha})\otimes(\rho'_i\boxtimes\mu_{2\alpha'_{i,j}}))=1$.
It follows that $\varepsilon((\rho\boxtimes\mu_{2\alpha})\otimes\phi_e)=1$.

Next, we calculate $\varepsilon((\rho\boxtimes\mu_{2\alpha})\otimes\phi_o(\rho))$.
By definition in \eqref{diamond}, we may write
$
\phi_o(\rho)=\phi^{<\alpha}_o(\rho)\boxplus\phi^{\geq\alpha}_o(\rho).
$
It follows that
$$
\varepsilon((\rho\boxtimes\mu_{2\alpha})\otimes\phi_o(\rho))
=
\varepsilon((\rho\boxtimes\mu_{2\alpha})\otimes\phi^{<\alpha}_o(\rho))
\cdot\varepsilon((\rho\boxtimes\mu_{2\alpha})\otimes\phi^{\geq\alpha}_o(\rho)).
$$

Since $\phi^{<\alpha}_o(\rho)=\boxplus_{\{1\leq j\leq s'_{i_0}\colon \beta'_{i_{0},j}<\alpha\}}\varrho'_{i_0}\boxtimes\mu_{2\beta'_{i_0,j}+1}$,
we have, by Part \eqref{lm:eps-cong} of Lemma \ref{lm:eps-tensor}, that
\begin{align}
 & \varepsilon((\rho\boxtimes\mu_{2\alpha})\otimes\phi^{<\alpha}_o(\rho))\nonumber\\
=& \prod_{\{1\leq j\leq s'_{i_0}\colon \beta'_{i_{0},j}<\alpha\}}\det(\rho)(-1)^{b'_{i_0}\alpha(2\beta'_{i_0,j}+1)}
\det(\varrho'_{i_0})(-1)^{a\alpha(2\beta'_{i_0,j}+1)}(-1)^{2\beta'_{i_0,j}+1}
\nonumber\\
=&(-1)^{\#\phi^{<\alpha}_o(\rho)}\times \prod_{\{1\leq j\leq s'_{i_0}\colon \beta'_{i_{0},j}<\alpha\}}\det(\rho)(-1)^{b'_{i_0}\alpha}
\det(\varrho'_{i_0})(-1)^{a\alpha},
\label{eq:lm:even-<}
\end{align}
where $i_0$ is the index of $\varrho'_{i_0}$ with $\varrho'_{i_0}=\rho$ and
$\#\phi^{<\alpha}_o(\rho)$ is the number of irreducibles in $\phi^{<\alpha}_o(\rho)$, i.e., the cardinality of $\{1\leq j\leq s'_{i_0}\colon \beta'_{i_0,j}<\alpha_{i,j}\}$.

Since $\phi^{\geq\alpha}_o(\rho)=\boxplus_{\{1\leq j\leq s'_{i_0}\colon \beta'_{i_{0},j}\geq\alpha\}}\varrho'_{i_0}\boxtimes\mu_{2\beta'_{i_0,j}+1}$,
we have, by Part \eqref{lm:eps-cong} of Lemma \ref{lm:eps-tensor} again, that
\begin{align}
 & \varepsilon((\rho\boxtimes\mu_{2\alpha})\otimes\phi^{\geq\alpha}_o(\rho))\nonumber\\
=& \prod_{\{1\leq j\leq s'_{i_0}\colon \beta'_{i_{0},j}\geq\alpha\}}\det(\rho)(-1)^{b'_{i_0}\alpha(2\beta'_{i_0,j}+1)}
\det(\varrho'_{i_0})(-1)^{a\alpha(2\beta'_{i_0,j}+1)}(-1)^{2\alpha}
\nonumber\\
=&\prod_{\{1\leq j\leq s'_{i_0}\colon \beta'_{i_{0},j}\geq\alpha\}}\det(\rho)(-1)^{b'_{i_0}\alpha }
\det(\varrho'_{i_0})(-1)^{a\alpha}. \label{eq:lm:even->}
\end{align}

On the other hand, because $\phi^-_o(\rho)=\boxplus^{s'}_{\substack{i=1\\ i\ne i_0}}\boxplus^{s'_i}_{j=1}\varrho'_{i}\boxtimes\mu_{2\beta'_{i,j}+1}$, we have, by Part \eqref{lm:eps-ncong} of Lemma \ref{lm:eps-tensor}, that
\begin{equation}\label{eq:lm:even-c}
\varepsilon((\rho\boxtimes\mu_{2\alpha})\otimes\phi^-_o(\rho))= \prod^{s'}_{\substack{i=1\\ i\ne i_0}}\prod^{s'_i}_{j=1}\det(\rho)(-1)^{b'_i\alpha(2\beta'_{i,j}+1)}
\det(\varrho'_i)(-1)^{a\alpha(2\beta'_{i,j}+1)}.	
\end{equation}

Finally, by taking the product of \eqref{eq:lm:even-<}, \eqref{eq:lm:even->}, and \eqref{eq:lm:even-c}, we obtain that
\begin{align*}
&
\varepsilon((\rho\boxtimes\mu_{2\alpha})\otimes\phi_o(\rho))
\cdot \varepsilon((\rho\boxtimes\mu_{2\alpha})\otimes\phi^-_o(\rho))\\
=&\left(\prod^{s'}_{i=1}\prod^{s'_i}_{j=1}\det(\rho)(-1)^{b'_i\alpha}
\det(\varrho'_i)(-1)^{a\alpha}\right)\times (-1)^{\#\phi^{<\alpha}_o(\rho)}\\
=&\det(\rho)(-1)^{(\sum^{s'}_{i=1}b'_is'_i)\alpha}
\times
(\prod^{s'}_{i=1}\det(\varrho'_i)(-1)^{s'_i})^{a\alpha}
\times
(-1)^{\#\phi^{<\alpha}_o(\rho)}\\
=&\det(\rho)(-1)^{\alpha\dim\phi}
\times \det(\phi)(-1)^{a\alpha}\times
(-1)^{\#\phi^{<\alpha}_o(\rho)}.
\end{align*}
Indeed, since $\phi_e$ is of even dimension,
$\sum^{s'}_{i=1}b'_is'_i$ and $\dim\phi$ are of the same parity.

If $\varphi$ is of symplectic type, continuing with \eqref{eq:prop:even-B}, we obtain \eqref{eq:prop:even-summand}.
If $\varphi$ is of orthogonal type, then $\phi$ is  of symplectic type,
which implies  that $\det(\phi)=1$.
One also has  \eqref{eq:prop:even-summand}.

Next, we show \eqref{eq:lm:chi-star-o} by considering  summands of form $\varrho\boxtimes\mu_{2\beta+1}$ in $\varphi$.
Write $b=\dim\varrho$.

{\bf Type (2).}
Assume that  $\varrho$ is of even dimension (that is, $b$ is even). In this case, since $\varrho\boxtimes\mu_{2\beta+1}$ has even dimension,
$\apair{(e_i)\vert_{\varrho\boxtimes\mu_{2\beta+1}}}\cong \BZ_2$ is a subgroup of $\CS_\varphi$ regardless to the type of $\varphi$.
Similarly, denote by $1_{\varrho\boxtimes\mu_{2\beta+1}}$  the nontrivial element in $\apair{(e_i)\vert_{\varrho\boxtimes\mu_{2\beta+1}}}$.
By definition, $\chi^\star_{\varphi}(1_{\varrho\boxtimes\mu_{2\beta+1}})$ equals
\begin{equation}\label{eq:prop:odd-1}
\begin{cases}
	\det(\phi)(-1)^{\frac{b(2\beta+1)}{2}}
\ve((\varrho\boxtimes\mu_{2\beta+1})\otimes\phi) &\text{ if $\varphi$ is of symplectic type}\\
	\det(\varrho)(-1)^{m}\ve((\varrho\boxtimes\mu_{2\beta+1})\otimes\phi) &\text{ if $\varphi$ is of orthogonal type},
\end{cases}	
\end{equation}
where $m=\frac{\dim\phi}{2}$.
Under the assumption, we need to show that
\begin{equation}\label{eq:prop:odd-summand-1}
\chi^\star_{\varphi}(1_{\varrho\boxtimes\mu_{2\beta+1}})=
\prod^{s'}_{i=1}
\CE(\varrho,\varrho'_i)^{s'_i}
(-1)^{\#\phi^{>\beta}_e(\varrho)}.	
\end{equation}
Recall from \eqref{eq:CE} that
$$
\CE(\varrho,\varrho'_i)=\begin{cases}
	\det(\varrho'_i)(-1)^{\frac{b}{2}}\varepsilon(\varrho\otimes\varrho'_i), &\text{ if $\varphi$ is of symplectic type}\\
	\det(\varrho)(-1)^{\frac{b'_i}{2}}\varepsilon(\varrho\otimes\varrho'_i), &\text{ if $\varphi$ is of orthogonal type}.
\end{cases}
$$
We may write that $\phi=\phi_e\boxplus\phi_o=\phi_e(\varrho)\boxplus\phi^-_e(\varrho)\boxplus\phi_o$, following
from \eqref{eq:decomp-e}, \eqref{eq:decomp-o}, and \eqref{minus} again.
It follows that
\begin{align*}
 \ve((\varrho\boxtimes\mu_{2\beta+1})\otimes\phi)
=& \varepsilon((\varrho\boxtimes\mu_{2\beta+1})\otimes\phi_e)
\cdot
\varepsilon((\varrho\boxtimes\mu_{2\beta+1})\otimes\phi_o,\psi_{F});\\
\ve((\varrho\boxtimes\mu_{2\beta+1})\otimes\phi_e)
=& \varepsilon((\varrho\boxtimes\mu_{2\beta+1})\otimes\phi_e(\varrho))
\cdot \varepsilon((\varrho\boxtimes\mu_{2\beta+1})\otimes\phi^-_e(\varrho)).
\end{align*}
Here only the term $\varepsilon((\varrho\boxtimes\mu_{2\beta+1})\otimes\phi_o,\psi_{F})$ is possibly dependent of $\psi_{F}$.
By Part \eqref{lm:eps-symplectic-odd} of Lemma \ref{lm:eps-tensor}, when   $\varrho\otimes\varrho'_i$ is of symplectic type, we have that
$$
\varepsilon((\varrho\boxtimes\mu_{2\beta+1})\otimes (\varrho'_i\boxtimes\mu_{2\beta'_{i,j}+1}),\psi_{F})=\varepsilon(\varrho\otimes \varrho'_{i}),
$$
which is independent of $\psi_{F}$.
As $\phi_o=\boxplus^{s'}_{i=1}\boxplus^{s'_i}_{j=1}\varrho'_i\boxtimes\mu_{2\beta'_{i,j}+1}$, we have that
\begin{align*}
\varepsilon((\varrho\boxtimes\mu_{2\beta+1})\otimes\phi_o,\psi_{F})
=&\prod^{s'}_{i=1}\prod^{s'_i}_{j=1}
\varepsilon((\varrho\boxtimes\mu_{2\beta+1})\otimes (\varrho'_i\boxtimes\mu_{2\beta'_{i,j}+1}),\psi_{F})\\
=& \prod^{s'}_{i=1}
\varepsilon(\varrho\otimes \varrho'_{i})^{s'_i},
\end{align*}
which is independent of $\psi_{F}$.

Now, let us calculate the first two terms: $\varepsilon((\varrho\boxtimes\mu_{2\beta+1})\otimes\phi_e(\varrho))$ and
$\varepsilon((\varrho\boxtimes\mu_{2\beta+1})\otimes\phi^-_e(\varrho))$.

Since $\phi^-_e(\varrho)=\boxplus^{r'}_{\substack{i=1\\i\ne i_0}}\boxplus^{r'_i}_{j=1}\rho'_i\boxtimes\mu_{2\alpha'_{i,j}}$,
by Part \eqref{lm:eps-ncong} of Lemma \ref{lm:eps-tensor}, we have that
\begin{align}
&\varepsilon((\varrho\boxtimes\mu_{2\beta+1})\otimes\phi^-_e(\varrho))\nonumber\\
=&\prod^{r'}_{\substack{i=1\\i\ne i_0}}\prod^{r'_i}_{j=1}	\det(\varrho)(-1)^{a'_i\alpha'_{i,j}(2\beta+1)}
\det(\rho'_i)(-1)^{b\alpha'_{i,j}(2\beta+1)}, \label{eq:lm:odd-c}
\end{align}
where $i_0$ is the index of $\rho'_{i_0}$ with $\rho'_{i_0}=\varrho$.

Similar to  \eqref{eq:lm:even-<} and \eqref{eq:lm:even->}, we write that $\phi_e(\varrho)=\phi^{\leq \beta}_e(\varrho)\boxplus \phi^{> \beta}_e(\varrho)$. Since $\phi^{\leq \beta}_e(\varrho)=\boxplus_{\{1\leq j\leq r'_{i_0}\colon \alpha'_{i_0,j}\leq \beta\}}\rho'_{i_0}\boxtimes\mu_{2\alpha_{i'_0,j}}$, we have, by Part \eqref{lm:eps-cong} of Lemma \ref{lm:eps-tensor}, that
\begin{align}
 & \varepsilon((\varrho\boxtimes\mu_{2\beta+1})\otimes\phi^{\leq \beta}_e(\varrho))\nonumber\\
=& \prod_{\{1\leq j\leq r'_{i_0}\colon \alpha'_{i_0,j}\leq \beta\}}\det(\varrho)(-1)^{a'_{i_0}\alpha'_{i_0,j}(2\beta+1)}
\det(\rho'_{i_0})(-1)^{b\alpha'_{i_0,j}(2\beta+1)}\cdot (-1)^{2\alpha'_{i_0,j}}\nonumber\\
=&\prod_{\{1\leq j\leq r'_{i_0}\colon \alpha'_{i_0,j}\leq \beta\}}\det(\varrho)(-1)^{a'_{i_0}\alpha'_{i_0,j}}
\det(\rho'_{i_0})(-1)^{b\alpha'_{i_0,j}}. \label{eq:lm:odd-<}
\end{align}
Since $\phi^{> \beta}_e(\varrho)=\boxplus_{\{1\leq j\leq r'_{i_0}\colon \alpha'_{i_0,j}> \beta\}}\rho'_{i_0}\boxtimes\mu_{2\alpha'_{i_0,j}}$,
we have, by Part \eqref{lm:eps-cong} of Lemma \ref{lm:eps-tensor}, that
\begin{align}
 & \varepsilon((\varrho\boxtimes\mu_{2\beta+1})\otimes\phi^{> \beta}_e(\varrho))\nonumber\\
=& \prod_{\{1\leq j\leq r'_{i_0}\colon \alpha'_{i_0,j}> \beta\}}\det(\varrho)(-1)^{a'_{i_0}\alpha'_{i_0,j}(2\beta+1)}
\det(\rho'_{i_0})(-1)^{b\alpha'_{i_0,j}(2\beta+1)}\cdot (-1)^{2\beta+1}\nonumber\\
=&(-1)^{\#\phi^{>\beta}_e(\varrho)}\times \prod_{\{1\leq j\leq r'_{i_0}\colon \alpha'_{i_0,j}> \beta\}}\det(\varrho)(-1)^{a'_{i_0}\alpha'_{i_0,j}}
\det(\rho'_{i_0})(-1)^{b\alpha'_{i_0,j}}, \label{eq:lm:odd->}
\end{align}
where $\#\phi^{>\beta}_e(\varrho)$ is the number of irreducibles in $\phi^{>\beta}_e(\varrho)$.

Finally, by taking the product of \eqref{eq:lm:odd-c}, \eqref{eq:lm:odd-<}, and \eqref{eq:lm:odd->}, we obtain that
\begin{align*}
 &\varepsilon((\varrho\boxtimes\mu_{2\beta+1})\otimes\phi_e(\varrho))
\cdot \varepsilon((\varrho\boxtimes\mu_{2\beta+1})\otimes\phi^-_e(\varrho))\\
=&\left(\prod^{r'}_{i=1}\prod^{r'_i}_{j=1}
\det(\varrho)(-1)^{a'_i\alpha'_{i,j}}
\det(\rho'_i)(-1)^{b\alpha'_{i,j}}\right)
\times (-1)^{\#\phi^{>\beta}_e(\varrho)}\\
=&\det(\varrho)(-1)^{\sum^{r'}_{i=1}\sum^{r'_i}_{j=1} a'_i\alpha'_{i,j}}
\cdot (-1)^{\#\phi^{>\beta}_e(\varrho)}.
\end{align*}
Recall that $b$ is even and $\det\rho'_i(-1)^b=1$.

If $\varphi$ is of symplectic type,  then $\varrho$ is of symplectic type, which implies that $\det(\varrho)(-1)=1$.
Continuing with \eqref{eq:prop:odd-1} and by
$$
\det(\phi)(-1)=\det(\phi_o)(-1)=\prod^{s'}_{i=1}\det(\varrho_i)(-1)^{s'_i},
$$
one has \eqref{eq:prop:odd-summand-1}.

If $\varphi$ is of orthogonal type and $b$ is even, then $\varrho$ and $\rho'_i$ are of orthogonal type and $\varrho'_i$ is of symplectic type.
Because
$$
m\equiv\sum^{r'}_{i=1}\sum^{r'_i}_{j=1} a'_i\alpha'_{i,j}+\sum^{s'}_{i=1}s'_i\frac{b'_i}{2}\bmod{2},
$$
continuing with \eqref{eq:prop:odd-1}, one obtains \eqref{eq:prop:odd-summand-1}.

{\bf Type (3).}
Assume that $b=\dim\varrho$ is odd, which implies that $\varphi$ is of orthogonal type.
Let $\varrho_0\boxtimes \mu_{2\beta_0+1}$ be any different summand  in $\varphi$ such that $b_0:=\dim\varrho_0$ is odd.
Consider the following  subgroup of $\CS_\varphi$
$$\apair{(e_i)\vert_{\varrho\boxtimes\mu_{2\beta+1}\boxplus\varrho_0\boxtimes \mu_{2\beta_0+1}}\colon (e_i)\in\CS_\varphi }\cong \BZ_2.$$
Denote by $1_{\varrho\boxtimes\mu_{2\beta+1}\boxplus\varrho_0\boxtimes \mu_{2\beta_0+1}}$ the non-trivial element in the above subgroup.
If $\varrho_0\boxtimes \mu_{2\beta_0+1}$ does not exist, we do not need to consider this case as $\chi_\varphi^\star$ is a character of $\CS_\varphi$.
Then
we have
\begin{align*}
 & \chi^\star_{\varphi}(1_{\varrho\boxtimes\mu_{2\beta+1}\boxplus\varrho_0\boxtimes \mu_{2\beta_0+1}})\\
=& \det(\varrho)(-1)^{m}\det(\varrho_0)(-1)^{m}
\ve((\varrho\boxtimes\mu_{2\beta+1})\otimes\phi)
\ve((\varrho_0\boxtimes\mu_{2\beta_0+1})\otimes\phi).
\end{align*}
Similar to the above calculation, one obtains that
\begin{align*}
 &\chi^\star_{\varphi}(1_{\varrho\boxtimes\mu_{2\beta+1}\boxplus\varrho_0\boxtimes \mu_{2\beta_0+1}})\\
=& \det(\varrho)(-1)^{m}\cdot \det(\varrho_0)(-1)^{m}\\
&\times \prod^{r'}_{i=1}\prod^{r'_i}_{j=1}
(\det(\varrho)(-1)\det(\varrho_{0})(-1))^{a'_i\alpha'_{i,j}}
\prod^{r'}_{i=1}\prod^{r'_i}_{j=1}\det(\rho'_i)(-1)^{(b+b_0)\alpha'_{i,j}} \\
&\times\prod^{s'}_{i=1}
\varepsilon(\varrho\otimes \varrho'_{i})^{s'_i}
(-1)^{\#\phi^{>\beta}_e(\varrho)}
\prod^{s'}_{i=1}
\varepsilon(\varrho_0\otimes \varrho'_{i})^{s'_i}
(-1)^{\#\phi^{>\beta}_e(\varrho_0)}\\
=&\prod^{s'}_{i=1}
(\det(\varrho)(-1)^{\frac{b'_{i}}{2}}\varepsilon(\varrho\otimes \varrho'_{i}))^{s'_i}
(-1)^{\#\phi^{>\beta}_e(\varrho)}\\
&\times \prod^{s'}_{i=1}
(\det(\varrho_{0})(-1)^{\frac{b'_{i}}{2}}\varepsilon(\varrho_0\otimes \varrho'_{i}))^{s'_i}
(-1)^{\#\phi^{>\beta}_e(\varrho_0)}.
\end{align*}
Finally, by putting all the calculations above together, we obtain \eqref{eq:lm:chi-star-o}.
\end{proof}

At the end of this section, we present an example of Lemma \ref{lm:key-lem},
which also will be used in the proof of Theorem \ref{pp:dsqlp}.
Let $\varphi^\dagger$ be a discrete $L$-parameter of the decomposition defined in \eqref{eq:varphi-o-dagger}.
\begin{exmp}\label{ex:psi}
Let $\psi$ be a discrete $L$-parameter of the different type from $\varphi^\dagger$ and $\psi$ be given in Item \eqref{item:prop-1} of Theorem \ref{pp:dsqlp}.
Applying Lemma \ref{lm:key-lem} to both $\varphi^\dagger$ and $\psi$, one has
\begin{equation}\label{eq:chi-dagger}
\chi^\star_{\varphi^\dagger}((e_{i}))= \prod^s_{i=1}\left(
(\prod^{k}_{l=1}
\CE(\varrho_i,\delta_l))\cdot
(\prod^{r}_{j=1}\CE(\varrho_i,\rho_{j})^{m_{j}})
 \cdot (-1)^{n_{i}} \right)^{e_{i}}.	
\end{equation}
\end{exmp}
Remark that from this example the explicit description on $\FD_{\ell_0}(\varphi,\chi)$ is reduced to find $\psi$ satisfying Equation \eqref{eq:chi-dagger} of the minimal dimension.

\subsection{Proof of Theorem \ref{pp:dsqlp}} \label{sec:proof}
Following Corollary \ref{pp:gtq}, we have
$$
\FD_{\ell_0}(\varphi,\chi)=\FD_{\ell_0}(\varphi_\Box,\chi),
$$
and all local $L$-parameters $\phi$ in $\FD_{\ell_0}(\varphi_\Box,\chi)$ are discrete.
The proof is reduced to the case that $\varphi$ is discrete.
It is enough to show that Items  \eqref{item:prop-dim}, \eqref{item:prop-1},  \eqref{item:prop-2}, and \eqref{item:prop-3} are {\sl necessary} and
{\sl sufficient} to characterize the set $\FD_{\ell_0}(\varphi,\chi)$.
First, assume that $\chi^\star_\varphi=\chi$.
Applying Lemma \ref{lm:key-lem}, we conclude that Items  \eqref{item:prop-1} and \eqref{item:prop-2} are the necessary conditions for the descent parameters in $\FD_{\ell_0}(\varphi,\chi)$.
Note that Item \eqref{item:prop-dim} holds by the definition.
Then assume that $\phi$ is of the decomposition $\psi\boxplus\varphi_\sgn$, where $\psi$ is given in Items  \eqref{item:prop-1} and satisfies \eqref{item:prop-2}.
We show that for such $\phi$,
$\chi^\star_\varphi(\varphi,\phi)=\chi$ is equivalent to $\chi_{\lceil\varphi_o\rceil}=\chi^\star_{\varphi^\dagger}$, which is Item \eqref{item:prop-3}.
The calculation of $\chi^\star_{\varphi^\dagger}$ is given in Example \ref{ex:psi}.
Finally, by the minimality condition on $\dim\phi$, Items  \eqref{item:prop-dim}, \eqref{item:prop-1}, \eqref{item:prop-2}, and \eqref{item:prop-3} are sufficient.

{\bf Necessity.} Take $\phi$ in $\FD_{\ell_0}(\varphi,\chi)$.
Then $\phi$ is of the minimal even dimension such that $\chi^\star_\varphi=\chi$ for the pair $(\varphi,\phi)$.
By Corollary \ref{pp:gtq}, $\phi$ is discrete.

First, we consider the subrepresentation $\phi_o(\rho_i)$.
By \eqref{eq:lm:chi-star-e} in Lemma \ref{lm:key-lem}, the parity of $\#\phi^{<\alpha_{i,j}}_o(\rho_i)$ is determined by $\chi(1_{\rho_i\boxtimes\mu_{\alpha_i,j}})$.
By \eqref{eq:lm:chi-star-o}, the value $\chi^\star_\varphi((e_{i,j})\vert_{\varphi_o})$
is partially affected by
 $\varepsilon(\varrho\otimes \varrho'_{i'})^{s'_{i'}}$ with $\varrho'_{i'}\cong \rho_i$ where $s'_{i'}=\#\phi_o(\rho_i)$, and more precisely by the parity of $s'_{i'}$.
By the minimality of $\dim\phi$, one has that
$$
\phi^{<\alpha_{i,r_i}}_o(\rho_i)=\boxplus^{r}_{i=1}\boxplus_{j\in \sgn_{o,\rho_i}(\chi)}\rho_i\boxtimes\mu_{2\alpha_{i,j}+1},
$$
where $\sgn_{o,\rho_i}(\chi)$ is defined in \eqref{sgn-o}.

Next, we consider the component $\phi_e(\varrho_i)$.
For $1\leq j_1<j_2\leq s_i$, by Lemma \ref{lm:key-lem}, we have that
\begin{equation}\label{eq:prop:varrho}
\chi(1_{\varrho_i\boxtimes\mu_{2\beta_{i,j_1}+1}\boxplus\varrho_i\boxtimes \mu_{2\beta_{i,j_2}+1}})=(-1)^{\#\phi^{>\beta_{i,j_1}}_e(\varrho_i)+\#\phi^{>\beta_{i,j_2}}_e(\varrho_i)}.
\end{equation}
The parity of $\#\phi^{>\beta_{i,j_1}}_e(\varrho_i)\pm \#\phi^{>\beta_{i,j_2}}_e(\varrho_i)$ is uniquely determined by $\chi(1_{\varrho_i\boxtimes\mu_{2\beta_{i,j_1}+1}\boxplus\varrho_i\boxtimes \mu_{2\beta_{i,j_2}+1}})$.
In addition, $\chi^\star_{\varphi}((e_i)\vert_{\varphi^-(\varrho)})$ is independent of $\phi_e(\varrho_i)$. The only requirement on $\phi^{<\beta_{i,s_i}}_e(\varrho_i)$ is that \eqref{eq:prop:varrho} holds for all $1\leq j_1<j_2\leq s_i$.
Then by the minimality of $\dim\phi$ one has, for $1\leq j<r_i$, that
$$
\phi^{>\beta_{i,j}}_e(\varrho_i)\boxminus\phi^{>\beta_{i,j+1}}_e(\varrho_i)
=\begin{cases}
	\varrho_i\boxtimes\mu_{2(\beta_{i,j}+1)} &\text{ if } \chi(\cdot)=-1,\\
	0&\text{ if } \chi(\cdot)=1.
\end{cases}
$$
Here $\chi(\cdot)=\chi(1_{\varrho_i\boxtimes\mu_{2\beta_{i,j}+1}\boxplus\varrho_i\boxtimes \mu_{2\beta_{i,j+1}+1}})$ for simplicity.
Thus, we obtain that
$$
\phi^{<\beta_{i,s_i}}_e(\varrho_i)=
\boxplus^{s}_{i=1}\boxplus_{j\in \sgn_{e,\varrho_i}(\chi)}\varrho_i\boxtimes\mu_{2\beta_{i,j}+2}，
$$
where $\sgn_{e,\varrho_i}(\chi)$ is defined in \eqref{sgn-e}.

Now we rewrite $\phi=\psi\boxplus \varphi_\sgn$, where $\psi$ and $\varphi_\sgn$ (see \eqref{eq:phi-sgn} for definition)
have no common irreducibles.
By the above discussion, Example \ref{ex:psi} and Lemma \ref{lm:key-lem}, we may assume that $\psi^{\leq \alpha_{i,r_i}}_o(\rho_i)$ and $\psi^{\leq \beta_{i,s_i}}_e(\varrho_i)$ are zero for all $\rho_i$ and $\varrho_i$.
In the rest of the proof, we will repeatedly apply Lemma \ref{lm:key-lem} and the minimality of $\dim\phi$ to obtain the requirements on $\psi$.
First, no matter what $\psi_e(\rho)$ for $\rho\notin\{\varrho_1,\dots,\varrho_{s}\}$ are, $\chi^\star_\varphi$ does not change in \eqref{eq:lm:chi-star-e} and \eqref{eq:lm:chi-star-o}.
It follows that $\psi_e(\rho)=0$ for $\rho\notin\{\varrho_1,\dots,\varrho_{s}\}$.
Next, note that for all $\varrho_i$ only the parity of $\#\psi^{>\beta_{i,s_i}}_e(\varrho_i)$ has non-trivial contribution in \eqref{eq:lm:chi-star-e} and \eqref{eq:lm:chi-star-o}, which implies, for all $1\leq j\leq s$, that
\begin{equation}\label{eq:prop-pf-tau-1}
\psi_e(\varrho_i)=n_j\varrho_j\boxtimes\mu_{2\beta_{j,s_j}+2},	
\end{equation}
where $n_j\in\{0,1\}$.
Finally, let us consider $\psi_o(\rho)$ in two cases:  $\rho\notin\{\rho_1,\dots,\rho_r\}$ or $\rho\cong\rho_i$ for some $i$.
If $\rho\notin\{\rho_1,\dots,\rho_r\}$, only the parity of $\#\psi_o(\rho)$ is involved in \eqref{eq:lm:chi-star-o}.
When $\psi_o(\rho)\ne 0$,
one has that
\begin{equation}\label{eq:prop-pf-tau-2}
\psi_o(\rho)=\rho\boxtimes 1. 	
 \end{equation}
If $\rho\cong\rho_i$ for some $i$, only the parity of $\#\sgn_{o,\rho_i}(\chi)-\#\psi^{\geq\alpha_{i,r_i}}_o(\rho_i)$ possibly changes the value of $\chi^\star_\varphi$ in \eqref{eq:lm:chi-star-o}.
Thus, we obtain that
\begin{equation}\label{eq:prop-pf-tau-3}
\psi_o(\rho_i)=m_i\rho_i\boxtimes\mu_{2\alpha_{i,r_i}+1},	
\end{equation}
where $m_i\in\{0,1\}$.
Combining \eqref{eq:prop-pf-tau-1}, \eqref{eq:prop-pf-tau-2} and \eqref{eq:prop-pf-tau-3} all together, we obtain Items \eqref{item:prop-1} and \eqref{item:prop-2}, which are necessary conditions on $\psi$.

{\bf Sufficiency.}
Assume that $\phi=\psi\boxplus\varphi_\sgn$ and $\psi$ is of the form in Item \eqref{item:prop-1} satisfying Item \eqref{item:prop-2}.
By the above discussion and \eqref{eq:lm:chi-star-e}, $\chi\vert_{\varphi_e}=\chi^\star_\varphi\vert_{\varphi_e}$.
In order to complete the proof,
it is sufficient to show that $\chi\vert_{\varphi_o}=\chi^\star_\varphi\vert_{\varphi_o}$ if and only if $\chi_{\lceil\varphi_o\rceil}=\chi^\star_{\varphi^\dagger}$.
In the rest of the proof, all the characters are on the corresponding subgroups $\CS_{\varphi_o}$, $\CS_{\lceil\varphi_o\rceil}$ and $\CS_{\varphi^\dagger}$.

By applying \eqref{eq:lm:chi-star-o} in Lemma \ref{lm:key-lem} to $\varphi$ and $\varphi^\dagger$, and by \eqref{eq:chi-dagger} in Example \ref{ex:psi}, we have
\begin{equation}\label{eq:prop:chi-star-dagger}
\chi^\star_\varphi((e_{i,j})\vert_{\varphi_o})=\chi^\star_{\varphi^\dagger}((\sum^{s_i}_{j=1}e_{i,j}))\cdot
\prod^{s}_{i=1}\prod^{s_{i}}_{j=1}((-1)^{\#\varphi^{>\beta_{i,j}}_{\sgn,e}(\varrho_i)})^{e_{i,j}}.	
\end{equation}

Define $f\colon (e_{i,j})\vert_{\varphi_o}\in\CS_{\varphi_o}\mapsto (\sum^{s_i}_{j=1}e_{i,j})\in A_{\varphi^\dagger},$
where $\CS_{\varphi_o}$ is considered as a subgroup of $\CS_\varphi$.
It is easy to check that $f$ is surjective on $\CS_{\varphi^\dagger}$.
(In general it is not surjective on $A_{\varphi^\dagger}$.)
Denote $f\vert_{\lceil\varphi_o\rceil}$ to be the restriction map into $\CS_{\lceil\varphi_o\rceil}$.
Then $f\vert_{\lceil\varphi_o\rceil}$ is an isomorphism.
Following \eqref{eq:prop:chi-star-dagger} and by $\#\varphi^{>\beta_{i,s_i}}_{\sgn,e}(\varrho_i)=0$ for all $1\leq i\leq s$, one has that
\begin{equation}\label{eq:prop:chi-ceil=dagger}
\chi^\star_\varphi((e_{i,j})\vert_{\lceil\varphi_o\rceil})=\chi^\star_{\varphi^\dagger}((e_{i,s_i}))=\chi^\star_{\varphi^\dagger}(f((e_{i,j})\vert_{\lceil\varphi_o\rceil})).
\end{equation}
By \eqref{eq:prop:chi-ceil=dagger},
if $\chi\vert_{\varphi_o}=\chi^\star_\varphi\vert_{\varphi_o}$
then $\chi_{\lceil\varphi_o\rceil}=\chi^\star_{\varphi^\dagger}$ as $f\vert_{\lceil\varphi_o\rceil}$ is an isomorphism.

Suppose that $\chi_{\lceil\varphi_o\rceil}=\chi^\star_{\varphi^\dagger}$.
Let $\varrho_i\boxtimes \mu_{2\beta_{i,j}+1}$ be an irreducible subrepresentation of $\varphi_o$.
We consider two cases: $\dim\varrho_i$ is even and $\dim\varrho_i$ is odd, respectively,
as $\CS_{\varphi_o}$ is generated by the elements of two types $1_{\varrho_i\boxtimes \mu_{2\beta_{i,j}+1}}$, and $1_{\varrho_i\boxtimes \mu_{2\beta_{i,j}+1}\boxplus\varrho_{i'}\boxtimes \mu_{2\beta_{i',j'}+1}}$ with $i\ne i'$, or $i=i'$ and $j\ne j'$.

Assume that $\dim\varrho_i$ is even.
In this case, $1_{\varrho_i\boxtimes \mu_{2\beta_{i,j}+1}}$ is in $\CS_{\varphi_o}$.
Set
$$
j_0=\#\{j\leq l< s_i\colon l\in\sgn_{e,\varrho_i}(\chi)\}=\#\varphi^{>\beta_{i,j}}_{\sgn,e}(\varrho_i).
$$
By the definition of $\sgn_{e,\varrho_i}(\chi)$, one has that
$$
\chi(1_{\varrho_i\boxtimes \mu_{2\beta_{i,j}+1}})=(-1)^{j_0}\chi(1_{\varrho_i\boxtimes \mu_{2\beta_{i,s_i}+1}})=
(-1)^{j_0}\chi_{\lceil\varphi_o\rceil}(1_{\varrho_i\boxtimes \mu_{2\beta_{i,s_i}+1}}).
$$
By \eqref{eq:prop:chi-star-dagger}, we have that
$
\chi^\star_\varphi(1_{\varrho_i\boxtimes \mu_{2\beta_{i,j}+1}})=\chi^\star_{\varphi^\dagger}(1_{\varrho_i\boxtimes 1})(-1)^{\#\varphi^{>\beta_{i,j}}_{\sgn,e}(\varrho_i)}.
$
Thus, we obtain that $\chi(1_{\varrho_i\boxtimes \mu_{2\beta_{i,j}+1}})=\chi^\star_\varphi(1_{\varrho_i\boxtimes \mu_{2\beta_{i,j}+1}})$.

Assume that $\dim\varrho_i$ is odd.
Consider the elements in $\CS_{\varphi_o}$ of form $1_{\varrho_i\boxtimes \mu_{2\beta_{i,j}+1}\boxplus\varrho_{i'}\boxtimes \mu_{2\beta_{i',j'}+1}}$ with $i\ne i'$, or $i=i'$ and $j\ne j'$.
Suppose that $i=i'$ and $j<j'$.
Set
$$
j_0=\#\{j\leq l< j'\colon l\in\sgn_{e,\varrho_i}(\chi)\}=\#\varphi^{>\beta_{i,j}}_{\sgn,e}(\varrho_i)-\#\varphi^{>\beta_{i,j'}}_{\sgn,e}(\varrho_i).
$$
One has that
$$
\chi(1_{\varrho_i\boxtimes \mu_{2\beta_{i,j}+1}\boxplus\varrho_{i}\boxtimes \mu_{2\beta_{i,j'}+1}})=(-1)^{j_0}=
(-1)^{\#\varphi^{>\beta_{i,j}}_{\sgn,e}(\varrho_i)-\#\varphi^{>\beta_{i,j'}}_{\sgn,e}(\varrho_i)}.
$$
Since $f(1_{\varrho_i\boxtimes \mu_{2\beta_{i,j}+1}\boxplus\varrho_{i}\boxtimes \mu_{2\beta_{i,j'}+1}})=0\in \CS_{\varphi^\dagger}$, one has, by \eqref{eq:prop:chi-star-dagger}, that
$$
\chi^\star_\varphi(1_{\varrho_i\boxtimes \mu_{2\beta_{i,j}+1}\boxplus\varrho_{i}\boxtimes \mu_{2\beta_{i,j'}+1}})
=\chi^\star_{\varphi^\dagger}(0)
(-1)^{\#\varphi^{>\beta_{i,j}}_{\sgn,e}(\varrho_i)+\#\varphi^{>\beta_{i,j'}}_{\sgn,e}(\varrho_i)},
$$
which is equal to $\chi(1_{\varrho_i\boxtimes \mu_{2\beta_{i,j}+1}\boxplus\varrho_{i}\boxtimes \mu_{2\beta_{i,j'}+1}})$.

Suppose that $i\ne i'$.
Denote $1_{(i,j),(l,k)}=1_{\varrho_i\boxtimes \mu_{2\beta_{i,j}+1}\boxplus\varrho_{l}\boxtimes \mu_{2\beta_{l,k}+1}}$.
When $i=l$ and $j=k$, define that $1_{(i,j),(l,k)}=0$. Then $1_{(i,j),(l,k)}$ is in $\CS_{\varphi_o}$.
We rewrite that
$$
1_{\varrho_i\boxtimes \mu_{2\beta_{i,j}+1}\boxplus\varrho_{i'}\boxtimes \mu_{2\beta_{i',j'}+1}}=1_{(i,j),(i,s_i)}+1_{(i,s_i),(i',s_{i'})}+1_{(i',s_{i'}),(i',j')}.
$$
In the above case, we proved that
$\chi(1_{(i,j),(i,s_i)})=\chi^\star_\varphi(1_{(i,j),(i,s_i)})$
and
$\chi(1_{(i',s_{i'}),(i',j')})=\chi^\star_\varphi(1_{(i',s_{i'}),(i',j')})$.

It remains to show that $\chi(1_{(i,s_i),(i',s_{i'})})=\chi^\star_\varphi(1_{(i,s_i),(i',s_{i'})})$ with $i\ne i'$.
By definition, one has that $\chi(1_{(i,s_i),(i',s_{i'})})=\chi_{\lceil\varphi_o\rceil}(1_{(i,s_i),(i',s_{i'})})$.
Note that $\#\varphi^{>\beta_{i,s_i}}_{\sgn,e}(\varrho_i)=\#\varphi^{>\beta_{i',s_{i'}}}_{\sgn,e}(\varrho_{i'})=0$.
It then follows from \eqref{eq:prop:chi-star-dagger} that
$
\chi^\star_\varphi(1_{(i,s_i),(i',s_{i'})})=\chi^\star_{\varphi^\dagger}(1_{\varrho_i\boxtimes 1\boxplus\varrho_{i'}\boxtimes 1}),
$
which implies that $\chi(1_{(i,s_i),(i',s_{i'})})=\chi^\star_\varphi(1_{(i,s_i),(i',s_{i'})})$.
Therefore, we complete the proof of Theorem \ref{pp:dsqlp}.



\section{Proof of Theorems \ref{th:ldds} and \ref{th:sdlp}}\label{sec-MTHM}


In this section, we will apply our main results in Theorem \ref{pp:dsqlp} to prove Theorems \ref{th:ldds} and \ref{th:sdlp}.
First, we use the descents of discrete parameters studied in Section \ref{sec-DLP}
to recover the local descents of representations via the local Langlands correspondence.

\begin{prop}[Part \eqref{thm:sdlp:part-1} of Theorem \ref{th:sdlp}]\label{cor:genric-disc}
For a $\pi\in\Pi(G_n)$ with a generic $L$-parameter $\varphi$,
the first occurrence index of $\pi$ can be calculated by
$$
\ell_0(\pi)=\max_{\chi\in\CO_\CZ(\pi)}\{\ell_0(\varphi,\chi)\}=n-\frac{\dim(\varphi_\Box)}{2}+\max_{\chi\in\CO_\CZ(\pi)}\{\ell_0(\varphi_\Box,\chi)\},
$$
where $\CO_\CZ(\pi)$ is defined in \eqref{eq:CZ} and $\varphi_\Box$ is defined in \eqref{eq:discrete}.
Moreover, suppose that
$\chi_0$ is such a character that
$$
\ell_0(\varphi_\Box,\chi_0)=\max_{\chi\in\CO_\CZ(\pi)}\{\ell_0(\varphi_\Box,\chi)\}
$$
holds, and that $\pi=\pi_a(\varphi,\chi_0)$ under the local Langlands correspondence $\iota_a$.
Define that $\pi_\Box:=\pi_a(\varphi_\Box,\chi_0)\in\Pi(G_{n_0})$ where $n_0=\frac{\dim\varphi_\Box}{2}$.
If $\disc(\CO_{\ell_0(\pi)})=\disc(\CO_{\ell_0(\pi_\Box)})$,
then an irreducible square integrable quotient occurs in $\CD_{\CO_{\ell_0(\pi)}}(\pi)$ if and only if it occurs in
$\CD_{\CO_{\ell_0(\pi_\Box)}}(\pi_\Box)$.
\end{prop}

\begin{proof}
For each $\chi\in\CS_\varphi\cong \CS_{\varphi_\Box}$ and a generic $L$-parameter $\phi$ of the type different from that of $\varphi$,
since $\chi^\star(\varphi,\phi)=\chi^\star(\varphi_\Box,\phi)$,
we have $\FD_{\ell}(\varphi,\chi)=\FD_{\ell}(\varphi_\Box,\chi)$.
Hence $\ell_0(\varphi,\chi)=n-\frac{\dim(\varphi_\Box)}{2}+\ell_0(\varphi_\Box,\chi)$.
It suffices to show that $\ell_0(\pi)=\max_{\chi\in\CO_\CZ(\pi)}\{\ell_0(\varphi,\chi)\}$.

By Corollary \ref{pp:gtq}, $\CD_{\CO_{\ell_0}}(\pi)$ is nonzero for some rational orbit $\CO_{\ell_0}$ if and only if  there exists a nonzero irreducible square-integrable representation $\sig$ of $G^{\CO_{\ell_0}}_n$ such that $m(\pi,\sig)\ne 0$.
If $m(\pi,\sig)\ne 0$,
assume that $\pi=\pi_a(\varphi,\chi_\pi)$ and $\sig=\pi_a(\phi,\chi_\sig)$.
By Theorem \ref{thm:W12b}, $(\chi_\pi,\chi_\sig)=\chi^\star(\varphi,\phi)$ and
then $[\phi]_c$ is in $\FD_{\ell_0(\pi)}(\varphi,\chi_\pi)$.
It follows that
$$
\ell_0(\pi)\leq \ell_0(\varphi,\chi_\pi)\leq \max_{\chi\in\CO_\CZ(\pi)}\{\ell_0(\varphi,\chi)\}.
$$
On the other hand, let $\chi_0$ be in $\CO_{\CZ}(\pi)$ such that
$$
\ell_0(\varphi,\chi_0)=\max_{\chi\in\CO_\CZ(\pi)}\{\ell_0(\varphi,\chi)\}
=n-\frac{\dim(\varphi_\Box)}{2}+\max_{\chi\in\CO_\CZ(\pi)}\{\ell_0(\varphi_\Box,\chi)\}.
$$
Take  an equivalence class $[\phi]_c$ of the local $L$-parameters in $\FD_{\ell^{\max}_0}(\varphi,\chi_0)$ at the first occurrence index $\ell^{\max}_0=\ell_0(\varphi,\chi_0)$.
Choose the local Langlands correspondence $\iota_a$ such that $\pi=\pi_a(\varphi,\chi_0)$.
By Definition \ref{df:foi-L}, $\chi_0=\chi^\star_{\varphi}(\varphi,\phi)$.
Denote $\chi_\phi=\chi^\star_{\phi}(\varphi,\phi)$ and $\sig=\pi_a(\phi,\chi_\phi)$ under the same Langlands correspondence $\iota_a$.
By Theorem \ref{thm:W12b}, we have $m(\pi,\sig)=1$ and $\sig$ is an irreducible quotient in $\CD_{\CO_{\ell^{\max}_0}}(\pi)$.
Then $\CD_{\CO_{\ell^{\max}_0}}(\pi)\ne 0$ and $\ell_0(\pi)\geq \ell^{\max}_0$.
Hence, we have that $\ell_0(\pi)=\ell^{\max}_0$.

Consider the local descents given by
the rational orbits $\CO_{\ell_0(\pi)}$ and $\CO_{\ell_0(\pi_\Box)}$ with $\disc(\CO_{\ell_0(\pi)})=\disc(\CO_{\ell_0(\pi_\Box)})$.
Referring to Theorem \ref{thm:W12b} and Remark \ref{rmk:O-V-W}, the local Langlands correspondence $\iota_b$ for $\pi$ is determined by $\disc(\CO_{\ell_0(\pi)})$,
so is the same choice for $\pi_\Box$.
Recall that the corresponding character of $\CS_\varphi\cong \CS_{\varphi_\Box}$ is denoted by $\chi_b:=\chi_b(\pi)=\chi_b(\pi_\Box)$.

By the definition in \eqref{eq:CZ}, $\CO_\CZ(\pi)=\CO_\CZ(\pi_\Box)$ and then
$\ell_0(\pi_\Box)=\max_{\chi\in\CO_\CZ(\pi)}\{\ell_0(\varphi_\Box,\chi)\}$.
By the above proof, $\ell_0(\pi)=n-\frac{\dim\varphi_\Box}{2}+\ell_0(\pi_\Box)$.
If $\disc(\CO_{\ell_0(\pi)})=\disc(\CO_{\ell_0(\pi_\Box)})$, then $G^{\CO_{\ell_0(\pi)}}_n\cong G^{\CO_{\ell_0(\pi_\Box)}}_{n_0}$.

If $\ell_0(\varphi,\chi_b(\pi))<\ell_0(\pi)$ (equivalently, $\ell_0(\varphi_\Box,\chi_b(\pi_\Box))<\ell_0(\pi_\Box)$), then both $\FD_{\ell_0(\pi)}(\varphi,\chi_b)$ and $\FD_{\ell_0(\pi_\Box)}(\varphi_\Box,\chi_b)$ are empty.
By Theorem \ref{thm:W12b}, $\CD_{\ell_0(\pi)}(\pi)=\CD_{\ell_0(\pi_\Box)}(\pi_\Box)=0$.

Assume that $\ell_0(\varphi,\chi_a(\pi))=\ell_0(\pi)$ for some $\iota_a$ as denoted in the proposition.
With the above notation, an irreducible square-integrable representation $\sig$ of $G^{\CO_{\ell_0}}_n$ occurs in $\CD_{\CO_{\ell_0(\pi)}}(\pi)$ if and only if $(\chi_a(\pi),\chi_\sig)=\chi^\star(\varphi,\phi)$ by Theorem \ref{thm:W12b}.
Recall that by the definition of $\pi_\Box$, $\chi_a(\pi)=\chi_a(\pi_\Box)$ under the same local Langlands correspondence.
By $\chi^\star(\varphi,\phi)=\chi^\star(\varphi_\Box,\phi)$,
we have $\chi^\star(\varphi_\Box,\phi)=(\chi_a(\pi_\Box),\chi_\sig)$,
which is equivalent that $\sig$ is also an irreducible quotient of $\CD_{\CO_{\ell_0(\pi)}}(\pi_\Box)$.
\end{proof}

In order to prove Theorem \ref{th:ldds},
it remains to show that the irreducible quotients of the local descent $\CD_{\CO_{\ell_0}}(\pi)$ (at the first occurrence index $\ell_0=\ell_0(\pi)$ of $\pi$)
belong to different Bernstein components of $G_n^{\CO_{\ell_0}}(F)$.

\begin{thm}[Theorem \ref{th:ldds}]\label{pp:dsds}
For any $\pi\in\Pi(G_n)$ with a generic $L$-parameter,
 irreducible quotients of the local descent $\CD_{\CO_{\ell_0}}(\pi)$ at the first occurrence index $\ell_0=\ell_0(\pi)$ of $\pi$ belong to different Bernstein components.
Moreover, $\CD_{\CO_{\ell_0}}(\pi)$ can be written as a
multiplicity-free direct sum
of irreducible square-integrable representations of $G_n^{\CO_{\ell_0}}(F)$, and hence is square-integrable and admissible.
\end{thm}

\begin{rmk}
In general, the local descent $\CD_{\CO_{\ell_0}}(\pi)$ at the first occurrence index could be a direct sum of infinitely many irreducible non-supercuspidal square-integrable representations.
When $\ell<\ell_0$, the descent $\CD_{\CO_{\ell}}(\pi)$ may not be completely reducible.
\end{rmk}

\begin{proof}
Assume that $\pi=\pi_a(\varphi,\chi)$ under local Langlands correspondence $\iota_a$ for both even and odd special orthogonal groups discussed in Section \ref{ssec-rllc}.
The local descent
$\CD_{\CO_{\ell_0}}(\pi)$ is a smooth representation of $G_n^{\CO_{\ell_0}}(F)$.
By the Bernstein decomposition, $\CD_{\CO_{\ell_0}}(\pi)$ is a direct sum of its Bernstein components.
Thus, it is sufficient to show that if $\sig_1$ and $\sig_2$ are non-isomorphic irreducible quotients of $\CD_{\CO_{\ell_0}}(\pi)$, i.e.
$$
m(\pi_a(\varphi,\chi),\sig_1)=m(\pi_a(\varphi,\chi),\sig_2)=1,
$$
then $\sig_1$ and $\sig_2$ have different cuspidal supports.
Corollary \ref{pp:gtq} asserts that $\sig_1$ and $\sig_2$ are square-integrable.

By Proposition \ref{cor:genric-disc}, we may assume, without loss of generality, that $\varphi$ is discrete.
Then $\sig_1=\pi_a(\phi_1,\xi_1)$ and $\sig_2=\pi_a(\phi_2,\xi_2)$, where $\phi_1$ and $\phi_2$ are of the form in Theorem \ref{pp:dsqlp}.
We have decompositions $\phi_i=\psi_i\boxplus\varphi_{\sgn}$ for $i=1,2$.
Let $\lam_{\phi_i}$ be the infinitesimal characters of $\phi_i$ (as explained in \cite[Page~69]{A13}, for instance).
That is,
$$
\lam_{\phi_i}(w)=\phi_i\ppair{w,
\begin{pmatrix}
|w|^{\frac{1}{2}}& \\ &|w|^{-\frac{1}{2}}
 \end{pmatrix} }.
$$
Referring to \cite{AMS15}, if $\lam_{\phi_1}\ne \lam_{\phi_2}$, then $\Pi_{\phi_1}(G_n^{\CO_{\ell_0}})$ and $\Pi_{\phi_2}(G_n^{\CO_{\ell_0}})$ belong to different Bernstein components.
Because $\phi_1\ne \phi_2$, we have that $\psi_1\ne \psi_2$, which implies that
$\lam_{\psi_1}\ne \lam_{\psi_2}$ by the definition of $\psi_i$.
Then $\lam_{\phi_1}\ne \lam_{\phi_2}$.
Hence each Bernstein component of the local descent $\CD_{\CO_{\ell_0}}(\pi)$ has a unique irreducible quotient, which is square-integrable
by Corollary \ref{pp:gtq}. By the definition of the $\ell$-th local descent $\CD_{\CO_\ell}(\pi)$, which is the $\ell$-th maximal quotient of
the $\ell$-th twisted Jacquet module $\CJ_{\CO_\ell}(\pi)$, it follows that each Bernstein component of the local descent $\CD_{\CO_{\ell_0}}(\pi)$
is an irreducible square-integrable representation of $G_n^{\CO_{\ell_0}}(F)$. Therefore, the local descent $\CD_{\CO_{\ell_0}}(\pi)$ is
a multiplicity-free direct sum of irreducible square-integrable representations of $G_n^{\CO_{\ell_0}}(F)$, and hence is itself square-integrable.
Moreover, for any compact open subgroup $K$ of $G_n^{\CO_{\ell_0}}(F)$,
there are only finitely many square-integrable representations of $G_n^{\CO_{\ell_0}}(F)$ with $K$-fixed vectors.
Thus, the $K$-invariant subspace of $\CD_{\CO_{\ell_0}}(\pi)$ is finite and hence the local descent $\CD_{\CO_{\ell_0}}(\pi)$ is admissible.
\end{proof}

Now, let us finish the proof of Part \eqref{thm:sdlp:part-2} of Theorem \ref{th:sdlp}.
Following Theorem \ref{pp:dsds}, for each $F$-rational orbit $\CO_{\ell_0}$,
$\CD_{\ell_0}(\pi)$ is  either zero or a direct sum of square-integrable representations.
Recall that an irreducible square-integrable representation $\sig$ is a subrepresentation of $\CD_{\ell_0}(\pi)$ if and only if the data associated to $\sig$ belong to $\FD_{\ell_0}(\varphi,\chi)$ for some $\chi\in\CO_\CZ(\pi)$ by Theorem \ref{thm:W12b}.
To prove this theorem,   we will characterize $\sig$ by the descents of $L$-parameters in Theorem \ref{pp:dsqlp}.

Fix a choice of $F$-rational orbit $\CO_{\ell_0}$.
Then the quadratic space $W$ defining $G^{\CO_{\ell_0}}_n$ is determined by
$\disc(W)=(-1)^{\Fn-1}\disc(\CO_{\ell_0})\disc(V_\Fn)$ (refer to Remark \ref{rmk:O-V-W}).
Following Theorem \ref{thm:W12b}, we choose the local Langlands correspondence $\iota_a$ where $a=(-1)^{\Fn}\disc(\CO_\ell)$ for the even special orthogonal group.
For the odd special orthogonal group, the normalization is unique.

Assume that $G_n$ is a even special orthogonal group.
Let $\pi=\pi_a(\varphi,\chi_a)$ where $\chi_a=\chi_a(\pi)$ via   $\iota_a$.
If $\ell_0(\varphi,\chi_a) < \ell_0(\pi)$, then $\FD_{\ell_0}(\varphi,\chi_a)$ is empty and the local descent $\CD_{\ell_0}(\pi)$ is zero for the $F$-rational orbit $\CO_{\ell_0}$.
If $\ell_0(\varphi,\chi_a)=\ell_0(\pi)$, then $\FD_{\ell_0}(\varphi,\chi_a)$ is not empty.
By the definition of $\FD_{\ell_0}(\varphi,\chi_a)$, for $\phi\in \FD_{\ell_0}(\varphi,\chi_a)$, $\chi_a=\chi^\star(\varphi,\phi)$ and denote $\chi_\phi$ to $\chi^\star(\varphi,\phi)$.
Under $\iota_a$,  by \eqref{eq:thm:HW}, the corresponding representation $\pi_a(\phi,\chi_\phi)$ of $G^{\CO_{\ell_0}}_n$ occurs in $\CD_{\CO_{\ell_0}}(\pi)$.
This gives the decomposition \eqref{item:main-thm-even} in Theorem \ref{th:sdlp}.

Assume that $G_n$ is an odd special orthogonal group.
In this case, the set $\CO_\CZ(\pi)$ is a singleton and $\pi=\pi(\varphi,\chi)$. Thus $\ell_0(\varphi,\chi_\pi)=\ell_0(\pi)$ and $\ell_0(\pi)=\ell_0(\varphi,\chi)$.
Let $W$ be the quadratic space defining $G^{\CO_{\ell_0}}_n$.
By \eqref{eq:thm:DV}, only the $L$-parameters $\phi$ satisfying $\det\phi=\disc(\CO_{\ell_0})\disc(V_\Fn)$ correspond representations of $G^{\CO_{\ell_0}}_n$.
Thus if $\det\phi=\disc(\CO_{\ell_0})\disc(V_\Fn)$, by the definition of $\FD_{\ell_0}(\varphi,\chi)$, $W$ satisfies Equations \eqref{eq:thm:DV} and \eqref{eq:thm:HV}.
Then $\pi_a(\phi,\chi^\star_\phi(\varphi,\phi))$ is a representation of $G^{\CO_{\ell_0}}_n$, where $a=-\disc(\CO_{\ell_0})$.
This gives the decomposition \eqref{item:main-thm-odd} in Theorem \ref{th:sdlp}.

With Proposition \ref{cor:genric-disc} together, we finally complete the proof of Theorem \ref{th:sdlp}

Furthermore, by following  Theorem \ref{pp:dsqlp} and Lemma \ref{lm:key-lem}, we will give a formula of
$\chi^\star_\phi(\varphi,\phi)$ (simply written as $\chi^\star_\phi$) in Theorem \ref{th:sdlp}.
For $L$-parameters $\phi\in \FD_{\ell_0}(\varphi,\chi)$,
$(\phi,\chi^\star_\phi)$ determine the irreducible square-integrable representations $\sigma=\pi_a(\phi,\chi^\star_{\phi}(\varphi,\phi))$
of $G_n^{\CO_{\ell_0}}(F)$, via the given local Langlands correspondence $\iota_a$.
These $\sig$ occur as irreducible summands in the local descent $\CD_{\CO_{\ell_0}}(\pi(\varphi,\chi))$.

Assume that $\phi\in\FD_{\ell_0}(\varphi,\chi)$ is as given in Theorem \ref{pp:dsqlp}, which can be written as
$$
\phi=(\boxplus^k_{l=1}\delta_l\boxtimes 1) \boxplus(\boxplus^r_{{\rm i}=1}m_{\rm i}\rho_{\rm i}\boxtimes\mu_{2\alpha_{{\rm i},r_{\rm i}}+1})
\boxplus(\boxplus^s_{{\rm i'}=1}n_{\rm i'}\varrho_{\rm i'}\boxtimes\mu_{2\beta_{{\rm i'},s_{\rm i'}}+2})
\boxplus \varphi_\sgn.
$$
Here the notation follows from Theorem \ref{pp:dsqlp}.
We may use more self-explanatory notation for the elements in $A_\phi$ to make the formula clearer.
A general element of $A_\phi$ is written as
$$
((e_{\delta,l}),(e_{\rho,{\rm i}}),(e_{\varrho,{\rm i'}}),(e_{i,j}),(\Fe_{i',j'})).
$$
Those components correspond to the components determined by the summands: $\delta_{\rm i}\boxtimes 1$,
$m_{\rm i}\rho_{\rm i}\boxtimes\mu_{2\alpha_{{\rm i},r_{\rm i}}+1}$, $n_{\rm i'}\varrho_{\rm i'}\boxtimes\mu_{2\beta_{{\rm i'},s_{\rm i'}}+2}$,
$\rho_i\boxtimes\mu_{2\alpha_{i,j}+1}$, and $\varrho_{i'}\boxtimes\mu_{2\beta_{i',j'}+2}$, respectively.
Note that $(e_{i,j})$ and $(\Fe_{i',j'})$ are indexed by $1\leq i \leq r$ and $j\in \sgn_{o,\rho_i}(\chi)$, and  by $1\leq i'\leq s$ and $j'\in \sgn_{e,\varrho_{i'}}(\chi)$, respectively.
\begin{cor}\label{character}
With the notation as in Theorem \ref{th:sdlp} and Theorem \ref{pp:dsqlp},
for $\phi\in \FD_{\ell_0}(\varphi,\chi)$, the character
$\chi^\star_\phi((e_{\delta,l}),(e_{\rho,{\rm i}}),(e_{\varrho,{\rm i'}}),(e_{i,j}),(\Fe_{i',j'}))$
can be explicitly written as the following product:
\begin{align*}
&\prod^{k}_{l=1}  (\prod^{s}_{l'=1}\CE(\delta_l,\varrho_{l'})^{s_{l'}})^{e_{\delta,l}}
\times\prod^{r}_{{\rm i}=1} (\prod^{s}_{l'=1}\CE(\rho_{{\rm i}},\varrho_{l'})^{s_{l'}})^{m_{\rm i}e_{\rho,{\rm i}}} \\
\times& \prod^{r}_{i=1}\prod_{j\in \sgn_{o,\rho_i}(\chi)} (\prod^{s}_{l'=1}\CE(\rho_{i},\varrho_{l'})^{s_{l'}}\cdot (-1)^{r_i-j})^{e_{i,j}}\\
\times& \prod^{s}_{{\rm i}'=1}(-1)^{n_{{\rm i}'}s_{\rm i'}e_{\rm i'}}
\times \prod_{i'=1}^{s}\prod_{j'\in \sgn_{e,\varrho_{i'}}(\chi)} (-1)^{j'\Fe_{i',j'}},
\end{align*}
which can also be written as the following product:
\begin{align*}
&\prod^{k}_{l=1}  (\prod^{s}_{l'=1}\CE(\delta_l,\varrho_{l'})^{s_{l'}})^{e_{\delta,l}}
\times (\prod^{s}_{l'=1}\CE(\rho_{{\rm i}},\varrho_{l'})^{s_{l'}})^{\sum^{r}_{{\rm i}=1} m_{\rm i}e_{\rho,{\rm i}}+\sum^{r}_{i=1}\sum_{j\in \sgn_{o,\rho_i}(\chi)} e_{i,j}} \\
&\times \prod^{r}_{i=1}\prod_{j\in \sgn_{o,\rho_i}(\chi)}   (-1)^{(r_i-j)e_{i,j}}
\times \prod^{s}_{{\rm i}'=1}(-1)^{n_{{\rm i}'}s_{\rm i'}e_{\rm i'}}\times
\prod_{i'=1}^{s}\prod_{j'\in \sgn_{e,\varrho_{i'}}(\chi)} (-1)^{j'\Fe_{i',j'}}.
\end{align*}
\end{cor}



\section{Examples}\label{sec:ex}


We will consider the descents for two special families of the discrete local $L$-parameters.
The spectral decomposition of the local descents in those cases can be even more explicitly described. Also, we will discuss Conjecture \ref{conj-p1}
via some examples.

\subsection{ Cuspidal Local $L$-parameters}\label{sec:clp}
In order to understand the summand $\varphi_\sgn$ defined in \eqref{eq:phi-sgn} occurring in the local descent $\FD_{\ell_0}(\varphi,\chi)$,
we give an example on such summands for cuspidal local $L$-parameters, and their descents.

In \cite[Conjecture 7.5]{AMS15}, Aubert, Moussaoui, and Solleveld
conjectured that an $L$-packet $\Pi_\varphi(G_n)$ of a reductive group contains a supercuspidal representation if and only if
$\varphi$ is a cuspidal $L$-parameter. For the case of split special orthogonal groups considered here, 
it was proved by M\oe glin in \cite{M11} (see also \cite{Xu17}).
One may write cuspidal $L$-parameters as 
$$
\varphi=\boxplus^{r}_{i=1}\boxplus^{r_i}_{j=1}\rho_i\boxtimes\mu_{2j}\boxplus
\boxplus^{s}_{i=1}\boxplus^{s_i}_{j=1}\varrho_i\boxtimes\mu_{2j-1},
$$
where $\sum^s_{i=1}s_ib_i$ (here $b_i=\dim\varrho_i$) is even.
It follows from \cite{M11} (see also \cite[Proposition 3.7]{Mou15}) that $\varphi$ is a cuspidal $L$-parameter  for split even orthogonal group $G_n$
when $\varphi$ is of orthogonal type and $\prod^{s}_{i=1}\det\varrho_i=1$.

Similar to \eqref{eq:e-i}, we write the
elements in $\CS_\varphi$ in the form $((e_{i,j}),(\Fe_{i,j}))$, indexed by the set
$$\{e_{i,j}\colon 1\leq i\leq r,~ 1\leq j\leq r_i\}\cup
\{\Fe_{i,j}\colon 1\leq i\leq s,~ 1\leq j\leq s_i\}.
$$
The index sets correspond to the summands $\rho_i\boxtimes\mu_{2j}$ and $\varrho_i\boxtimes\mu_{2j-1}$ respectively.
If $\varrho_i$ has even dimension for all $1\leq i\leq s$,
then
$$
\CS_\varphi=\{((e_{i,j}),(\Fe_{i,j}))\in\BZ_2^{r}\times \BZ_2^s\colon
 e_{i,j}, \Fe_{i,j}\in\{0,1\}\}.
$$
Otherwise, $\CS_\varphi$ is the subgroup of $\BZ_2^{r}\times \BZ_2^s$ consisting of elements with the condition that  $\sum_{i,j}\Fe_{i,j}\dim\varrho_i$ is even.
Define the character $\chi$ in $\widehat{\CS}_\varphi$ by
$$
\chi((e_{i,j}),(\Fe_{i,j}))=\prod^r_{i=1} \prod^{r_i}_{j=1}(-1)^{(j+r_i)e_{i,j}}\cdot
\prod^s_{i=1} \prod^{s_i}_{j=1}(-1)^{(j+s_i)\Fe_{i,j}}.
$$

\begin{rmk}
In some cases, the associated representation $\pi_a(\varphi,\chi)$ is supercuspidal.
For instance, take
$
\varphi=\boxplus^{r}_{i=1}\boxplus^{2r_i}_{j=1}\rho_i\boxtimes\mu_{2j}.
$
By definition, $\chi(1_{\rho_i\boxtimes\mu_{2}})=-1$ for all $1\leq i\leq r$ and $\chi((1))=1$ where $(1)$ is the element with $e_{i,j}=1$ for all $i,j$.
Thus, $\pi_a(\varphi,\chi)$ is a representation of  a symplectic group or a split even special orthogonal group, because $\det(\varphi)=1$.
By \cite{Mou15}, $\pi_a(\varphi,\chi)$ is supercuspidal.
\end{rmk}

Under the above assumption, by \eqref{eq:phi-sgn}, we have that
$$
\varphi_\sgn=\boxplus^{r}_{i=1}\boxplus^{2[\frac{r_{i}}{2}]}_{j=1}\rho_i\boxtimes\mu_{2(r_{i}-j)+1}\boxplus
\boxplus^{s}_{i=1}\boxplus^{s_i-1}_{j=1}\varrho_i\boxtimes\mu_{2j}.
$$
By convention, the summand $\varrho_i\boxtimes\mu_{2j}$ is empty if $s_i=1$.
As $\sum^s_{j=1}s_jb_j$ is even and
$$
\dim\varphi-\dim\varphi_\sgn=\sum^r_{i=1}2a_i[\frac{r_i+1}{2}]+\sum^s_{j=1}s_jb_j,
$$
we deduce that $\dim\varphi_\sgn$ is even.
By the definition in \eqref{eq:varphi-o-dagger}, one has that
$$
\lceil \varphi_o\rceil=\boxplus^s_{i=1}\varrho_i\boxtimes\mu_{2s_i-1}
\text{ and }
\chi_{\lceil \varphi_o\rceil}=1.
$$

According to Theorem \ref{pp:dsqlp}, if $\phi\in\FD_{\ell_0}(\varphi,\chi)$, we may decompose $\phi$ as $\phi=\psi\boxplus \varphi_\sgn$,
where $\psi$ satisfies the conditions in Theorem \ref{pp:dsqlp}.
Then $\psi=0$ is the unique choice.
In fact, $\dim\psi=0$ is even and of the minimal dimension.
By convention, $\chi^\star_{\varphi^\dagger}(\varphi^\dagger,\psi)=1$.
Thus,
\begin{equation}
\FD_{\ell_0}(\varphi,\chi)=
[\varphi_\sgn]_c=\{\varphi_\sgn,\varphi^c_\sgn\}
\text{ and }
\ell_0=\sum^r_{i=1}a_i[\frac{r_i+1}{2}]+\sum^s_{j=1}\frac{s_jb_j}{2}.	
\end{equation}
Similarly, assume that the elements in $\CS_\phi$ are of form $((e'_{i,j}),(\Fe'_{i,j}))$,
where $e'_{i,j}$ and $\Fe'_{i,j}$ correspond to $\rho_{i}\boxtimes \mu_{2(r_i-j)+1}$-component and $\varrho_{i}\boxtimes \mu_{2j}$-component respectively.
Define the character $\zeta\in \widehat{\CS}_\phi$ with the following conditions: for all $i$
\begin{itemize}
	\item $\zeta(1_{\rho_{i}\boxtimes \mu_{2(r_i-j)+1}})=1$ when $j=2[\frac{r_{i}}{2}]$ and
	$\zeta(1_{\varrho_{i}\boxtimes \mu_{2}})=-1$;
	\item $\zeta(1_{\rho_{i}\boxtimes \mu_{2(r_i-j)-1}\boxplus \rho_{i}\boxtimes \mu_{2(r_i-j)+1}})=-1$ for $1\leq j<2[\frac{r_{i}}{2}]$;
	\item $\zeta(1_{\varrho_{i}\boxtimes \mu_{2j}\boxplus \varrho_{i}\boxtimes \mu_{2(j+1)}})=-1$  for $1\leq j<s_i-1$.
\end{itemize}
By Lemma \ref{lm:key-lem}, one has that
$
(\chi,\zeta)=\chi^\star(\varphi,\phi).
$
By choosing the quadratic spaces and the Langlands correspondence as in Theorem \ref{thm:W12b}, we have that
$
m(\pi_a(\varphi,\chi),\pi_a(\phi,\zeta))=1.
$

For example,
when $G_n=\SO(V_{2n+1})$ is an odd special orthogonal group, take $\CO_{\ell_0}$ with $\disc(\CO_{\ell_0})=\disc(V_{2n+1})$ as $\det(\varphi_\sgn)=1$, which is the unique $F$-rational orbit such that $\CD_{\CO_{\ell_0}}(\pi(\varphi,\chi))\ne 0$ and
$$
\CD_{\CO_{\ell_0}}(\pi(\varphi,\chi))=\begin{cases}
	\pi_a(\varphi_\sgn,\zeta) &\text{ if some $\dim\rho_i$ is odd,}\\
	\pi_a(\varphi_\sgn,\zeta)\oplus \pi_a(\varphi^c_\sgn,\zeta) &\text{ otherwise,}
\end{cases}
$$
where $a=-\disc(V_{2n+1})$.
Note that $\CD_{\CO_{\ell_0}}(\pi(\varphi,\chi))$ is a representation of a pure inner form of split even orthogonal group as $\det(\varphi_\sgn)=1$.

\subsection{ Discrete unipotent representations}\label{sec:dur}
Following Arthur (\cite{A89}) and M\oe glin (\cite{M96D}), 
a representation $\pi(\phi,\chi)$ is a unipotent representation if  and only if $\phi$ is trivial on the inertia subgroup $\CI=\CI_F$ of the
local Weil group $\CW_F$. 
In this case, such a $\phi$ is called a unipotent local $L$-parameter.
We apply the local descent method to give an explicit description on the descent of discrete unipotent representations.

Denote $\xi_{un}=(\cdot,\eps)_F$ to be the nontrivial unramified quadratic character of $F^\times$, which is also regarded as a character $\CW_F$ via the local class field theory.
Here $\eps$ is a non-square element in $F$ with absolute value 1.
Let $\varphi$ be a discrete unipotent  $L$-parameter. Then it can be written as
\begin{equation}
\varphi=\begin{cases}
	\boxplus^r_{i=1}1\boxtimes \mu_{2a_i}\boxplus\boxplus^s_{j=1} \xi_{un}\boxtimes \mu_{2b_j}, &\text{ if } G_n=\SO_{2n+1}\\
	\boxplus^r_{i=1}1\boxtimes \mu_{2a_i+1}\boxplus\boxplus^s_{j=1} \xi_{un}\boxtimes \mu_{2b_j+1}, &\text{ if } G_n=\SO_{2n}.
\end{cases}
\end{equation}
Recall the calculation on the local root number from Example \ref{ex:eps-tensor-char}. We obtain the following.

\begin{cor}\label{cor:unipotent-odd}
Let $\pi(\varphi,\eta)$ be a discrete unipotent representation of $G_n=\SO(V_{2n+1})$.
If $\disc(\CO_{\ell_0})\neq\det(\varphi_\sgn)\disc(V_{2n+1})$, then $\CD_{\CO_{\ell_0}}(\pi)=0$; and if 
$\disc(\CO_{\ell_0})=\det(\varphi_\sgn)\disc(V_{2n+1})$, then 
$\CD_{\CO_{\ell_0}}(\pi)=\pi_a(\varphi_\sgn,\zeta)$ is irreducible. 
Here $\zeta=\chi^\star_{\varphi_\sgn}(\varphi,\varphi_\sgn)$ and $a=-\det(\varphi_\sgn)\disc(V_{2n+1})$.
\end{cor}
Note that when $G_n=\SO(V_{2n+1})$, the descent $\varphi_\sgn$ is invariant under $c$-conjugate.
And the local descent $\CD_{\CO_{\ell_0}}(\pi)$ is an irreducible unipotent discrete representation.


Now, let us consider the case $G_n=\SO(V_{2n})$.
Suppose that an irreducible unipotent discrete representation $\pi$ in the $L$-packet ${\Pi}_\varphi(G_n)$.
One can choose a local Langlands correspondence $\iota_a$ such that $\pi=\pi_a(\varphi,\eta)$ with
the condition that  $\eta(1_{1\boxtimes \mu_{2a_r+1}\boxplus \xi_{un}\boxtimes \mu_{2b_s+1}})=1$.
To prove this, let us  consider the action of $\CZ$ on $\widehat{\CS}_\varphi$.
In this case, as $\det\varphi=\eps^s$,
the character $\eta_\varpi$ (defined in Section \ref{ssec-rllc}) is equal to
$$
\eta_\varpi(((e_i),(\Fe_j)))= \prod^s_{j=1}(-1)^{\Fe_j},
$$
where $\varpi$ is a uniformizer in $F$.
Then $\chi(1_{1\boxtimes \mu_{2a_r+1}\boxplus \xi_{un}\boxtimes \mu_{2b_s+1}})=1$ or $\chi\otimes\eta_\varpi(1_{1\boxtimes \mu_{2a_r+1}\boxplus \xi_{un}\boxtimes \mu_{2b_s+1}})=1$ for any $\chi\in \widehat{\CS}_\varphi$.
Hence, by the definition of $\CO_\CZ(\pi)$ in \eqref{eq:CZ}, there exists a character $\chi$ in $\CO_\CZ(\pi)$ such that $\chi(1_{1\boxtimes \mu_{2a_r+1}\boxplus \xi_{un}\boxtimes \mu_{2b_s+1}})=1$.
Then we  can choose $\iota_a$ such that $\chi_a(\pi)=\chi$, which is the desired normalization for the local Langlands correspondence.


\begin{cor}\label{cor:unipotent-even}
Let $\pi$ be a discrete unipotent representation of $G_n=\SO(V_{2n})$ in $\Pi_\varphi[G_n^*]$.
Choose the local Langlands correspondence $\iota_a$ such that $\pi=\pi_a(\pi,\eta)$ with $\eta(1_{1\boxtimes \mu_{2a_r+1}\boxplus \xi_{un}\boxtimes \mu_{2b_s+1}})=1$. 
If $\disc(\CO_{\ell_0})\neq a$, then $\CD_{\CO_{\ell_0}}(\pi)=0$; and if $\disc(\CO_{\ell_0})=a$, 
then $\CD_{\CO_{\ell_0}}(\pi)=\pi(\varphi_\sgn,\zeta)$ is an irreducible representation of $\SO(V_{2m+1})$ with $m=\dim\varphi_{\sgn}/2$. 
Here $\zeta=\chi^\star_{\varphi_\sgn}(\varphi,\varphi_\sgn)$, $\disc(V_{2m+1})=-a \eps^s$ and
$$
\Hss(V_{2m+1})= (-1,-1)^{\frac{m(m+1)}{2}}((-1)^{m+1},\disc(V_{2m+1}))\eta((1)).
$$
\end{cor}

\subsection{On Conjecture \ref{conj-p1}}\label{sec:cp1}
We will show that Conjecture \ref{conj-p1} holds for certain representations of $\SO^*_7$.

Referring to \cite{CM93}, for $\SO^*_7$, all stable unipotent orbits are parameterized by the following partitions $\underline{p}$, respectively
\begin{equation}\label{eq:order}
[7],\quad [5,1^2],\quad[3^2,1],\quad
[3,2^2],\quad [3,1^4],\quad [2^2,1^3],\quad
[1^7],	
\end{equation}
where the powers indicate the multiplicities in the partitions,
and the corresponding unipotent orbits are listed following the topological order.
In particular, $[7]$ is for the regular unipotent orbit and $[1^7]$ is for the trivial orbit.
In this case, this topological order is a total order.
Hence, for an irreducible smooth representation $\pi$,
the set $\Fp^m(\pi)$ is a singleton.
We may assume that
\begin{equation}\label{eq:Fp-pi}
\Fp^m(\pi)=\{\underline{p}=[p_1p_2\cdots p_r]\} \text{ where } p_1\geq p_2\geq \cdots\geq p_r>0.	
\end{equation}

Let $\pi$ be an irreducible square-integrable representation of $\SO(V_7,F)$ and $\varphi$ be its $L$-parameter.
We are going to apply our main results to the following two types of $\varphi$:
\begin{enumerate}
\item $\varphi=\chi_1\boxtimes \mu_{4}\boxplus \chi_2\boxtimes \mu_{2}$ or
\item $\varphi=\chi_1\boxtimes \mu_{2}\boxplus \chi_2\boxtimes \mu_{2}\boxplus \chi_3\boxtimes \mu_{2}$,
\end{enumerate}
where $\chi_i$ are quadratic characters,
and to verify that Conjecture \ref{conj-p1} holds for the representations in the corresponding Vogan packets.
For simplicity, we assume that $-1$ is a square in $F^\times$.
Then $(\det\varphi,-1)_F=1$ in \eqref{eq:CE} for all $L$-parameters $\varphi$.
We take $V_7$ satisfying $\disc(V_7)=-1=1 \bmod{F^{\times 2}}$.
For the odd  special orthogonal group, the local Langland correspondence is unique and denote $\chi_\pi$ to be the corresponding character in $\wh{\CS}_\varphi$.

To verify Conjecture \ref{conj-p1}, we will construct an $L$-parameter $\phi$ in $\FD_{\ell}(\varphi,\chi_\pi)$ for some $\ell\in\{1,2\}$, which implies
 the descent $\FD_{\ell}(\varphi,\chi_\pi)$ is not empty.
By Theorem \ref{th:sdlp}, for such $\phi$ the $F$-rational orbit $\CO_\ell$ and the quadratic space $W$ defining $G^{\CO_{\ell}}_n$ are determined by
$\disc(\CO_\ell)=\disc(W)=\det\phi\bmod{F^{\times 2}}$.
Following \eqref{eq:thm:HW} and \eqref{eq:thm:HV}, after simplification, $\Hss(V_7)=\Hss(W)=\chi_\pi((1))$.
The local Langlands correspondence $\iota_a$ for the even special orthogonal group $\SO(W)$ is normalized by $a=\det\phi\bmod{F^{\times 2}}$.
Hence we obtain that the following irreducible square-integrable representation of $\SO(W,F)$
$$
\sig=\pi_{\det\phi}(\phi,\chi^\star(\varphi,\phi))
$$
occurs in $\CD_{\CO_\ell}(\pi)$  for the above chosen $F$-rational orbit,
where
$$
\disc(W)=\det\phi \text{ and }	\Hss(W)=\Hss(V_7). 	
$$
It follows that $p_1$ in \eqref{eq:Fp-pi} is equal to or greater than $2\ell+1$.
Because $\ell\in\{1,2\}$, we have a lower bound $p_1\geq 3$ by the total order in \eqref{eq:order}.

Now, if $p_1=7$, $\pi$ is generic and then $\ell_0=3$.
Conjecture  \ref{conj-p1} holds.
If $p_1=5$, then the nonvanishing the twisted Jacquet module associated $\underline{p}=[51^2]$ is equivalent to the nonvanishing of the local descent $\CD_{\CO_2}(\pi)$ by Lemma \ref{lm:giq}.
By definition, the first occurrence index $\ell_0$ equals $2$.
If $p_1=3$, by the above lower bound $p_1\geq 3$ we have $p_1=3$.
Then $\ell_0=1$ in this case, which implies the conjecture.

Furthermore, for these two types of parameters, our results may explicitly determine $\Fp^m(\pi)$ in terms of $(\varphi,\chi_\pi)$ associated to $\pi$.
The detailed calculation will be given in the rest.

{\bf Type (1):} Assume that $\varphi=\chi_1\boxtimes \mu_{4}\boxplus \chi_2\boxtimes \mu_{2}$.
Then $\CS_\varphi\cong \BZ_2\times\BZ_2$.
Denote $\zeta_{+}$ (resp. $\zeta_{-}$) to be the trivial (resp. non-trivial) character of $\BZ_2$.
Then we may write the characters of $\CS_\varphi$ as $\zeta_{\pm}\otimes\zeta_\pm$.
Its Vogan packet $\Pi_{\varphi}[\SO_7^*]$ contains four representations $\pi(\varphi,\zeta_{\pm}\otimes\zeta_\pm)$.

If $\pi=\pi(\varphi,\zeta_{+}\otimes\zeta_+)$, then $\pi$ is the unique generic representation in $\Pi_{\varphi}[\SO_7^*]$.
We have $\ell_0=3$ and $\udl{p}=[7]$.

For $\pi=\pi(\varphi,\zeta_{-}\otimes\zeta_-)$, choose
$$
\phi=\begin{cases}
	\chi_1\boxplus\chi_2, &\text{ if }\chi_1\ne\chi_2\\
	\chi_1\boxplus\chi', &\text{ if }\chi_1=\chi_2,
\end{cases}
$$
where $\chi'$ is a quadratic character not isomorphic to $\chi_1$ and $\chi_2$.
Since $\pi$ is non-generic, we have $\ell_0(\pi)\leq 2$.
By Lemma \ref{lm:key-lem},  $\phi\in \FD_{2}(\varphi,\zeta_{-}\otimes\zeta_-)$
and then $\ell_0(\pi)\geq 2$.
It follows  that $\ell_0(\pi)=2$ and $\Fp^{m}(\pi)=[5,1^2]$.

When $\pi=\pi(\varphi,\zeta_{+}\otimes\zeta_-)$ or $\pi(\varphi,\zeta_{-}\otimes\zeta_+)$, $\Hss(V_7)=-1$ (i.e., $\SO_7$ is non-split) and $\ell_0\leq 2$.
Similarly, we have $\chi_2\boxplus\chi'\in\FD_{2}(\varphi,\zeta_{+}\otimes\zeta_-)$
and $\phi=\chi_1\boxplus\chi'\in\FD_{2}(\varphi,\zeta_{-}\otimes\zeta_+)$.
In both cases, $\ell_0\geq 2$ and then $\ell_0(\pi)=2$ and $\Fp^{m}(\pi)=[5,1^2]$.

{\bf Type (2):} Assume that $\varphi=\chi_1\boxtimes \mu_{2}\boxplus \chi_2\boxtimes \mu_{2}\boxplus \chi_3\boxtimes \mu_{2}$.
Since $\pi$ is square-integrable, $\chi_1$, $\chi_2$ and $\chi_3$ are distinct.
Then $\CS_\varphi\cong\BZ_2\times\BZ_2\times\BZ_2$
and the character of $\CS_\varphi$ is of form $\zeta_1\otimes\zeta_2\otimes\zeta_3$,
where $\zeta_i$ for $1\leq i\leq 3$ are characters of $\BZ_2$.

Let $\pi=\pi(\varphi,\zeta_1\otimes\zeta_2\otimes\zeta_3)$.
If $\zeta_i=\zeta_+$ for all $1\leq i\leq 3$, then $\pi$ is generic and $\ell_0(\pi)=3$.
All the other representations are non-generic. For the rest cases, we always have $\ell_0(\pi)\leq 2$, i.e., $\udl{p}\ne [7]$.

As $F^\times/(F^\times)^2$ contains at least 4 elements,
there exists a character $\chi'$ of $F^\times$ such that $\chi'^2=1$ and $\chi'$ is not isomorphic to any $\chi_i$ for $1\leq i\leq 3$.
When $\zeta_{i_0}=\zeta_-$ for some $i_0$ and $\zeta_i=\zeta_+$ for all $i\ne i_0$, we may take $\phi=\chi_{i_0}\boxplus\chi'$.
In this case, $\SO_7$ is non-split.
It follows that $\phi\in\FD_{2}(\varphi,\otimes^3_{i=1}\zeta_i)$
and then $\ell_0(\pi)=2$ and $\Fp^{m}(\pi)=[5,1^2]$.

If $\zeta_{i_0}=\zeta_+$ for some $i_0$ and $\zeta_i=\zeta_-$ for all $i\ne i_0$, we may take $\phi=\chi_i\boxplus \chi_j$ where $\{i,j\}=\{1,2,3\}\smallsetminus\{i_0\}$.
One has that $\phi\in\FD_{2}(\varphi,\otimes^3_{i=1}\zeta_i)$ and hence $\ell_0(\pi)=2$ and $\Fp^{m}(\pi)=[5,1^2]$.

If $\zeta_i=\zet_-$ for all $1\leq i\leq 3$, we may take $\phi=\boxplus^3_{i=1}\chi_i\boxplus \chi'$, which is in $\FD_1(\varphi,\otimes^3_{i=1}\zeta_i)$.
We have the lower bound $\ell_0(\pi)\geq 1$.
By the above discussion, this implies Conjecture \ref{conj-p1} for this representation.
In this case, one may also explicitly determine $\Fp^m(\pi)$ by calculating the symplectic root number $\varepsilon(\tau\times\chi_i)$ where $\tau$ is a supercuspidal representation of $\GL_2(F)$ with the central character $\omega_\tau=1$ (i.e., of symplectic type).
We omit the details here.

\begin{rmk}
Beside the above {\bf Type (1)} and {\bf Type (2)}, for any irreducible smooth representation $\pi$ of $\SO(V_7,F)$ in a generic $L$-packet,
we may obtain the lower bound $\ell_0(\pi)\geq 1$ by using an alternative global argument.
Then Conjecture \ref{conj-p1} holds for all pure inner forms of $\SO^*_7$.
However, such global arguments only work for the special orthogonal groups of lower rank.
\end{rmk}

\begin{rmk}[Counter Example]
We give an example to show that Conjecture \ref{conj-p1} may not be true for non-tempered representations.
Let $G^*_n=\SO^*_4$ be the split even orthogonal group.
Note that $\SO^*_4$ only have 4 stable unipotent orbits, whose corresponding partitions are
$$
[3,1], [2^2]^I, [2^2]^{II}, [1^4].
$$
Here $[2^2]^I$ and $[2^2]^{II}$ are the same partitions of 4 but give two different unipotent orbits.
We take $\pi$ to be the irreducible non-generic non-tempered infinite dimensional representation of $\SO^*_4$.
Then it is not generic and has a nonzero twisted Jacquet model associated to $[2^2]^I$ or $[2^2]^{II}$.
In this case, the largest part $p_1$ is even and not equal to $2\ell+1$ for all $\ell$.
In general, one can find a family of non-tempered  representations $\pi$ of $\SO_{2n}^*$, whose largest parts $p_1$ in the partitions of $\Fp^m(\pi)$ are even.
Hence Conjecture 1.8 fails for those non-tempered representations.
\end{rmk}

\end{document}